\documentclass[11pt,reqno]{amsart}

\usepackage{amsfonts}
\usepackage{eurosym}
\usepackage{amssymb}
\usepackage{amsthm}
\usepackage{amsmath}
\usepackage{amsaddr}
\usepackage{bm}
\usepackage{cite}
\usepackage{mathrsfs}
\usepackage{xcolor}
\usepackage[OT1]{fontenc}
\usepackage[left=2.5cm, right=2.5cm, top=3cm, bottom=2.5cm]{geometry}
\usepackage{hyperref}
\usepackage{mathtools}
\hypersetup{colorlinks=true, linkcolor=blue, citecolor=red, urlcolor=blue}
\usepackage{times}

\usepackage{enumitem}
\usepackage{xpatch}
\makeatletter
\AtBeginDocument{\xpatchcmd{\@thm}{\thm@headpunct{.}}{\thm@headpunct{}}{}{}}
\makeatother

\allowdisplaybreaks
\newtheorem{theorem}{Theorem}[section]

\newtheorem{corollary}[theorem]{Corollary}
\newtheorem{definition}[theorem]{Definition}
\newtheorem{lemma}[theorem]{Lemma}
\newtheorem{proposition}[theorem]{Proposition}

\theoremstyle{theremark}
\newtheorem{remark}[theorem]{Remark}

\numberwithin{equation}{section}

\usepackage{parskip}
\DeclareMathAlphabet\mathbfcal{OMS}{cmsy}{b}{n}

\newcommand{\abs}[1]{\left| #1 \right|}

\newcommand{\norm}[1]{\| #1 \|}

\newcommand{\Bignorm}[1]{\Big\| #1 \Big\|}
\newcommand{\ang}[2]{ \langle #1 , #2  \rangle}
\newcommand{\bigang}[2]{ \big< #1 , #2  \big>}

\newcommand{\scp}[2]{ \left( #1 , #2  \right)}
\newcommand{\bigscp}[2]{\big( #1 , #2 \big)}

\newcommand{\meano}[1]{{\langle #1 \rangle}_{\Omega}}
\newcommand{\meang}[1]{{\langle #1 \rangle}_{\Gamma}}
\newcommand{\mean}[2]{\textnormal{mean}\scp{#1}{#2}}

\newcommand{\R}{\mathbb R}
\newcommand{\N}{\mathbb N}
\newcommand{\n}{\mathbf{n}}

\newcommand{\intO}{\int_\Omega}

\newcommand{\intG}{\int_\Gamma}

\newcommand{\dtau}{\;\mathrm d\tau}

\newcommand{\dx}{\;\mathrm{d}x}

\newcommand{\dt}{\;\mathrm dt}
\newcommand{\ds}{\;\mathrm ds}

\newcommand{\dxs}{\;\mathrm{d}x\;\mathrm{d}s}
\newcommand{\dGs}{\;\mathrm{d}\Ga\;\mathrm{d}s}

\newcommand{\dG}{\;\mathrm d\Ga}

\newcommand{\ddt}{\frac{\mathrm d}{\mathrm dt}}

\newcommand{\delt}{\partial_{t}}
\newcommand{\delth}{\partial_{t}^{h}}

\newcommand{\deln}{\partial_\n}

\newcommand{\Grad}{\nabla}
\newcommand{\Lap}{\Delta}
\newcommand{\Div}{\textnormal{div}}
\newcommand{\Gradg}{\nabla_\Ga}
\newcommand{\Lapg}{\Delta_\Ga}
\newcommand{\Divg}{\textnormal{div}_\Ga}

\newcommand{\emb}{\hookrightarrow}

\newcommand{\suchthat}{\;\ifnum\currentgrouptype=16 \middle\fi|\;}

\newcommand{\e}{\mathbf{e}}

\newcommand{\Om}{\Omega}
\newcommand{\Ga}{\Gamma}

\definecolor{rosso}{rgb}{0.8,0,0}
\definecolor{violet}{rgb}{0.65,0,0.65}
\definecolor{darkgreen}{rgb}{0,0.5,0}

\def \no#1#2#3 {{\bf #1} (#3), #2.}
\def \eds#1#2#3 {#1, #2, #3.}

\makeatletter
\def\@settitle{\begin{center}%
  \baselineskip14\p@\relax
    \huge
  \@title
  \end{center}%
}
\makeatother

\begin{document}

\title[Bulk-surface Cahn--Hilliard model with non-degenerate mobility]
{\emph{Well-posedness and long-time behavior of a bulk-surface Cahn--Hilliard model with non-degenerate mobility}}
\author[Jonas Stange]{Jonas Stange}


\address{Fakult\"{a}t f\"{u}r Mathematik\\ 
Universit\"{a}t Regensburg \\
93040 Regensburg, Germany\\
\href{mailto:jonas.stange@ur.de}{jonas.stange@ur.de}
}

\subjclass[2020]{35K35, 35D30, 35A01, 35A02, 35Q92, 35B65}

%

%
\keywords{Cahn--Hilliard equation, bulk-surface interaction, dynamic boundary conditions, non-degenerate mobility, uniqueness, regularity, convergence to steady states.}


\begin{abstract}
We study a bulk-surface Cahn--Hilliard model with non-degenerate mobility and singular potentials in two dimensions. Following the ideas of the recent work by Conti, Galimberti, Gatti, and Giorgini [Calc. Var. Partial Differential Equations, 64(3):Paper No. 87, 32, 2025] for the Cahn--Hilliard equation with homogeneous Neumann boundary conditions, we show the uniqueness of weak solutions together with a continuous dependence estimate for sufficiently regular mobility functions. Next, under weaker assumptions on the mobility functions, we show the existence of a weak solution that exhibits the propagation of uniform-in-time regularity and satisfies the instantaneous separation property. Lastly, we consider the long-time behavior and prove that the unique weak solution converges to a solution of the stationary bulk-surface Cahn--Hilliard equation. Our approach for the uniqueness proof relies on a new well-posedness and regularity theory for a bulk-surface elliptic system with non-constant coefficients, which may be of independent interest.
\end{abstract}

\maketitle

\section{Introduction}
\label{SECT:INTRO}
\noindent
In this paper, we investigate the following bulk-surface Cahn--Hilliard model
\begin{center}
\begin{subequations}\label{EQ:SYSTEM}
    \begin{align}
        \label{EQ:SYSTEM:1}
        &\delt\phi = \Div(m_\Om(\phi)\Grad\mu) && \text{in} \ \Om\times(0,\infty), \\
        \label{EQ:SYSTEM:2}
        &\mu = -\Lap\phi + F'(\phi)   && \text{in} \ \Om\times(0,\infty), \\
        \label{EQ:SYSTEM:3}
        &\delt\psi = \Divg(m_\Ga(\psi)\Gradg\theta) - \beta m_\Om(\phi)\deln\mu && \text{on} \ \Ga\times(0,\infty), \\
        \label{EQ:SYSTEM:4}
        &\theta = - \Lapg\psi + G'(\psi) + \alpha\deln\phi && \text{on} \ \Ga\times(0,\infty), \\
        \label{EQ:SYSTEM:5}
        &\begin{cases} K\deln\phi = \alpha\psi - \phi &\text{if} \ K\in [0,\infty), \\
        \deln\phi = 0 &\text{if} \ K = \infty
        \end{cases} && \text{on} \ \Ga\times(0,\infty), \\
        \label{EQ:SYSTEM:6}
        &\begin{cases} 
        L m_\Om(\phi)\deln\mu = \beta\theta - \mu &\text{if} \  L\in[0,\infty), \\
        m_\Om(\phi)\deln\mu = 0 &\text{if} \ L=\infty
        \end{cases} &&\text{on} \ \Ga\times(0,\infty), \\
        \label{EQ:SYSTEM:7}
        &\phi\vert_{t=0} = \phi_0 &&\text{in} \ \Om, \\
        \label{EQ:SYSTEM:8}
        &\psi\vert_{t=0} = \psi_0 &&\text{on} \ \Ga.
    \end{align}
\end{subequations}
\end{center}
System \eqref{EQ:SYSTEM} is a special case of the model proposed in \cite{Knopf2024}. There, the authors considered an extended version of \eqref{EQ:SYSTEM} with convection.

In \eqref{EQ:SYSTEM}, $\Om\subset\R^2$ is a bounded domain with boundary $\Ga\coloneqq\partial\Om$. We use the abbreviation $Q = \Om\times(0,\infty)$ and $\Sigma = \Ga\times(0,\infty)$. The outward pointing unit normal vector on $\Ga$ is denoted by $\n$, while $\deln$ denotes the outward normal derivative on the boundary. Moreover, the symbols $\Gradg$ and $\Lapg$ stand for the surface gradient and the Laplace--Beltrami operator on $\Ga$, respectively.

The functions $\phi:Q\rightarrow\R$ and $\mu:Q\rightarrow\R$ denote the phase-field and the chemical potential in the bulk, whereas $\psi:\Sigma\rightarrow\R$ and $\theta:\Sigma\rightarrow\R$ represent the phase-field and the chemical potential on the boundary, respectively. The functions $m_\Om,m_\Ga:[-1,1]\rightarrow\R$ are the so-called Onsager mobilities, and typically depend on the phase-field variables $\phi$ and $\psi$, respectively. They model the spatial locations and intensity at which the diffusion processes take place.

In \eqref{EQ:SYSTEM}, the time evolution of the bulk-variables $\phi$ and $\mu$ is governed by the \textit{bulk Cahn--Hilliard subsystem} \eqref{EQ:SYSTEM:1}-\eqref{EQ:SYSTEM:2}, while the evolution of the surface quantities $\psi$ and $\theta$ is given by the \textit{surface Cahn--Hilliard subsystem} \eqref{EQ:SYSTEM:3}-\eqref{EQ:SYSTEM:4}, which is coupled to the bulk by expressions involving the normal derivatives $\deln\phi$ and $\deln\mu$. Moreover, the phase-field $\phi$ and $\psi$ are coupled by the boundary condition \eqref{EQ:SYSTEM:5}, while the chemical potentials $\mu$ and $\theta$ are coupled by the boundary condition \eqref{EQ:SYSTEM:6}. Here, the parameters $K,L\in[0,\infty]$ are used to distinguish different types of these coupling conditions and $\alpha,\beta\in\R$ describe different physical phenomena, see, for instance, \cite{Giorgini2023, Knopf2024} for a more extensive description. These types of boundary conditions fall into the class of \textit{dynamical boundary conditions}, which generalize the classical homogeneous Neumann boundary conditions
\begin{align*}
    \deln\phi = \deln\mu = 0\qquad\text{on~}\Ga\times(0,\infty),
\end{align*}
typically imposed in standard Cahn--Hilliard models. Although Neumann boundary conditions are widely used, they can be overly restrictive when a precise description of boundary dynamics is required, see, e.g., \cite{Giorgini2023, Knopf2024}. This has led to the development and analysis of several dynamic boundary condition formulations in the literature, we refer to the recent survey \cite{Wu2022} and the references therein for a comprehensive overview. In particular, dynamic boundary conditions of Cahn–Hilliard type that incorporate mass exchange between bulk and boundary have received considerable attention in recent years, see, e.g., \cite{Goldstein2011, Liu2019, Knopf2020, Knopf2021a}. This context motivates the study of the coupled bulk-surface system \eqref{EQ:SYSTEM}.

The functions $F^\prime$ and $G^\prime$ are the derivatives of double-well potentials $F$ and $G$, respectively. A physically motivated example of such a double-well potential, especially in applications related to material science, is the \textit{Flory-Huggins potential}, which is also referred to as the \textit{logarithmic potential.} It is given as
\begin{align*}
    W_{\mathrm{log}}(s) \coloneqq \frac{\Theta}{2}\Big[(1+s)\ln(1+s) + (1-s)\ln(1-s)\Big] - \frac{\Theta_0}{2}s^2, \qquad s\in[-1,1],
\end{align*}
with the convention that $0\ln 0$ is interpreted as zero. The positive parameters $\Theta$ and $\Theta_0$ denote the temperature of the mixture and the critical temperature below which phase separation processes occur, respectively, and are supposed to satisfy $\Theta_0 - \Theta > 0$. Since $W_{\mathrm{log}}^\prime(s) \rightarrow \pm\infty$ as $s\rightarrow\pm 1$, the potential $W_{\mathrm{log}}$ is a so-called singular potential. In this contribution, we consider a more general class of singular potentials (see Section~\ref{SUBSEC:ASS}) such that, for example, the choice $F = G = W_{\mathrm{log}}$ is admissible.

The free energy functional associated with system \eqref{EQ:SYSTEM} reads as
\begin{align}
	\label{INTRO:ENERGY}
	\begin{split}
		E(\phi,\psi) &= \intO \frac12\abs{\Grad\phi}^2 + F(\phi) \dx + \intG \frac12\abs{\Gradg\psi}^2 + G(\psi) \dG \\
		&\quad + \chi(K)\intG\frac{1}{2}\abs{\alpha\psi - \phi}^2\dG.
	\end{split}
\end{align}
Here, to account for the different cases corresponding to the choice of $K$, the function
\begin{align*}
	\chi:[0,\infty]\rightarrow[0,\infty), \quad \chi(r) \coloneqq
	\begin{cases}
		r^{-1}, &\text{if } r\in (0,\infty), \\
		0, &\text{if } r\in\{0,\infty\}
	\end{cases}
\end{align*}
is used.
We observe that sufficiently regular solutions of the system \eqref{EQ:SYSTEM} satisfy the \textit{mass conservation law}
\begin{align}
	\label{INTRO:MASS}
	\begin{dcases}
		\beta\intO \phi(t)\dx + \intG \psi(t)\dG = \beta\intO \phi_0 \dx + \intG \psi_0\dG, &\textnormal{if } L\in[0,\infty), \\
		\intO\phi(t)\dx = \intO\phi_0\dx \quad\textnormal{and}\quad \intG\psi(t)\dG = \intG\psi_0\dG, &\textnormal{if } L = \infty
	\end{dcases}
\end{align}
for all $t\in[0,\infty)$ and the \textit{energy identity}
\begin{align}
	\label{INTRO:ENERGY:ID}
	\begin{split}
		\ddt E(\phi,\psi)  &= - \intO m_\Om(\phi)\abs{\Grad\mu}^2\dx
		- \intG m_\Ga(\psi)\abs{\Gradg\theta}^2\dG
		- \chi(L) \intG (\beta\theta-\mu)^2\dG 
	\end{split}
\end{align}
on $[0,\infty)$. We note that the right-hand side of \eqref{INTRO:ENERGY:ID} is non-positive, which means that the energy dissipates over time, and the terms appearing on the right-hand side of \eqref{INTRO:ENERGY:ID} can be interpreted as the dissipation rate.

\textbf{Goals and novelties of this paper.} 
System~\eqref{EQ:SYSTEM} has been extensively studied in the literature. We refer, for instance, to \cite{Goldstein2011,Liu2019,Knopf2021a,Knopf2020,Colli2020,Fukao2021,Miranville2020,Lv2024a,Lv2024b}. An extended model with convection was recently analyzed in \cite{Knopf2024} for regular potentials, and in \cite{Giorgini2025, Knopf2025} for singular potentials. However, existing analytical results on uniqueness and higher regularity have so far required the mobility functions $m_\Om$ and $m_\Ga$ to be constant. This represents a significant limitation, as in many physical applications, the diffusion intensity may vary spatially and is not expected to be uniform throughout the domain. For the classical Cahn--Hilliard equation with homogeneous Neumann boundary conditions, the recent work \cite{Conti2025} establishes the uniqueness and the propagation regularity of weak solutions for non-degenerate mobility functions in two dimensions, and thus improved the state of the art dating back to the work \cite{Barrett1999}. We further would like to mention the recent work \cite{Elliott2024}, where the authors established a weak-strong uniqueness principle for the Cahn--Hilliard equation on an evolving surface. Building on the approach developed in \cite{Conti2025}, we extend the uniqueness theory to the Cahn--Hilliard equation with dynamic boundary conditions \eqref{EQ:SYSTEM}. To the best of our knowledge, this is the first result addressing uniqueness and propagation of regularity in the setting of dynamic boundary conditions with non-constant mobilities.
Our proof relies on two main ingredients: a novel well-posedness and regularity theory for a bulk-surface elliptic system with non-constant coefficients (see Section~\ref{Section:EBS}), as well as a recently established result on higher time-regularity for the phase-fields (see \cite[Theorem~3.3]{Giorgini2025}). Then, using our proof, particularly the key differential inequality established therein, we are able to demonstrate the propagation of regularity for weak solutions. To this end, we introduce an additional regularization procedure for the mobility functions, which allows us to work under minimal assumptions on their regularity. As a consequence, we also obtain the instantaneous separation property. This separation property, in turn, enables us to prove convergence to a single stationary state by means of the standard {\L}ojasiewicz--Simon approach. It is worth mentioning that our results can be readily adapted to the case of the convective model proposed in \cite{Knopf2024} for sufficiently regular prescribed velocity fields. This might be useful in further analysis of related models for two-phase flows with bulk-surface interaction, see, e.g., \cite{Giorgini2023, Gal2024} for a bulk-surface Navier--Stokes--Cahn--Hilliard model or \cite{Knopf2025a} for a bulk-surface Navier--Stokes--Cahn--Hilliard model in an evolving domain.

\textbf{Structure of this paper.} The rest of this contribution is structured as follows. In Section~\ref{Section:Prelim}, we collect some notation, assumptions, preliminaries, and important tools. After introducing the notation of a weak solution of \eqref{EQ:SYSTEM}, we state our main results in Section~\ref{Section:Main}. In Section~\ref{Section:EBS}, we establish a new well-posedness and regularity theory for a bulk-surface elliptic system with non-constant coefficients. Then, Section~\ref{Section:Uniqueness} is devoted to the proof of the well-posedness of weak solutions to \eqref{EQ:SYSTEM}. Afterwards, in Section~\ref{Section:PropRegSep}, we show that there exists a weak solution that admits the propagation of regularity and satisfies the instantaneous separation property. Lastly, in Section~\ref{Section:LongTime}, we study the long-time behavior of the unique weak solution of \eqref{EQ:SYSTEM} and show its convergence to a single stationary state as $t\rightarrow\infty$.

\medskip

\noindent
\section{Functional framework, assumptions and preliminaries}
\label{Section:Prelim}
\subsection{Notation and Function Spaces.}
For any Banach space $X$, we denote its norm by $\norm{\cdot}_X$, its dual space by $X^\prime$, and the associated duality pairing of elements $\phi\in X^\prime$ and $\zeta\in X$ by $\ang{\phi}{\zeta}_X$. The space $L^p(I;X)$, $1\leq p \leq +\infty$, denotes the set of all strongly measurable $p$-integrable functions mapping from any interval $I\subset\R$ into $X$, or, if $p = +\infty$, essentially bounded functions. Moreover, the space $W^{1,p}(I;X)$ consists of all functions $f\in L^p(I;X)$ such that $\delt f\in L^p(I;X)$, where $\delt f$ denotes the vector-valued distributional derivative of $f$. Furthermore, $L^p_{\mathrm{uloc}}(I;X)$ denotes the space of functions $f\in L^p(I;X)$ such that
\begin{align*}
    \norm{f}_{L^p_{\mathrm{uloc}}(I;X)} \coloneqq \sup_{t\geq 0}\Big(\int_{I\cap [t,t+1)}\norm{f(s)}_X^p\ds\Big)^{\frac1p} < \infty.
\end{align*}
If $I\subset\R$ is a finite interval, we find that $L^p_{\mathrm{uloc}}(I;X) = L^p(I;X)$. Further, we denote the space of continuous functions mapping from $I$ to $X$ by $C(I;X)$.

Let $\Om\subset\R^d$, $d\in\{2,3\}$, be a bounded domain with sufficiently smooth boundary $\Ga \coloneqq \partial\Om$. For any $1 \leq p \leq \infty$ and $k\in\N_0$, the Lebesgue and Sobolev spaces for functions mapping from $\Om$ to $\R$ are denoted as $L^p(\Om)$ and $W^{k,p}(\Om)$, respectively. Here, we use $\N$ for the set of natural numbers excluding zero and $\N_0 \coloneqq \N\cup\{0\}$. For $1 \leq p \leq \infty$ and $s\geq 0$, we denote by $W^{s,p}(\Om)$ the Sobolev-Slobodeckij spaces. If $p = 2$, we write $H^s(\Om) = W^{s,2}(\Om)$. In particular, $H^0(\Om)$ can be identified with $L^2(\Om)$.
The Lebesgue, Sobolev, and Sobolev-Slobodeckij spaces on the boundary can be defined similarly, provided that $\Ga$ is sufficiently regular. As before, we write $H^s(\Ga) = W^{s,2}(\Ga)$ and identify $H^0(\Ga)$ with $L^2(\Ga)$.

Next, we introduce the product spaces
\begin{align*}
    \mathcal{L}^p \coloneqq L^p(\Om)\times L^p(\Ga), \quad\text{and}\quad \mathcal{W}^{s,p} \coloneqq W^{s,p}(\Om)\times W^{s,p}(\Ga),
\end{align*}
for any real numbers $s\geq 0$ and $p\in[1,\infty]$, provided that the boundary $\Ga$ is sufficiently regular. We abbreviate $\mathcal{H}^s \coloneqq \mathcal{W}^{s,2}$ and identify $\mathcal{L}^2$ with $\mathcal{H}^0$. Note that $\mathcal{H}^s$ is a Hilbert space with respect to the inner product
\begin{align*}
    \bigscp{\scp{\phi}{\psi}}{\scp{\zeta}{\xi}}_{\mathcal{H}^s} \coloneqq \scp{\phi}{\zeta}_{H^s(\Om)} + \scp{\psi}{\xi}_{H^s(\Ga)} \qquad\text{for all~} \scp{\phi}{\psi},\scp{\zeta}{\xi}\in\mathcal{H}^s
\end{align*}
and its induced norm $\norm{\cdot}_{\mathcal{H}^s} \coloneqq \scp{\cdot}{\cdot}_{\mathcal{H}^s}^{\frac12}$. We recall that the duality pairing can be expressed as
\begin{align*}
    \ang{\scp{\phi}{\psi}}{\scp{\zeta}{\xi}}_{\mathcal{H}^s} \coloneqq \scp{\phi}{\zeta}_{L^2(\Om)} + \scp{\psi}{\xi}_{L^2(\Ga)}
\end{align*}
for all $(\zeta,\xi)\in\mathcal{H}^s$ if $(\phi,\psi)\in\mathcal{L}^2$.
For $L\in[0,\infty]$ and $\beta\in\R$, we introduce the linear subspace
\begin{align*}
    \mathcal{H}_{L}^1 \coloneqq
    \begin{cases}
        \mathcal{H}^1, &\text{if } L \in (0,\infty] , \\
     \displaystyle   \{(\phi,\psi)\in\mathcal{H}^1 : \phi = \beta\psi \text{ a.e.~on } \Ga\}, &\text{if } L=0.
    \end{cases}
\end{align*}
The space $\mathcal{H}_{L}^1$ is a Hilbert space endowed with the inner product $\scp{\cdot}{\cdot}_{\mathcal{H}_{L}^1} \coloneqq \scp{\cdot}{\cdot}_{\mathcal{H}^1}$ and its induced norm.  Moreover, we define the product
\begin{align*}
    \ang{\scp{\phi}{\psi}}{\scp{\zeta}{\xi}}_{\mathcal{H}_{L}^1} \coloneqq \scp{\phi}{\zeta}_{L^2(\Om)} + \scp{\psi}{\xi}_{L^2(\Ga)}
\end{align*}
for all $\scp{\phi}{\psi}, \scp{\zeta}{\xi}\in\mathcal{L}^2$. By means of the Riesz representation theorem, this product can be extended to a duality pairing on $(\mathcal{H}_{L}^1)^\prime\times\mathcal{H}_{L}^1$, which will also be denoted as $\ang{\cdot}{\cdot}_{\mathcal{H}_{L}^1}$.

For $(\phi,\psi)\in(\mathcal{H}^1_L)^\prime$, we define the generalized bulk-surface mean
\begin{align*}
    \mean{\phi}{\psi} \coloneqq \frac{\ang{\scp{\phi}{\psi}}{\scp{\beta}{1}}_{\mathcal{H}^1_L}}{\beta^2\abs{\Om} + \abs{\Ga}},
\end{align*}
which reduces to
\begin{align*}
    \mean{\phi}{\psi} = \frac{\beta\abs{\Om}\meano{\phi} + \abs{\Ga}\meang{\psi}}{\beta^2\abs{\Om} + \abs{\Ga}}
\end{align*}
if $(\phi,\psi)\in\mathcal{L}^2$, where
\begin{align*}
\meano{\phi}=\frac{1}{\abs{\Omega}} \int_{\Omega} \phi \, \dx,
\quad
\meang{\psi}=\frac{1}{\abs{\Gamma}} \int_{\Gamma} \psi \, \dG.
\end{align*}
We then define the closed linear subspace
\begin{align*}
    \mathcal{V}_{L}^1 &\coloneqq \begin{cases} 
    \{\scp{\phi}{\psi}\in\mathcal{H}^1_L : \mean{\phi}{\psi} = 0 \}, &\text{if~} L\in[0,\infty), \\
    \{\scp{\phi}{\psi}\in\mathcal{H}^1: \meano{\phi} = \meang{\psi} = 0 \}, &\text{if~}L=\infty.
    \end{cases} 
\end{align*}
Note that this space is a Hilbert space with respect to the inner product $\scp{\cdot}{\cdot}_{\mathcal{H}^1}$.

Now, we set
\begin{align*}
    \chi(L) \coloneqq
    \begin{cases}
        L^{-1}, &\text{if } L\in(0,\infty), \\
            0, &\text{if } L\in\{0,\infty\},
    \end{cases}
\end{align*}
and we introduce a bilinear form on $\mathcal{H}^1\times\mathcal{H}^1$ by defining
\begin{align*}
    \bigscp{\scp{\phi}{\psi}}{\scp{\zeta}{\xi}}_{L} \coloneqq &\intO\Grad\phi\cdot\Grad\zeta \dx + \intG\Gradg\psi\cdot\Gradg\xi \dG + \chi(L)\intG (\beta\psi-\phi)(\beta\xi-\zeta)\dG
\end{align*}
for all $ \scp{\phi}{\psi}, \scp{\zeta}{\xi}\in\mathcal{H}^1$. Moreover, we set 
\begin{align*}
    \norm{\scp{\phi}{\psi}}_{L} \coloneqq \bigscp{\scp{\phi}{\psi}}{\scp{\phi}{\psi}}_{L}^{\frac12}
\end{align*}
for all $\scp{\phi}{\psi}\in\mathcal{H}^1$. The bilinear form $\scp{\cdot}{\cdot}_{L}$ defines an inner product on $\mathcal{V}^1_{L}$, and $\norm{\cdot}_{L}$ defines a norm on $\mathcal{V}^1_{L}$, that is equivalent to the norm $\norm{\cdot}_{\mathcal{H}^1}$ (see \cite[Corollary A.2]{Knopf2021}). Hence, the space $\mathcal{V}^1_{L}$ endowed with $\scp{\cdot}{\cdot}_{L}$ is a Hilbert space.

Next, we define the space
\begin{align*}
    \mathcal{V}_{L}^{-1} \coloneqq \begin{cases} 
    \{\scp{\phi}{\psi}\in(\mathcal{H}^1_L)^\prime : \mean{\phi}{\psi} = 0 \}, &\text{if~} L\in[0,\infty), \\
    \{\scp{\phi}{\psi}\in(\mathcal{H}^1)^\prime: \meano{\phi} = \meang{\psi} = 0 \}, &\text{if~}L=\infty.
    \end{cases}
\end{align*}
Using the Lax--Milgram theorem, one can show that for any $(\phi,\psi)\in\mathcal{V}^{-1}_{L}$, there exists a unique weak solution $\mathcal{S}_{L}(\phi,\psi) = \big(\mathcal{S}_{L}^\Om(\phi,\psi),\mathcal{S}_{L}^\Ga(\phi,\psi)\big)\in\mathcal{V}^1_{L}$ to the following elliptic problem with bulk-surface coupling
\begin{subequations}\label{BSE}
    \begin{alignat}{2}
        -\Lap\mathcal{S}_{L}^\Om(\phi,\psi) &= \phi&&\qquad\text{in~}\Om, \\
        -\Lapg\mathcal{S}_{L}^\Ga(\phi,\psi) + \beta\deln\mathcal{S}_{L}^\Om(\phi,\psi) &= \psi&&\qquad\text{on~}\Ga, \\
        L\deln\mathcal{S}_{L}^\Om(\phi,\psi) &= \beta\mathcal{S}_{L}^\Ga(\phi,\psi) - \mathcal{S}_{L}^\Om(\phi,\psi) &&\qquad\text{on~}\Ga,
    \end{alignat}
\end{subequations}
in the sense that it satisfies the weak formulation
\begin{align*}
    \big(\mathcal{S}_{L}(\phi,\psi),(\zeta,\xi)\big)_{L} = \bigang{(\phi,\psi)}{(\zeta,\xi)}_{\mathcal{H}^1_L}
\end{align*}
for all test functions $(\zeta,\xi)\in\mathcal{H}^1_L$. Consequently, there exists a constant $C > 0$, depending only on $\Om, L$ and $\beta$ such that
\begin{align*}
    \norm{\mathcal{S}_{L}(\phi,\psi)}_{L}\leq C\norm{(\phi,\psi)}_{(\mathcal{H}^1_L)^\prime}
\end{align*}
for all $(\phi,\psi)\in\mathcal{V}^{-1}_{L}$. This allows us to define a solution operator
\begin{align*}
    \mathcal{S}_{L}:\mathcal{V}^{-1}_{L}\rightarrow\mathcal{V}^1_{L}, \quad(\phi,\psi)\mapsto \mathcal{S}_{L}(\phi,\psi) = \big(\mathcal{S}_{L}^\Om(\phi,\psi),\mathcal{S}_{L}^\Ga(\phi,\psi)\big)
\end{align*}
as well as an inner product and its induced norm on $\mathcal{V}^{-1}_{L}$ via
\begin{align*}
    \big((\phi,\psi),(\zeta,\xi)\big)_{L,\ast} &\coloneqq \big(\mathcal{S}_{L}(\phi,\psi),\mathcal{S}_{L}(\zeta,\xi)\big)_{L}, \\
    \norm{(\phi,\psi)}_{L,\ast} &\coloneqq \big((\phi,\psi),(\phi,\psi)\big)_{L,\ast}^{\frac12}
\end{align*}
for all $(\phi,\psi), (\zeta,\xi)\in\mathcal{V}^{-1}_{L}$. This norm is equivalent to the norm $\norm{\cdot}_{(\mathcal{H}^1_L)^\prime}$ on $\mathcal{V}^{-1}_{L}$, see, e.g., \cite[Theorem~3.3 and Corollary~3.5]{Knopf2021} for a proof if $L\in(0,\infty)$. In the other cases, the proof can be carried out similarly.

Lastly, let $m\in\R$ if $L\in[0,\infty)$ or $m = (m_1,m_2)\in\R^2$ if $L = \infty$. Then we define
\begin{align*}
    \mathcal{W}_{K,L,m} \coloneqq \begin{cases} 
    \{\scp{\phi}{\psi}\in\mathcal{H}^1 : \mean{\phi}{\psi} = m \}, &\text{if~} L\in[0,\infty), \\
    \{\scp{\phi}{\psi}\in\mathcal{H}^1: \meano{\phi} = m_1, \ \meang{\psi} = m_2 \}, &\text{if~}L=\infty.
    \end{cases}
\end{align*}
\subsection{Important tools}

Throughout this paper, we will frequently use the bulk-surface Poincar\'{e} inequality, which has been established in \cite[Lemma A.1]{Knopf2021}:
\begin{lemma}
    \label{Prelim:Poincare}
    Let $K\in[0,\infty)$ and $\alpha,\beta\in\R$ such that $\alpha\beta\abs{\Om} + \abs{\Ga}\neq 0$. Then, there exists a constant $C_P > 0$, depending only on $K,\alpha,\beta$ and $\Om$ such that
    \begin{align*}
        \norm{(\zeta,\xi)}_{\mathcal{L}^2} \leq C_P \norm{(\zeta,\xi)}_{K}
    \end{align*}
    for all pairs $(\zeta,\xi)\in\mathcal{H}^1_K$ satisfying $\mean{\zeta}{\xi} = 0$.
\end{lemma}

Furthermore, we recall the following interpolation inequality:
\begin{lemma}
    There exists a constant $C > 0$, such that for all $2 \leq r < \infty$ it holds that
	\begin{align}\label{Prelim:Est:Inteprol}
		\norm{(\zeta,\xi)}_{\mathcal{L}^r}\leq C\sqrt{r}\norm{(\zeta,\xi)}_{\mathcal{L}^2}^{\frac2r}\norm{(\zeta,\xi)}_{\mathcal{H}^1}^{\frac{r-2}{r}} \qquad\text{for all~}(\zeta,\xi)\in\mathcal{H}^1.
	\end{align}
\end{lemma}

Since $\Gamma$ is a $1$-dimensional submanifold of $\R$, we can readily obtain \eqref{Prelim:Est:Inteprol} as an extension of the following Gagliardo--Nirenberg--Sobolev inequality (see \cite[Proposition and Remark~1]{Ozawa1995}): there exists a constant $C > 0$, such that for all $2 \leq r < \infty$ it holds that
\begin{align*}
    \norm{\zeta}_{L^r(\Om)} \leq C\sqrt{r}\norm{\zeta}_{L^2(\Om)}^{\frac{2}{r}}\norm{\zeta}_{H^1(\Om)}^{\frac{r-2}{r}} \qquad\text{for all~}\zeta\in H^1(\Om).
\end{align*}

Lastly, we present the following uniform variant of the Gronwall lemma. A proof can be found, e.g., in \cite[Chapter~III, Lemma~1.1]{Temam1997}.

\begin{lemma}\label{Lemma:Gronwall}
    Let $g,h,y$ be three positive locally integrable functions on $(t_0,\infty)$ such that $y^\prime$ is locally integrable on $(t_0,\infty)$ and which satisfy
    \begin{align*}
        \ddt y &\leq gy + h, \\
        \int_t^{t+r} g(s)\ds \leq a_1, \quad \int_t^{t+r} h(s)\ds &\leq a_2, \quad \int_t^{t+r} y(s)\ds \leq a_3 \quad\text{for all~}t\geq t_0,
    \end{align*}
    where $r,a_1,a_2,a_3$ are positive constants. Then it holds
    \begin{align*}
        y(t) \leq \left(\frac{a_3}{r} + a_2\right)e^{a_1} \qquad\text{for all~}t\geq t_0 + r.
    \end{align*}
\end{lemma}
\medskip

\subsection{Main Assumptions.}
\label{SUBSEC:ASS}

\begin{enumerate}[label=\textnormal{\bfseries(A\arabic*)}]
    \item \label{Ass:Constants} The constants $\alpha,\beta\in\R$ appearing in system \eqref{EQ:SYSTEM} are supposed to satisfy $\alpha\in[-1,1]$ as well as $\alpha\beta\abs{\Om} + \abs{\Ga} \neq 0$.
    \item \label{Ass:Mobility} For the mobility functions we require $m_\Om,m_\Ga\in C([-1,1])$. Furthermore, we assume the existence of constants $m^\ast,M^\ast > 0$ such that
    \begin{align}\label{Ass:Mobility:Bound}
        0 < m^\ast \leq m_\Om(s), m_\Ga(s) \leq M^\ast \qquad\text{for all~}s\in[-1,1].
    \end{align}

    \item \label{Ass:Potentials} For the potentials, we assume that $F,G:\R\rightarrow\R$ are of the form
    \begin{align*}
        F(s) = F_1(s) + F_2(s), \qquad G(s) = G_1(s) + G_2(s),
    \end{align*}
    where $F_1,G_1\in C([-1,1])\cap C^2(-1,1)$ such that $F_1(0) = F_1^\prime(0) = G_1(0) = G_1^\prime(0) = 0$,
    \begin{align*}
        \lim_{s\searrow -1} F_1^\prime(s) = \lim_{s\searrow -1} G_1^\prime(s) = -\infty \quad\text{and}\quad
        \lim_{s\nearrow 1} F_1^\prime(s) = \lim_{s\nearrow 1} G_1^\prime(s) = +\infty,
    \end{align*}
    and there exist constants $\Theta_\Om,\Theta_\Ga > 0$ such that
    \begin{align}\label{Assumption:Pot:Convexity}
        F_1^{\prime\prime}(s) \geq \Theta_\Om \quad\text{and~}\quad G_1^{\prime\prime}(s) \geq \Theta_\Ga \qquad\text{for all~}s\in(-1,1).
    \end{align}
    We extend $F_1$ and $G_1$ on $\R$ by defining $F_1(s) = G_1(s) = + \infty$ for $s\not\in[-1,1]$. For $F_2$ and $G_2$, we assume $F_2,G_2\in C^1(\R)$ such that their derivatives are globally Lipschitz continuous. Lastly, we require that the singular part of the boundary potential dominates the singular part of the bulk potential in the sense that there exist constants $\kappa_1, \kappa_2 > 0$ such that
    \begin{align}\label{Ass:Potentials:Domination}
        \abs{F_1^\prime(\alpha s)}\leq \kappa_1 \abs{G_1^\prime(s)} + \kappa_2 \qquad\text{for all~}s\in(-1,1).
    \end{align}
\end{enumerate}
For certain results established in this work, we require additional growth conditions on the singular components $F_1$ and $G_1$ of the potentials.

\begin{enumerate}[label=\textnormal{\bfseries(A\arabic*)}, start=4]
    \item \label{Ass:Potentials:Sep} We assume that one of the following conditions hold: 
    \begin{enumerate}[label=\textnormal{\bfseries(A4.\arabic*)}]
        \item \label{Ass:Potentials:Sep:1} There exist constants $C_\sharp > 0$ and $\gamma_\sharp\in[1,2)$ such that
            \begin{align}\label{Ass:Potentials:Growth:1}
                F_1^{\prime\prime}(s) \leq C_\sharp \e^{C_\sharp\abs{F_1^\prime(s)}^{\gamma_\sharp}}\qquad\text{for all~}s\in(-1,1).
            \end{align}
        \item \label{Ass:Potentials:Sep:2} As $\delta\searrow 0$, for some $\kappa > \frac12$, it holds that
        \begin{align}\label{Ass:Potentials:Growth:2}
            \frac{1}{F_1^\prime(1 - 2\delta)} = O\left(\frac{1}{\abs{\ln\delta}^\kappa}\right), \qquad \frac{1}{\abs{F_1^\prime(-1 + 2\delta)}} = O\left(\frac{1}{\abs{\ln\delta}^\kappa}\right).
        \end{align}
    \end{enumerate}
\end{enumerate}

\medskip
\section{Main results} 
\label{Section:Main}

\subsection{Well-posedness of weak solutions}

We start by introducing the notion of a weak solution.

\begin{definition}\label{DEF:SING:WS}
    Let $K,L\in[0,\infty]$, and let $\scp{\phi_0}{\psi_0}\in\mathcal{H}^1_K$ be an initial datum satisfying
    \begin{subequations}\label{cond:init}
    \begin{align}\label{cond:init:int}
        \norm{\phi_0}_{L^\infty(\Om)} \leq 1, \qquad \norm{\psi_0}_{L^\infty(\Ga)} \leq 1.
    \end{align}
    In addition, we assume that
    \begin{align}\label{cond:init:mean:L}
       &\beta\,\mean{\phi_0}{\psi_0}\in (-1,1), \quad \mean{\phi_0}{\psi_0}\in(-1,1), \quad \text{if~ }L\in[0,\infty),
    \end{align}
    and 
    \begin{align}
     \label{cond:init:mean:inf}
      &  \meano{\phi_0}\in(-1,1), \quad \meang{\psi_0}\in(-1,1), \quad \text{if~ }L=\infty.
    \end{align}
    \end{subequations}
    The quadruplet $(\phi,\psi,\mu,\theta)$ is a called a weak solution of the system \eqref{EQ:SYSTEM} on $[0,T]$ for $T > 0$ if the following properties hold:
    \begin{enumerate}[label=\textnormal{(\roman*)}, ref=\thetheorem(\roman*), topsep=0ex, leftmargin=*, itemsep=1.5ex]
        \item The functions $\phi, \psi, \mu$ and $\theta$ satisfy
        \begin{subequations}
        \label{REG:SING}
            \begin{align}
                &\scp{\phi}{\psi} \in C([0,T];\mathcal{L}^2)\cap H^1(0,T;(\mathcal{H}_{L}^1)^\prime)\cap 
                L^\infty(0,T;\mathcal{H}_{K}^1), \label{REGPP:SING}\\
                &\scp{\mu}{\theta}\in L^2(0,T;\mathcal{H}_{L}^1) \label{REGMT:SING}, \\
                &\scp{F^\prime(\phi)}{G^\prime(\psi)}\in L^2(0,T;\mathcal{L}^2) \label{REGLC:SING},
            \end{align}
        \end{subequations}
        and it holds 
        \begin{equation}
            \label{PROP:CONF}
            \abs{\phi} < 1 \quad\text{a.e.~in $Q$}
            \quad\text{and}\quad
            \abs{\psi} < 1 \quad\text{a.e.~on $\Sigma$}.
        \end{equation}
    
    \item \label{DEF:SING:WS:IC} The initial conditions are satisfied in the following sense:
    \begin{align*}
        \phi\vert_{t=0} = \phi_0 \quad\text{a.e.~in } \Omega, \quad\text{and} \quad\psi\vert_{t=0} = \psi_0 \quad\text{a.e.~on }\Gamma.
    \end{align*}
    \item \label{DEF:SING:WS:WF} The variational formulation
    \begin{subequations}\label{SING:WF}
        \begin{align}
        &\begin{aligned}
            \ang{\scp{\delt\phi}{\delt\psi}}{\scp{\zeta}{\xi}}_{\mathcal{H}_{L}^1}  &= - \intO m_\Om(\phi)\Grad\mu\cdot\Grad\zeta\dx - \intG  m_\Ga(\psi)\Gradg\theta\cdot\Gradg\xi\dG \\
            &\quad - \chi(L)\intG(\beta\theta-\mu)(\beta\xi - \zeta)\dG,  \label{WF:PP:SING}
        \end{aligned}
            \\
        &\begin{aligned}
            \intO \mu\,\eta\dx + \intG\theta\,\vartheta\dG &=  \intO\Grad\phi\cdot\Grad\eta + F^\prime(\phi)\eta\dx
            + \intG \Gradg\psi\cdot\Gradg\vartheta + G^\prime(\psi)\vartheta\dG\label{WF:MT:SING}
            \\
            &\quad + \chi(K)\intG(\alpha\psi-\phi)(\alpha\vartheta - \eta) \dG, 
        \end{aligned}
        \end{align}
    \end{subequations}
    holds a.e. on $[0,T]$ for all $\scp{\zeta}{\xi}\in\mathcal{H}_{L}^1, \scp{\eta}{\vartheta}\in\mathcal{H}_{K}^1$.
    \item \label{DEF:SING:WS:MCL} The functions $\phi$ and $\psi$ satisfy the mass conservation law
    \begin{align}\label{MCL:SING}
        \begin{dcases}
            \beta\intO \phi(t)\dx + \intG \psi(t)\dG = \beta\intO \phi_0 \dx + \intG \psi_0\dG, &\textnormal{if } L\in[0,\infty), \\
            \intO\phi(t)\dx = \intO\phi_0\dx \quad\textnormal{and}\quad \intG\psi(t)\dG = \intG\psi_0\dG, &\textnormal{if } L = \infty
        \end{dcases}
    \end{align}
    for all $t\in[0,T]$.
    \item \label{DEF:SING:WS:WEDL} The energy inequality
    \begin{align}\label{WEDL:SING}
        \begin{split}
            &E(\phi(t),\psi(t)) + \int_0^t\intO m_\Om(\phi)\abs{\Grad\mu}^2\dxs + \int_0^t\intG m_\Ga(\psi)\abs{\Gradg\theta}^2\dGs \\
            &\quad + \chi(L) \int_0^t\intG (\beta\theta-\mu)^2\dGs \leq E(\phi_0,\psi_0)
        \end{split}
    \end{align}
    holds for all $t\in[0,T]$.
    \end{enumerate}
\end{definition}

Our first main result is devoted to the well-posedness of weak solutions to \eqref{EQ:SYSTEM} in the sense of Definition~\ref{DEF:SING:WS}.

\begin{theorem}\label{Theorem}
	Assume that the assumptions \ref{Ass:Constants}-\ref{Ass:Potentials} hold, and let $K\in(0,\infty]$ and $L\in[0,\infty]$. Then, the following results hold:
	\begin{enumerate}[label = \textbf{(\Roman*)}, leftmargin=*]
		\item \textbf{Existence of weak solutions.} Let $(\phi_0,\psi_0)\in\mathcal{H}^1$ satisfy \eqref{cond:init}. Then there exists a weak solution $(\phi,\psi,\mu,\theta)$ of \eqref{EQ:SYSTEM} such that
	        \begin{subequations}
            \begin{align}\label{ContinuousDependence}
                &(\phi,\psi)\in L^\infty([0,\infty);\mathcal{H}^1)\cap L^4_{\mathrm{uloc}}([0,\infty);\mathcal{H}^2)\cap L^2_{\mathrm{uloc}}([0,\infty);\mathcal{W}^{2,p}), \\
                &(\delt\phi,\delt\psi)\in L^2([0,\infty);(\mathcal{H}^1_L)^\prime), \\
                &(F^\prime(\phi),G^\prime(\psi))\in L^2_{\mathrm{uloc}}([0,\infty);\mathcal{L}^p), \\
                &(\mu,\theta)\in L^2_{\mathrm{uloc}}([0,\infty);\mathcal{H}^1),
            \end{align}
        \end{subequations}
        for any $2\leq p <\infty$.
        \item \textbf{Uniqueness of weak solutions.} Suppose additionally that $m_\Om,m_\Ga\in C^2([-1,1])$. Let $(\phi_1,\psi_1)$, $(\phi_2,\psi_2)$ be two weak solutions originating from initial conditions $(\phi_1^0,\psi_1^0)$, $(\phi_2^0,\psi_2^0)$ satisfying \eqref{cond:init} as well as
    \begin{align}\label{cond:init:uniqueness:conv}
        \begin{dcases}
            \mean{\phi_1^0}{\psi_1^0} = \mean{\phi_2^0}{\psi_2^0}, &\textnormal{if } L\in[0,\infty), \\
            \meano{\phi_1^0} = \meano{\phi_2^0} \quad\textnormal{and}\quad \meang{\psi_1^0} = \meang{\psi_2^0}, &\textnormal{if } L = \infty.
        \end{dcases}
    \end{align}
    Then, for any $T > 0$, there exists a positive constant $C$ such that
    \begin{align*}
    	\norm{(\phi_1(t) - \phi_2(t),\psi_1(t) - \psi_2(t))}_{(\mathcal{H}^1_L)^\prime} \leq C\norm{(\phi_1^0 - \phi_2^0,\psi_1^0 - \psi_2^0)}_{(\mathcal{H}^1_L)^\prime}
    \end{align*}
    for all $t\in[0,T]$.
    The constant $C$ only depends on the parameters of the system, the final time $T$, and the initial free energies $E(\phi_1^0,\psi_1^0)$ and $E(\phi_2^0,\psi_2^0)$. In particular, the weak solution is unique.
	\end{enumerate}
\end{theorem}

\begin{remark}
    The restriction $K\in(0,\infty]$ in Theorem~\ref{Theorem} is necessary, as only in this case we are able to prove that
    \begin{align*}
        (\phi,\psi)\in L^4_{\mathrm{uloc}}([0,\infty);\mathcal{H}^2),
    \end{align*}
    which is an essential ingredient in the proof of the uniqueness of weak solutions. If $K = 0$, one only has
    \begin{align*}
        (\phi,\psi)\in L^3_{\mathrm{uloc}}([0,\infty);\mathcal{H}^2),
    \end{align*}
    see \cite[Theorem~3.3]{Giorgini2025}, which appears insufficient to establish the uniqueness of weak solutions to \eqref{EQ:SYSTEM}.
\end{remark}

\subsection{Propagation of regularity and instantaneous separation property.}
Our next result is concerned with the propagation of regularity. We show that under certain regularity assumptions on the mobility functions, there exists a weak solution to \eqref{EQ:SYSTEM} that enjoys higher regularity properties on the time interval $(\tau,\infty)$ for any $\tau > 0$. 

\begin{theorem}\label{Theorem:PropReg}
Suppose that the assumptions \ref{Ass:Constants}-\ref{Ass:Potentials:Sep} hold.
Let $K\in(0,\infty]$, $L\in[0,\infty]$, let $(\phi_0,\psi_0)\in\mathcal{H}^1$ be an initial datum satisfying \eqref{cond:init}, and let $\tau > 0$. Then there exists a weak solution $(\phi,\psi,\mu,\theta)$ to \eqref{EQ:SYSTEM} in the sense of Definition~\ref{DEF:SING:WS} satisfying
    \begin{align}
    	&(\phi,\psi)\in L^\infty(\tau,\infty;\mathcal{W}^{2,p}), \quad (\delt\phi,\delt\psi)\in L^\infty(\tau,\infty;(\mathcal{H}^1_L)^\prime)\cap L^2_{\mathrm{uloc}}([\tau,\infty);\mathcal{H}^1), \label{PropReg:tau:1}\\
    	&(\mu,\theta)\in L^\infty(\tau,\infty;\mathcal{H}^1_L)\cap L^4_{\mathrm{uloc}}([\tau,\infty);\mathcal{H}^2), \quad (F^\prime(\phi),G^\prime(\psi))\in L^\infty(\tau,\infty;\mathcal{L}^p) \label{PropReg:tau:2}
    \end{align}
    for any $2\leq p < \infty$. Moreover, the equations \eqref{EQ:SYSTEM:1}-\eqref{EQ:SYSTEM:2} are satisfied almost everywhere in $\Om\times(\tau,\infty)$, while \eqref{EQ:SYSTEM:3}-\eqref{EQ:SYSTEM:4} and the boundary conditions \eqref{EQ:SYSTEM:5}-\eqref{EQ:SYSTEM:6} are satisfied almost everywhere on $\Ga\times(\tau,\infty)$. In addition, if $m_\Om,m_\Ga\in C^2([-1,1])$, then $(\mu,\theta)\in L^2_{\mathrm{uloc}}([\tau,\infty);\mathcal{H}^3)$.
\end{theorem}

\begin{remark}
    In addition to the regularities stated in \eqref{PropReg:tau:1}-\eqref{PropReg:tau:2}, we have for any $\tau > 0$ that
    \begin{align}\label{PropReg:F'(psi)}
        F^\prime(\psi)\in L^\infty(\tau,\infty;L^p(\Ga)),
    \end{align}
    see \cite{Lv2024a, Lv2024b}.
\end{remark}

\begin{remark}
    In Theorem~\ref{Theorem:PropReg} we assume that the mobilities satisfy $m_\Om, m_\Ga\in C^1([-1,1])$, which guarantees the existence of at least one weak solution enjoying the regularity properties \eqref{PropReg:tau:1}-\eqref{PropReg:tau:2}. If, in addition, we assume $m_\Om,m_\Ga\in C^2([-1,1])$, then this weak solution is the unique weak solution provided by Theorem~\ref{Theorem}. In this case, the assumption \ref{Ass:Potentials:Sep} on the potential $F$ can also be dropped.
    
   Nevertheless, for certain choices of the parameters $K$ and $L$, a weak-strong uniqueness result can be established under the weaker assumption $m_\Om, m_\Ga\in C^1([-1,1])$. In particular, for such parameters and for any $\tau > 0$, every weak solution $(\phi,\psi,\mu,\theta)$ satisfying the regularity properties \eqref{PropReg:tau:1}-\eqref{PropReg:tau:2} coincides on $[\tau,\infty)$ with the unique strong solution having initial data $(\phi(\tau),\psi(\tau))$.
\end{remark}

As a consequence, we can prove that the weak solution from Theorem~\ref{Theorem:PropReg} satisfies the instantaneous separation property.

\begin{theorem}\label{Theorem:Separation}
Suppose that the assumptions from Theorem~\ref{Theorem:PropReg} hold, and consider the corresponding weak solution $(\phi,\psi,\mu,\theta)$ that satisfies the propagation of regularity. Then, for all $\tau > 0$, there exists $\delta > 0$, additionally depending on the norms of the initial data, such that
\begin{align}\label{Separation:tau}
    \norm{\phi(t)}_{L^\infty(\Om)} \leq 1 - \delta, \quad \norm{\psi(t)}_{L^\infty(\Ga)} \leq 1 - \delta \qquad\text{for all~}t\geq\tau.
\end{align}
\end{theorem}

\begin{remark}
    If there exists $(\mu_0,\theta_0)\in\mathcal{H}^1_L$ such that
    \begin{align*}
        \intO\mu_0\;\eta\dx + \intG\theta_0\;\vartheta\dG &= \intO \Grad\phi_0\cdot\Grad\eta + F^\prime(\phi_0)\dx + \intG \Gradg\psi_0\cdot\Gradg\vartheta + G^\prime(\psi_0)\dG \\
        &\quad + \chi(K) \intG (\alpha\psi_0 - \phi_0)(\alpha\vartheta - \eta)\dG
    \end{align*}
    for all $(\eta,\vartheta)\in\mathcal{H}^1_K$, then the weak solution $(\phi,\psi,\mu,\theta)$ from Theorem~\ref{Theorem:PropReg} is a strong solution, namely, \eqref{PropReg:tau:1}-\eqref{PropReg:tau:2} hold for $\tau = 0$. Furthermore, the separation property \eqref{Separation:tau} holds also for $\tau = 0$.
\end{remark}

The proofs of Theorem~\ref{Theorem:PropReg} and Theorem~\ref{Theorem:Separation} are presented in Section~\ref{Section:PropRegSep}.

\subsection{Long-time behavior.}
Thanks to the separation property proven in Theorem~\ref{Theorem:Separation}, we can show that the unique weak solution converges to a single equilibrium as $t\rightarrow\infty$.

\begin{theorem}\label{Theorem:LongTime}
    Suppose that the assumptions from Theorem~3.4 hold. In addition, assume that $m_\Om,m_\Ga\in C^2([-1,1])$ satisfy \ref{Ass:Mobility}, and that $F_1, G_1$ are real analytic on $(-1,1)$ and $F_2, G_2$ are real analytic on $\R$. Let $(\phi,\psi,\mu,\theta)$ be the unique global weak solution obtained in Theorem~\ref{Theorem}. Then it holds
    \begin{align*}
        \lim_{t\rightarrow\infty}\norm{(\phi(t) - \phi_\infty,\psi(t) - \psi_\infty)}_{\mathcal{H}^2} = 0,
    \end{align*}
    where $(\phi_\infty,\psi_\infty)\in\mathcal{H}^2$ is a solution to the stationary bulk-surface Cahn--Hilliard equation
    \begin{subequations}\label{SYSTEM:EQUILIBRIUM}
        \begin{alignat*}{2}
            -\Lap\phi_\infty + F^\prime(\phi_\infty) &= \mu_\infty &&\qquad\text{in~}\Om, \\
            -\Lapg\psi_\infty + G^\prime(\psi_\infty) + \alpha\deln\phi_\infty &= \theta_\infty &&\qquad\text{on~}\Ga, \\
            K\deln\phi_\infty &= \alpha\psi_\infty - \phi_\infty &&\qquad\text{on~}\Ga,
        \end{alignat*}
    \end{subequations}
    with
    \begin{align*}
        \begin{dcases}
            \beta\intO \phi_\infty\dx + \intG \psi_\infty\dG = \beta\intO \phi_0 \dx + \intG \psi_0\dG, &\textnormal{if } L\in[0,\infty), \\
            \intO\phi_\infty\dx = \intO\phi_0\dx \quad\textnormal{and}\quad \intG\psi_\infty\dG = \intG\psi_0\dG, &\textnormal{if } L = \infty.
        \end{dcases}
    \end{align*}
    The constants $\mu_\infty$ and $\theta_\infty$ appearing in \eqref{SYSTEM:EQUILIBRIUM} satisfy
    \begin{align*}
        \mu_\infty = \beta\theta_\infty = \frac{\beta}{\alpha\beta\abs{\Om} + \abs{\Ga}}\Big(\alpha\intO F^\prime(\phi_\infty)\dx + \intG G^\prime(\psi_\infty)\dG\Big)
    \end{align*}
    if $L\in[0,\infty)$, while in the case $L = \infty$, they are given by
    \begin{align*}
        \mu_\infty &= \frac{1}{\abs{\Om}}\Big(\intO F^\prime(\phi_\infty)\dx - \intG \deln\phi_\infty\dG\Big), \\
        \theta_\infty &= \frac{1}{\abs{\Ga}}\Big(\intG G^\prime(\psi_\infty) + \alpha\deln\phi_\infty\dG\Big).
    \end{align*}
\end{theorem}

\medskip
\section{Elliptic bulk-surface system with non-constant coefficients}
\label{Section:EBS}

In this section, we establish well-posedness and regularity results for an elliptic system with bulk-surface coupling and non-constant coefficients. These regularity results will be of crucial importance in the proof of the uniqueness of weak solutions to the system \eqref{EQ:SYSTEM}. Let $\Om\subset\R^d$, $d=2,3$, be a bounded domain with boundary $\Ga \coloneqq \partial\Om$. The precise system under investigation in this section is the following
\begin{subequations}\label{SYSTEM:EBS:MOB}
    \begin{alignat}{2}
        -\Div(m_\Om(\phi)\Grad u) &= f &&\qquad\text{in~}\Om, \\
        -\Divg(m_\Ga(\psi)\Gradg v) + \beta m_\Om(\phi)\deln u &= g &&\qquad\text{on~}\Ga, \\ 
        Lm_\Om(\phi)\deln u &= \beta v - u &&\qquad\text{on~}\Ga,
    \end{alignat}
\end{subequations}
where the mobility functions $m_\Om, m_\Ga$ are supposed to satisfy \ref{Ass:Mobility}.
Moreover, $\phi:\Om\rightarrow\R$ and $\psi:\Ga\rightarrow\R$ are given measurable functions with $\abs{\phi} \leq 1$ a.e. in $\Om$ and $\abs{\psi} \leq 1 $ a.e. on $\Ga$. System \eqref{SYSTEM:EBS:MOB} can be seen as an extension of \eqref{BSE} and the corresponding results proven in \cite{Knopf2021}.

We call $(u,v)\in\mathcal{H}^1_L$ a weak solution to \eqref{SYSTEM:EBS:MOB} if it satisfies the following weak formulation
\begin{align}\label{WF:EBS}
    \begin{split}
        &\intO m_\Om(\phi)\Grad u \cdot\Grad\zeta\dx + \intG m_\Ga(\psi)\Gradg v \cdot\Gradg\xi \dG + \chi(L)\intG (\beta v - u)(\beta\xi - \zeta)\dG \\
        &\quad = \bigang{(f,g)}{(\zeta,\xi)}_{\mathcal{H}^1_L}
    \end{split}
\end{align}
for all $(\zeta,\xi)\in\mathcal{H}^1_L$.

In our first result, we establish the existence of a unique weak solution to \eqref{SYSTEM:EBS:MOB}.

\begin{theorem}\label{Theorem:EBS:WS}
    Assume that $\Om$ is a Lipschitz domain, and let $(f,g)\in\mathcal{V}_{L}^{-1}$. Then there exists a unique weak solution $(u,v)\in\mathcal{H}^1_L$ to \eqref{SYSTEM:EBS:MOB}. Additionally, there exists a constant $C > 0$, depending only on $\Om, L, \beta$ and $m^\ast$ such that
    \begin{align}\label{Est:EBS:Apriori}
    	\norm{(u,v)}_{L} \leq C\norm{(f,g)}_{(\mathcal{H}^1_L)^\prime}.
    \end{align}
\end{theorem}

We omit the proof here, as it follows by analogous arguments to those used in the case of constant coefficients, see \cite[Theorem~3.3]{Knopf2021} for details.

In view of Theorem~\ref{Theorem:EBS:WS}, we can define a solution operator 
\begin{align*}
    \mathcal{S}_{L}[\phi,\psi]:\mathcal{V}_{L}^{-1}\rightarrow\mathcal{V}_{L}^1, \quad(f,g)\mapsto\mathcal{S}_{L}[\phi,\psi](f,g) = (\mathcal{S}_{L}^\Om[\phi,\psi](f,g),\mathcal{S}_{L}^\Ga[\phi,\psi](f,g)),
\end{align*}
where $\mathcal{S}_L[\phi,\psi](f,g)$ is the unique weak solution to \eqref{SYSTEM:EBS:MOB}. Moreover, we can define an inner product and its induced norm on $\mathcal{V}^1_{L}$ by
\begin{align*}
    \big((f,g),(\zeta,\xi)\big)_{L,[\phi,\psi]} &\coloneqq \intO m_\Om(\phi)\Grad f\cdot\Grad\zeta\dx + \intG m_\Ga(\psi)\Gradg g\cdot\Gradg\xi\dG \\
    &\quad + \chi(L)\intG (\beta g - f)(\beta\xi - \zeta)\dG, \\
    \norm{(f,g)}_{L,[\phi,\psi]} &\coloneqq \big((f,g),(f,g)\big)_{L,[\phi,\psi]}^{\frac12}
\end{align*}
for all $(f,g), (\zeta,\xi)\in\mathcal{V}^1_{L}$.
One readily sees that
\begin{align}\label{NormEquivalence:1}
    \min\{1,\sqrt{m^\ast}\}\norm{(f,g)}_{L,[\phi,\psi]} \leq \norm{(f,g)}_{L} \leq \max\{1,\sqrt{M^\ast}\}\norm{(f,g)}_{L,[\phi,\psi]}
\end{align}
for all $(f,g)\in\mathcal{H}^1_L$. In particular, the norms $\norm{\cdot}_{L}$ and $\norm{\cdot}_{L,[\phi,\psi]}$ are equivalent on $\mathcal{V}^1_{L}$.

Next, we define an inner product and its induced norm on $\mathcal{V}^{-1}_{L}$ by
\begin{align*}
    \big((f,g),(\zeta,\xi)\big)_{L,[\phi,\psi],\ast} &\coloneqq \big(\mathcal{S}_{L}[\phi,\psi](f,g),\mathcal{S}_{L}[\phi,\psi](\zeta,\xi)\big)_{L,[\phi,\psi]}, \\
    \norm{(f,g)}_{L,[\phi,\psi],\ast} &\coloneqq \big((f,g),(f,g)\big)_{L,[\phi,\psi],\ast}^{\frac12}
\end{align*}
for all $(f,g), (\zeta,\xi)\in\mathcal{V}^{-1}_{L}$. Using the respective weak formulation satisfied by the solution operators $\mathcal{S}_{L}$ and $\mathcal{S}_{L}[\phi,\psi]$, one can readily check that the norms $\norm{\cdot}_{L,[\phi,\psi],\ast}$ and $\norm{\cdot}_{L,\ast}$ are equivalent on $\mathcal{V}^{-1}_{L}$ with
\begin{align}\label{NormEquivalence}
    \min\{1,\sqrt{m^\ast}\}\norm{(f,g)}_{L,[\phi,\psi],\ast} \leq \norm{(f,g)}_{L,\ast} \leq \max\{1,\sqrt{M^\ast}\}\norm{(f,g)}_{L,[\phi,\psi],\ast}
\end{align}
for all $(f,g)\in\mathcal{V}^{-1}_{L}$. In particular, we also obtain that $\norm{\cdot}_{L,[\phi,\psi],\ast}$ and $\norm{\cdot}_{(\mathcal{H}^1_L)^\prime}$ are equivalent on $\mathcal{V}^{-1}_{L}$.

Besides, we have
\begin{align}\label{Est:fg:L^2:SolOp:1}
    \begin{split}
        \norm{(f,g)}_{\mathcal{L}^2} &= \sqrt{\big(\mathcal{S}_{L}[\phi,\psi](f,g),(f,g)\big)_{L,[\phi,\psi]}} \\
        &\leq \max\{1,\sqrt{M^\ast}\}\norm{\mathcal{S}_{L}[\phi,\psi](f,g)}_{L}^{\frac12}\norm{(f,g)}_{L}^{\frac12}
    \end{split}
\end{align}
for all $(f,g)\in\mathcal{V}^{-1}_{L}\cap\mathcal{H}^1_L$.

Our next result establishes higher regularity results for $\mathcal{S}_{L}[\phi,\psi](f,g)$.

\begin{proposition}\label{Proposition:BSE:HighReg:2}
	Let $\Omega$ be a domain of class $C^2$, let $(\phi,\psi)\in\mathcal{W}^{1,\infty}$, $m_\Om,m_\Ga\in C^1([-1,1])$, and consider a pair $(f,g)\in\mathcal{V}^{-1}_{L}\cap\mathcal{L}^2$. Then $\mathcal{S}_{L}[\phi,\psi](f,g)\in\mathcal{H}^2$. Additionally, there exists a constant $C > 0$ such that
	\begin{align}\label{Est:Sol:G:H^2:thm}
        \begin{split}
    		&\norm{\mathcal{S}_{L}[\phi,\psi](f,g)}_{\mathcal{H}^2} \\
            &\quad\leq C\big(\norm{(f,g)}_{\mathcal{L}^2} + \norm{(\Grad\phi\cdot\Grad\mathcal{S}_{L}^\Om[\phi,\psi](f,g),\Gradg\psi\cdot\Gradg\mathcal{S}_{L}^\Ga[\phi,\psi](f,g))}_{\mathcal{L}^2}\big).
        \end{split}
	\end{align}

\end{proposition}

\begin{proof}
	For the sake of brevity, we use again $(u,v)$ to denote the unique solution $\mathcal{S}_{L}[\phi,\psi](f,g)$ of \eqref{SYSTEM:EBS:MOB}. We start by showing that $(u,v)\in\mathcal{H}^2$. To this end, we make a case distinction according to the parameter $L\in[0,\infty]$.
	
	\textit{Case $L = 0$.} Choosing $\xi = 0$ in the weak formulation \eqref{WF:EBS}, we obtain
	\begin{align}\label{L=0:Test}
		\intO m_\Om(\phi)\Grad u\cdot\Grad\zeta\dx = \intO f\zeta\dx
	\end{align}
	for all $\zeta\in H^1_0(\Om)$. Noting on $\phi\in W^{1,\infty}(\Om)$, we can use $\zeta = \frac{\bar\zeta}{m_\Om(\phi)}\in H^1_0(\Om)$ as a test function in \eqref{L=0:Test} for some $\bar\zeta\in C_c^\infty(\Om)$. Consequently, it holds that
    \begin{align*}
        \intO \Grad u\cdot\Grad\bar\zeta\dx = \intO \frac{1}{m_\Om(\phi)}\Big(f + \Grad m_\Om(\phi)\cdot\Grad u\Big) \bar\zeta\dx
    \end{align*}
    for all $\bar\zeta\in C_c^\infty(\Om)$. In particular, as $\bar\zeta$ was arbitrary, this implies that the distributional derivative $\Lap u$ belongs to $L^2(\Om)$ and satisfies
	\begin{align*}
		-\Lap u = \frac{1}{m_\Om(\phi)}\Big(f + \Grad m_\Om(\phi)\cdot\Grad u\Big) \qquad\text{a.e. in~}\Om.
	\end{align*}
	As we further know that $u\vert_\Ga = \beta v \in H^1(\Ga)$, we may apply elliptic regularity theory for the Poisson--Dirichlet problem (see, e.g., \cite[Theorem~3.2]{Brezzi1987} or \cite[Theorem~A.2]{Colli2019a}) to conclude that $u\in H^{\frac32}(\Om)$ with
	\begin{align}\label{Est:BSE:H^3/2}
		\norm{u}_{H^{\frac32}(\Om)} \leq C\big(\norm{f}_{L^2(\Om)} + \norm{\Grad\phi\cdot\Grad u}_{L^2(\Om)} + \norm{v}_{H^1(\Ga)}\big).
	\end{align}
	Now, since $\Lap u\in L^2(\Om)$ and $u\in H^{\frac32}(\Om)$, we can use a variant of the elliptic trace theorem (see, e.g., \cite[Theorem~2.27]{Brezzi1987} or \cite[Theorem~A.1]{Colli2019a}) to deduce that $\deln u\in L^2(\Ga)$ with
	\begin{align}\label{Est:BSE:deln}
		\norm{\deln u}_{L^2(\Ga)} \leq C\norm{u}_{H^{\frac32}(\Om)}.
	\end{align}
	Consequently, we find that
	\begin{align}\label{Id:BSE:deln}
		\intO m_\Om(\phi)\Grad u\cdot\Grad\zeta\dx = \intO f\zeta\dx + \intG m_\Om(\phi)\deln u\zeta\dG \qquad\text{for all~}\zeta\in H^1(\Om).
	\end{align}
	Consider now an arbitrary function $\xi\in H^1(\Ga)$. According to the inverse trace theorem (see, e.g., \cite[Theorem~4.2.3]{Hsiao2008}), there exists a function $\bar\xi\in H^{\frac32}(\Om)$ such that $\bar\xi\vert_\Ga = \xi$ a.e. on $\Ga$. Choosing $\zeta = \beta\bar\xi$, we see that $(\zeta,\xi)\in\mathcal{H}^1_L$ is an admissible test function in \eqref{WF:EBS}. Using the identity \eqref{Id:BSE:deln}, we obtain
	\begin{align*}
		\intG m_\Ga(\psi)\Gradg v \cdot\Grad\xi\dG = \intG (g - \beta m_\Om(\phi)\deln u)\xi\dG.
	\end{align*}
	As $\xi\in H^1(\Ga)$ was arbitrarily chosen, we infer similarly to above that $v$ is a weak solution of the surface elliptic equation
	\begin{align*}
		-\Lapg v = \frac{1}{m_\Ga(\psi)}\Big( g - \beta m_\Om(\phi)\deln u + \Gradg m_\Ga(\psi)\cdot\Gradg v\Big) \qquad\text{on~}\Ga.
	\end{align*}
    Since $\psi\in W^{1,\infty}(\Ga)$ and $m_\Ga^\prime$ is bounded, we infer form the Sobolev inequality that $\Gradg m_\Ga(\psi)\in L^\infty(\Ga)$, which, in combination with $v\in H^1(\Ga)$ implies that $\Gradg m_\Ga(\psi)\cdot\Gradg v\in L^2(\Ga)$. Thus, since $m_\Om$ is bounded, and $\deln u\in L^2(\Ga)$ according to \eqref{Est:BSE:deln}, we find that
    \begin{align*}
        -\Lapg v = \tilde{g}\in L^2(\Ga).
    \end{align*}
	Recalling that $\Ga$ is a compact submanifold of class $C^2$ without boundary, we can apply regularity theory for elliptic equations on submanifolds (see, e.g., \cite[s.5, Theorem~1.3]{Taylor}) to infer $v\in H^2(\Ga)$, together with the estimate
	\begin{align}\label{Est:BSE:H^2:v:0}
		\begin{split}
			\norm{v}_{H^2(\Ga)} &\leq C\big(\norm{g}_{L^2(\Ga)} + 						\norm{\deln u}_{L^2(\Ga)} + \norm{\Gradg m_\Ga(\psi)						\cdot\Gradg v}_{L^2(\Ga)}\big) \\
			&\leq C\big(\norm{(f,g)}_{\mathcal{L}^2} + \norm{\Gradg 					\psi\cdot\Gradg v}_{L^2(\Ga)}\big).
		\end{split}
	\end{align}
	Here, we have additionally used \eqref{Est:EBS:Apriori}, \eqref{Est:BSE:H^3/2}, and \eqref{Est:BSE:deln} in the last inequality. Since $u\vert_\Ga = \beta v$ a.e. on $\Ga$, we further deduce that $u\vert_\Ga\in H^2(\Ga)$. Recalling that $-\Lap u\in L^2(\Om)$, we eventually conclude from elliptic regularity theory for the Poisson--Dirichlet problem (see, e.g., \cite[Theorem 3.2]{Brezzi1987} or \cite[Theorem A.2]{Colli2019a}) that $u\in H^2(\Om)$ with
	\begin{align}\label{Est:BSE:H^2:u:0}
		\begin{split}
			\norm{u}_{H^2(\Om)} &\leq C\big(\norm{f}_{L^2(\Om)} + \norm{\Grad m_\Om(\phi)\cdot\Grad u}_{L^2(\Om)} + \norm{v}_{H^2(\Ga)}\big) \\
			&\leq C\big(\norm{(f,g)}_{\mathcal{L}^2} + \norm{(\Grad\phi\cdot\Grad u,\Gradg\psi\cdot\Gradg v)}_{\mathcal{L}^2}\big).
		\end{split}
	\end{align}
	Finally, combining \eqref{Est:BSE:H^2:v:0}-\eqref{Est:BSE:H^2:u:0} leads to
	\begin{align}\label{Est:BSE:H^2:0}
		\norm{(u,v)}_{\mathcal{H}^2} \leq C\big(\norm{(f,g)}_{\mathcal{L}^2} + \norm{(\Grad\phi\cdot\Grad u,\Gradg\psi\cdot\Gradg v)}_{\mathcal{L}^2}\big).
	\end{align}	
	
	\textit{Case $L\in(0,\infty)$.} Here, we fix $\zeta = 0$. Then, the weak formulation \eqref{WF:EBS} reduces to
	\begin{align*}
		\intG m_\Ga(\psi)\Gradg v\cdot\Grad\xi\dG + \chi(L)\intG (\beta v - u)\beta\xi\dG = \intG g\xi\dG
	\end{align*}
	for all $\xi\in H^1(\Ga)$. This means that $v$ is a weak solution of the surface elliptic problem
	\begin{align*}
		-\Lapg v = \frac{1}{m_\Ga(\psi)}\Big(g - \beta\chi(L)(\beta v - u) + \Gradg m_\Ga(\psi)\cdot\Gradg v\Big) \qquad\text{on~}\Ga.
	\end{align*}
    Here, we can argue similarly to the case $L = 0$. Indeed, since $u\in H^1(\Om)$, we have by the trace theorem $\beta v - u\vert_\Ga\in H^{\frac12}(\Ga)$, and thus, in particular, $\beta v - u\vert_\Ga\in L^2(\Ga)$. Therefore, we deduce with elliptic regularity theory on submanifolds that $v\in H^2(\Ga)$ together with the estimate
	\begin{align}\label{Est:BSE:H^2:v:L}
		\begin{split}
			\norm{v}_{H^2(\Ga)} &\leq C\big(\norm{g}_{L^2(\Ga)} + 						\norm{v}_{L^2(\Ga)} + \norm{u}_{L^2(\Ga)} + \norm{\Gradg 					m_\Ga(\psi)\cdot\Gradg v}_{L^2(\Ga)}\big) \\
			&\leq C\big(\norm{(f,g)}_{\mathcal{L}^2} + 									\norm{\Gradg\psi\cdot\Gradg v}_{L^2(\Ga)}\big)
		\end{split}
	\end{align}
	for some constant $C > 0$. Next, we choose $\xi = 0$ in the weak formulation \eqref{WF:EBS}, which yields
	\begin{align*}
		\intO m_\Om(\phi)\Grad u\cdot\Grad\zeta \dx + \chi(L)\intG (\beta v - u)\zeta\dG = \intO f\zeta\dx
	\end{align*}
	for all $\zeta\in H^1(\Om)$. This means that $u$ is a weak solution to the Poisson--Neumann problem
	\begin{alignat*}{2}
		-\Lap u &= \frac{1}{m_\Om(\phi)}\Big( f + \Grad m_\Om(\phi)\cdot\Grad u\Big) &&\qquad\text{in~}\Om, \\
		\deln u &= \frac{1}{m_\Om(\phi)} \chi(L)(\beta v - u) &&\qquad\text{on~}\Ga.
	\end{alignat*}
	Applying elliptic regularity theory for Poisson’s equation with inhomogeneous Neumann boundary condition (see, e.g., \cite[s.5, Proposition~7.7]{Taylor}), we deduce $u\in H^2(\Om)$ together with the existence of a constant $C > 0$ such that
	\begin{align}\label{Est:BSE:H^2:u:L}
		\begin{split}
			\norm{u}_{H^2(\Om)} &\leq C\big(\norm{f}_{L^2(\Om)} + 						\norm{v}_{H^{\frac12}(\Ga)} + \norm{u}_{H^{\frac12}(\Ga)} + \norm{u}				_{L^2(\Om)} + \norm{\Grad m_\Om(\phi)\cdot\Grad u}_{L^2(\Om)}				\big) \\
			&\leq C\big(\norm{(f,g)}_{\mathcal{L}^2} + 									\norm{\Grad\phi\cdot\Grad u}_{L^2(\Om)}\big).
		\end{split}
	\end{align}
	Combining estimates \eqref{Est:BSE:H^2:v:L} and \eqref{Est:BSE:H^2:u:L} yields
	\begin{align}\label{Est:BSE:H^2:L}
		\norm{(u,v)}_{\mathcal{H}^2} \leq C\big(\norm{(f,g)}_{\mathcal{L}^2} + \norm{(\Grad \phi\cdot\Grad u,\Gradg\psi\cdot\Gradg v)}_{\mathcal{L}^2}\big).
	\end{align}
	
	\textit{Case $L = \infty$.} In this case, the system \eqref{SYSTEM:EBS:MOB} decouples to the non-homogeneous Poisson problem with a homogeneous Neumann boundary condition
	\begin{alignat}{2}
		-\Div(m_\Om(\phi)\Grad u) &= f  &&\qquad\text{in~}\Om, \label{Dirichlet-Problem:Neumann}\\
		m_\Om(\phi)\deln u &= 0 &&\qquad\text{on~}\Ga, \label{Dirichlet-Problem:NeumannBC}
	\end{alignat}
	and the non-homogeneous surface elliptic problem
	\begin{align}\label{LaplaceBeltrami:2}
		-\Divg(m_\Ga(\psi)\Gradg v) = g \qquad\text{on~}\Ga.
	\end{align}
	Applying elliptic regularity theory for Poisson's problem with homogeneous Neumman boundary condition \eqref{Dirichlet-Problem:Neumann}-\eqref{Dirichlet-Problem:NeumannBC} and regularity theory for the Laplace--Beltrami equation \eqref{LaplaceBeltrami:2}, respectively, we directly conclude the desired regularity $u\in H^2(\Om)$ and $v\in H^2(\Ga)$ along with the estimates
	\begin{align*}
		\norm{u}_{H^2(\Om)} &\leq C\big(\norm{f}_{L^2(\Om)} + \norm{\Grad\phi\cdot\Grad u}_{L^2(\Om)}\big), \\
		\norm{v}_{H^2(\Ga)} &\leq C\big(\norm{g}_{L^2(\Ga)} + \norm{\Gradg\psi\cdot\Gradg v}_{L^2(\Ga)}\big). \qedhere
	\end{align*}
\end{proof}

\begin{remark}
    Under the assumptions of Proposition~\ref{Proposition:BSE:HighReg:2}, the estimate \eqref{Est:fg:L^2:SolOp:1} even holds for all $(f,g)\in\mathcal{V}^{-1}_{L}\cap\mathcal{H}^1$, i.e., the functions $f$ and $g$ do not need to satisfy the trace relation $f = \beta g$ on $\Ga$ if $L = 0$. This will be essential 
    in the proof of Theorem~\ref{Theorem}. Moreover, since $\mean{f}{g} = 0$, the bulk-surface Poincar\'{e} inequality yields, for $K\in(0,\infty)$,
    \begin{align}\label{Est:fg:L^2:SolOp:2}
        \norm{(f,g)}_{\mathcal{L}^2} \leq \max\{1,\sqrt{M^\ast}\}C_PC\norm{\mathcal{S}_{L}[\phi,\psi](f,g)}_{L}^{\frac12}\norm{(f,g)}_{K}^{\frac12}.
    \end{align}
    In the case $K = \infty$, the bulk-surface Poincar\'{e} inequality is in general not available. Instead, we use
    \begin{align*}
        \norm{(f,g)}_{L}^{\frac12} \leq C\norm{(f,g)}_{\mathcal{H}^1}^{\frac12} &\leq C\big(\norm{(f,g)}_{\mathcal{L}^2}^{\frac12} + \norm{(\Grad f, \Gradg g)}_{\mathcal{L}^2}^{\frac12}\big) \\
        &= C\big(\norm{(f,g)}_{\mathcal{L}^2}^{\frac12} + \norm{(f,g)}_{K}^{\frac12}\big),
    \end{align*}
    which leads to
    \begin{align*}
        \norm{(f,g)}_{\mathcal{L}^2} &\leq \max\{1,\sqrt{M^\ast}\}C_PC\norm{\mathcal{S}_{L}[\phi,\psi](f,g)}_{L}^{\frac12}\big(\norm{(f,g)}_{\mathcal{L}^2}^{\frac12} + \norm{(f,g)}_{K}^{\frac12}\big) \\
        &\leq \frac12\norm{(f,g)}_{\mathcal{L}^2} + \max\{1,\sqrt{M^\ast}\}C_PC\norm{\mathcal{S}_{L}[\phi,\psi](f,g)}_{L} \\
        &\quad + \max\{1,\sqrt{M^\ast}\}C_PC\norm{\mathcal{S}_{L}[\phi,\psi](f,g)}_{L}^{\frac12}\norm{(f,g)}_{K}^{\frac12}.
    \end{align*}
    Consequently, we obtain
    \begin{align}\label{Est:fg:L^2:SolOp:2:K=infty}
        \begin{split}
            \norm{(f,g)}_{\mathcal{L}^2} &\leq \max\{1,\sqrt{M^\ast}\}C_PC\norm{\mathcal{S}_{L}[\phi,\psi](f,g)}_{L} \\
            &\quad + \max\{1,\sqrt{M^\ast}\}C_PC\norm{\mathcal{S}_{L}[\phi,\psi](f,g)}_{L}^{\frac12}\norm{(f,g)}_{K}^{\frac12}.
        \end{split}
    \end{align}
\end{remark}

Similarly to the proof of Proposition~\ref{Proposition:BSE:HighReg:2}, we can prove $\mathcal{H}^3$-regularity of the unique solution to \eqref{SYSTEM:EBS:MOB} provided that the mobility functions are more regular.

\begin{corollary}\label{Corollary:BSE:H3}
    Let $\Om$ be of class $C^3$, let $(\phi,\psi)\in\mathcal{W}^{2,4}$, and assume that $m_\Om,m_\Ga\in C^2([-1,1])$ and $(f,g)\in\mathcal{V}^{-1}_{L}\cap\mathcal{H}^1$. Then $\mathcal{S}_L[\phi,\psi](f,g)\in\mathcal{H}^3$, and there exists a constant $C > 0$ such that
    \begin{align}\label{Est:Sol:G:H^3}
        &\norm{\mathcal{S}_{L}[\phi,\psi](f,g)}_{\mathcal{H}^3} \nonumber \\
        &\quad\leq C\big(1 + \mathbf{1}_{\{0\}}(L)\norm{(\phi,\psi)}_{\mathcal{H}^2}\big) \nonumber \\
        &\qquad\times\Bigg( \Bignorm{\bigg(\frac{f}{m_\Om(\phi)},\frac{g}{m_\Ga(\psi)}\bigg)}_{\mathcal{H}^1} \\
        &\qquad\qquad + \Bignorm{\bigg(\frac{m_\Om^\prime(\phi)\Grad\phi\cdot\Grad\mathcal{S}_{L}^\Om[\phi,\psi](f,g)}{m_\Om(\phi)},\frac{m_\Ga^\prime(\psi)\Gradg\psi\cdot\Gradg\mathcal{S}_{L}^\Ga[\phi,\psi](f,g)}{m_\Ga(\psi)}\bigg)}_{\mathcal{H}^1}\Bigg), \nonumber
    \end{align}
    where $\mathbf{1}_{\{0\}}(\cdot)$ denotes the indicator function of the set $\{0\}$.
\end{corollary}

\begin{proof}
    Abbreviating again $(u,v) = \mathcal{S}_{L}[\phi,\psi](f,g)$, we can move along the lines from the proof of Proposition~\ref{Proposition:BSE:HighReg:2} and show that $(u,v)\in\mathcal{H}^3$ in combination with the estimate
    \begin{align*}
        &\norm{(u,v)}_{\mathcal{H}^3} \\
        &\quad\leq C\Bigg( \Bignorm{\bigg(\frac{f}{m_\Om(\phi)},\frac{g}{m_\Ga(\psi)}\bigg)}_{\mathcal{H}^1} + \Bignorm{\bigg(\frac{m_\Om^\prime(\phi)\Grad\phi\cdot\Grad u}{m_\Om(\phi)},\frac{m_\Ga^\prime(\psi)\Gradg\psi\cdot\Gradg v}{m_\Ga(\psi)}\bigg)}_{\mathcal{H}^1} \\
        &\qquad + \Bignorm{\frac{m_\Om(\phi)\deln u}{m_\Ga(\psi)}}_{H^1(\Ga)}\Bigg).
    \end{align*}
    In the following, we will show that the last summand on the right-hand side can be bounded in terms of the first two. To do so, we only have to take a closer look at the case $L = 0$. Indeed, if $L = \infty$, this term clearly vanishes, and for $L\in(0,\infty)$, we make use of the boundary condition $Lm_\Om(\phi)\deln u = \beta v - u$ a.e.~on $\Ga$ to find that
    \begin{align*}
        \Bignorm{\frac{m_\Om(\phi)\deln u}{m_\Ga(\psi)}}_{H^1(\Ga)} = \frac{1}{L}\Bignorm{\frac{\beta v - u}{m_\Ga(\psi)}}_{H^1(\Ga)} &\leq C \norm{(u,v)}_{\mathcal{H}^2} \\
        &\leq C\big(\norm{(f,g)}_{\mathcal{L}^2} + \norm{(\Grad\phi\cdot\Grad u,\Gradg\psi\cdot\Gradg v)}_{\mathcal{L}^2}\big)
    \end{align*}
    in view of \eqref{Est:Sol:G:H^2:thm}. Lastly, if $L = 0$, it holds that
    \begin{align*}
        \Gradg\bigg(\frac{m_\Om(\phi)\deln u}{m_\Ga(\psi)}\bigg) = \frac{m_\Om^\prime(\phi)\deln u\Gradg\phi}{m_\Ga(\psi)^2} + \frac{m_\Om(\phi)\Gradg\deln u}{m_\Ga(\psi)^2} - \frac{m_\Om(\phi)m_\Ga^\prime(\psi)\deln u\Gradg\psi}{m_\Ga(\psi)^2} \quad\text{a.e.~on~}\Ga.
    \end{align*}
    For the first term, we use the trace theorem together with the Sobolev inequality and find
    \begin{align*}
        \norm{\deln u\Gradg\phi}_{L^2(\Ga)} \leq \norm{\deln u}_{L^4(\Ga)}\norm{\Gradg\phi}_{L^4(\Ga)} \leq C\norm{\phi}_{H^2(\Om)}\norm{u}_{H^2(\Om)}.
    \end{align*}
    Next, for the second term, we argue as in the beginning of the proof of Theorem~\ref{Proposition:BSE:HighReg:2}, in particular as for \eqref{Est:BSE:H^3/2} and \eqref{Est:BSE:deln}. To be precise, we obtain
    \begin{align*}
        \norm{\Gradg\deln u}_{L^2(\Ga)} \leq \norm{\deln u}_{H^1(\Ga)} &\leq C\norm{u}_{H^{\frac52}(\Om)} \\
        &\leq C\Big(\Bignorm{\frac{f}{m_\Om(\phi)}}_{H^1(\Om)} + \Bignorm{\frac{m_\Om^\prime(\phi)\Grad\phi\cdot\Grad u}{m_\Om(\phi)}}_{H^1(\Om)} + \norm{v}_{H^2(\Ga)}\Big).
    \end{align*}
    Finally, for the last term, we combine the trace theorem again with the Sobolev inequality to deduce that
    \begin{align*}
        \norm{\deln u\Gradg\psi}_{L^2(\Ga)} \leq \norm{\deln u}_{L^4(\Ga)}\norm{\Gradg\psi}_{L^4(\Ga)}\leq C\norm{\psi}_{H^2(\Ga)}\norm{u}_{H^2(\Om)}.
    \end{align*}
    Thus, from the estimates above and \eqref{Est:Sol:G:H^2:thm} we conclude \eqref{Est:Sol:G:H^3}, which finishes the proof.
\end{proof}

Next, we aim to generalize the result from Proposition~\ref{Proposition:BSE:HighReg:2} to the case of $ L^p$-regularity theory.

\begin{proposition}\label{Corollary:BSE:HighReg:p}
	Let $\Om$ be of class $C^2$, let $p\in[2,\infty)$ and consider a pair $(f,g)\in\mathcal{V}^{-1}_{L}\cap\mathcal{L}^p$. Further, assume that $(\phi,\psi)\in\mathcal{W}^{2,4}$ such that $(\Grad\phi\cdot\Grad\mathcal{S}_{L}^\Om[\phi,\psi](f,g),\Gradg\psi\cdot\Gradg\mathcal{S}_{L}^\Ga[\phi,\psi](f,g))\in\mathcal{L}^p$. Then there exists a constant $C > 0$ such that
	\begin{align}\label{Est:BSE:W2p}
        \begin{split}
    		&\norm{\mathcal{S}_{L}[\phi,\psi](f,g)}_{\mathcal{W}^{2,p}} \\
            &\quad\leq C\big(\norm{(f,g)}_{\mathcal{L}^p} + \norm{(\Grad\phi\cdot\Grad\mathcal{S}_{L}^\Om[\phi,\psi](f,g),\Gradg\psi\cdot\Gradg\mathcal{S}_{L}^\Ga[\phi,\psi](f,g))}_{\mathcal{L}^p}\big).
        \end{split}
	\end{align}
\end{proposition}

\begin{proof}
	We adapt the proof of \cite[Proposition~A.1]{Knopf2025} for the case of constant mobility functions. As we already know from Proposition~\ref{Proposition:BSE:HighReg:2} that $(u,v) \coloneqq \mathcal{S}_{L}[\phi,\psi](f,g)\in\mathcal{H}^2$, it holds that
	\begin{alignat}{2}\label{Poisson-Problem}
		-\Lap u &= \frac{1}{m_\Om(\phi)}\Big(f + \Grad m_\Om(\phi)\cdot\Grad u\Big) &&\qquad\text{a.e. in~}\Om, \\
		L\deln u &= \frac{1}{m_\Om(\phi)}(\beta v - u) &&\qquad\text{a.e. on~}\Ga, \label{Poisson-Problem:BC}
	\end{alignat}
	as well as
	\begin{align}\label{LaplceBeltrami-Problem}
		-\Lapg v = \frac{1}{m_\Ga(\psi)}\Big(g - \beta m_\Om(\phi)\deln u + \Gradg m_\Ga(\psi)\cdot\Gradg v\Big) \qquad\text{a.e. on~}\Ga.
	\end{align}
	As in the proof of Proposition~\ref{Proposition:BSE:HighReg:2}, we consider the cases $L = 0$, $L\in(0,\infty)$ and $L = \infty$ separately.

	\textit{Case $L = 0$.} Since the boundary $\Ga$ is a $(d-1)$-dimensional submanifold of $\R^d$, Sobolev's embedding theorem implies that
	\begin{align*}
		u\vert_\Ga = \beta v \in W^{t,p}(\Ga) \qquad\text{with~} t = \frac52 + \frac{d-1}{p} - \frac d2.
	\end{align*}
	By assumption $f, \Grad\phi\cdot\Grad u \in L^p(\Om)$, we may apply elliptic regularity theory for Poisson's equation with an inhomogeneous Dirichlet boundary condition (see, e.g., \cite[Theorem 3.2]{Brezzi1987} or \cite[Theorem A.2]{Colli2019a}) to deduce
	\begin{align*}
		u\in W^{s,p}(\Om) \quad\text{with}\quad s = \min\Big\{2,\frac52 + \frac dp - \frac d2\Big\} \geq 1 + \frac 2p
	\end{align*}
	together with the estimate
	\begin{align*}
		\norm{u}_{W^{s,p}(\Om)} &\leq C\big(\norm{f}_{L^p(\Om)} + \norm{\Grad m_\Om(\phi)\cdot\Grad u}_{L^p(\Om)} + \norm{v}_{W^{t,p}(\Ga)}\big) \\
		&\leq C\big(\norm{f}_{L^p(\Om)} + \norm{\Grad\phi\cdot\Grad u}_{L^p(\Om)} + \norm{v}_{H^2(\Ga)}\big) \\
		&\leq C\big(\norm{(f,g)}_{\mathcal{L}^p} + \norm{(\Grad\phi\cdot\Grad u,\Gradg\psi\cdot\Gradg v)}_{\mathcal{L}^p}\big).
	\end{align*}
	This readily yields
	\begin{align*}
		\Grad u\in W^{s-1,p}(\Om) \emb W^{\frac2p,p}(\Om).
	\end{align*}
	Then, since $\frac2p - \frac1p = \frac1p$ is positive and not an integer, the trace theorem implies that
	\begin{align*}
		\deln u\in W^{\frac1p,p}(\Ga)\emb L^p(\Ga).
	\end{align*}
	Consequently, the assumptions $g, \Gradg\psi\cdot\Gradg v\in L^p(\Ga)$ allow us to apply regularity theory for the Laplace--Beltrami equation (see, e.g., \cite[Lemma B.1]{Vitillaro2017}) which yields that $v\in W^{2,p}(\Ga)$ and the existence of a constant $C > 0$ such that
	\begin{align}\label{Est:BSE:W2p:v:0}
		\begin{split}
			\norm{v}_{W^{2,p}(\Ga)} &\leq C\big(\norm{g}_{L^p(\Ga)} + 					\norm{\deln u}_{L^p(\Ga)} + \norm{\Gradg m_\Ga(\psi)						\cdot\Gradg v}_{L^p(\Ga)}\big) \\
			&\leq C\big(\norm{(f,g)}_{\mathcal{L}^p} + 									\norm{(\Grad\phi\cdot\Grad u,\Gradg\psi\cdot\Gradg v)}						_{\mathcal{L}^p}\big).
		\end{split}
	\end{align}
	This, in turn, entails
	\begin{align*}
		u\vert_\Ga = \beta v \in W^{2,p}(\Ga).
	\end{align*}
	Therefore, by means of elliptic regularity theory for Poisson's equation with an inhomogeneous Dirichlet boundary condition (see, e.g., \cite[Theorem 3.2]{Brezzi1987} or \cite[Theorem A.2]{Colli2019a}), we find that $u\in W^{2,p}(\Om)$ with
	\begin{align}\label{Est:BSE:W2p:u:0}
		\begin{split}
			\norm{u}_{W^{2,p}(\Om)} &\leq C\big(\norm{f}_{L^p(\Om)} + 					\norm{\Grad m_\Om(\phi)\cdot\Grad u}_{L^p(\Om)} + \norm{v}					_{W^{2,p}(\Ga)}\big) \\
			&\leq C\big(\norm{(f,g)}_{\mathcal{L}^p} + 									\norm{(\Grad\phi\cdot\Grad u, \Gradg\psi\cdot\Gradg v)}						_{\mathcal{L}^p}\big).
		\end{split}
	\end{align}
	Combining \eqref{Est:BSE:W2p:v:0} and \eqref{Est:BSE:W2p:u:0}, we eventually obtain \eqref{Est:BSE:W2p}, which finishes the proof in the case $L = 0$.
	
	\textit{Case $L\in(0,\infty)$.} Since $u\in H^2(\Om)$, the trace theorem as well as Sobolev's embedding theorem yield that
	\begin{align*}
		u\in H^{\frac32}(\Ga)\emb W^{t,p}(\Ga) \quad\text{with}\quad t = 2 + \frac{d-1}{p} - \frac d2.
	\end{align*}
	Then, since $\phi\in W^{2,4}(\Om)$, the trace theorem and the Sobolev embedding theorem further imply that
	\begin{align*}
		\Grad\phi\in W^{1,4}(\Om)\emb W^{\frac34,4}(\Ga) \emb L^\infty(\Ga).
	\end{align*}
	Consequently, recalling that the mobility function $m_\Om\in C^1([-1,1])$ satisfies \eqref{Ass:Mobility:Bound}, we readily deduce that $\frac{1}{m_\Om(\phi)} \in W^{1,\infty}(\Ga)$. As $v\in H^2(\Ga)\emb W^{t,p}(\Ga)$, we find that
	\begin{align*}
		\deln u = \frac{1}{Lm_\Om(\phi)}(\beta v - u)\in W^{t,p}(\Ga).
	\end{align*}
	Then, since $g,\Gradg\psi\cdot\Gradg v\in L^p(\Ga)$, we infer from regularity theory for the Laplace--Beltrami equation (see, e.g., \cite[Lemma~B.1]{Vitillaro2017}) that
	\begin{align}\label{Est:BSE:W2p:v:L}
		\begin{split}
			\norm{v}_{W^{2,p}(\Ga)} &\leq C\big(\norm{g}_{L^p(\Ga)} + 					\norm{\Gradg m_\Ga(\psi)\cdot\Gradg v}_{L^p(\Ga)} + 						\norm{\deln u}_{W^{t,p}(\Ga)}\big) \\
			&\leq C\big(\norm{g}_{L^p(\Ga)} + \norm{\Gradg\psi\cdot\Gradg 			v}_{L^p(\Ga)} + \norm{(u,v)}_{\mathcal{H}^2}\big) \\
			&\leq C\big(\norm{(f,g)}_{\mathcal{L}^p} + 									\norm{(\Grad\phi\cdot\Grad u,\Gradg\psi\cdot\Gradg v)}						_{\mathcal{L}^p}\big).
		\end{split}
	\end{align}
	Moreover, exploiting again $f,\Grad\phi\cdot\Grad u\in L^p(\Om)$, elliptic regularity theory for Poisson's equation with an inhomogeneous Neumann boundary condition (see, e.g., \cite[Theorem 3.2]{Brezzi1987} or \cite[Theorem A.2]{Colli2019a}) provides that
	\begin{align*}
		u\in W^{s,p}(\Om) \quad\text{with}\quad s = \min\Big\{2,3 + \frac dp - \frac d2\Big\} \geq 1 + \frac2p
	\end{align*}
	together with the estimate
	\begin{align*}
		\norm{u}_{W^{s,p}} &\leq C\big( \norm{f}_{L^p(\Om)} + 						\norm{\Grad m_\Om(\phi)\cdot\Grad u}_{L^p(\Om)} + \norm{\deln 			u}_{W^{t,p}(\Ga)}\big) \\
		&\leq C\big(\norm{(f,g)}_{\mathcal{L}^p} + \norm{(\Grad\phi\cdot\Grad u,\Gradg\psi\cdot\Gradg v)}_{\mathcal{L}^p}\big).
	\end{align*}
	Then, since $1+ \frac2p - \frac1p = 1 + \frac1p$ is positive and not an integer, the trace theorem shows that
	\begin{align*}
		u\vert_\Ga\in W^{1+\frac1p,p}(\Ga),
	\end{align*}
	which entails
	\begin{align*}
		\deln u = \frac{1}{Lm_\Om(\phi)}(\beta v - u)\in W^{1+\frac1p,p}(\Ga).
	\end{align*}
	Finally, using once more elliptic regularity theory for Poisson's equation with an inhomogeneous Neumann boundary condition (see, e.g., \cite[Theorem 3.2]{Brezzi1987} or \cite[Theorem A.2]{Colli2019a}), we conclude that $u\in W^{2,p}(\Om)$ along with
	\begin{align}\label{Est:BSE:W2p:u:L}
		\begin{split}
			\norm{u}_{W^{2,p}(\Om)} &\leq C\big(\norm{f}_{L^p(\Om)} + \norm{\Grad m_\Om(\phi)\cdot\Grad u}_{L^p(\Om)} + \norm{\deln u}_{W^{1+\frac1p,p}(\Ga)}\big) \\
			&\leq C\big(\norm{(f,g)}_{\mathcal{L}^p} + \norm{(\Grad\phi\cdot\Grad u,\Gradg\psi\cdot\Gradg v)}_{\mathcal{L}^p}\big).
		\end{split}
	\end{align}
	Combining the estimates \eqref{Est:BSE:W2p:v:L} and \eqref{Est:BSE:W2p:u:L} immediately yields \eqref{Est:BSE:W2p}.
	
	\textit{Case $L = \infty$.} In this case, both the Poisson--Neumann problem \eqref{Poisson-Problem}-\eqref{Poisson-Problem:BC} and the Laplace--Beltrami equation \eqref{LaplceBeltrami-Problem} are completely decoupled. Hence, we can apply elliptic regularity theory for Poisson's equation with a homogeneous Neumann boundary condition (see, e.g., \cite[Theorem 3.2]{Brezzi1987} or \cite[Theorem A.2]{Colli2019a}) and regularity theory for the Laplace--Beltrami equation (see, e.g., \cite[Lemma~B.1]{Vitillaro2017}), respectively, and directly reach the conclusion with the respective estimates, which finishes the proof.
\end{proof}

\textbf{Further properties of $\mathcal{S}_{L}$ and $\mathcal{S}_{L}[\phi,\psi]$ in two dimensions}. In the last part of the section, we provide two useful estimates for $\mathcal{S}_{L}[\phi,\psi]$ for $d = 2$ which are based on the estimates \eqref{Est:Sol:G:H^2:thm} and \eqref{Est:BSE:W2p}. These will be crucial for the mathematical analysis of \eqref{EQ:SYSTEM}. To this end, let $\Omega\subset\R^2$ be of class $C^2$ and $(\phi,\psi)\in\mathcal{W}^{2,4}$ with $\abs{\phi}\leq 1$ a.e. in $\Om$ and $\abs{\psi}\leq 1$ a.e. on $\Ga$. We further assume that the mobility functions $m_\Om,m_\Ga\in C^1([-1,1])$ satisfy \ref{Ass:Mobility}.

Let $(f,g)\in\mathcal{V}^{-1}_{L}\cap\mathcal{L}^2$. Then, using \eqref{Prelim:Est:Inteprol} and \eqref{Est:Sol:G:H^2} , we find
\begin{align*}
	\norm{\mathcal{S}_{L}[\phi,\psi](f,g)}_{\mathcal{H}^2} &\leq C\big(\norm{(\Grad\phi\cdot\Grad\mathcal{S}_{L}^\Om[\phi,\psi](f,g),\Gradg\psi\cdot\Gradg\mathcal{S}_{L}^\Ga[\phi,\psi](f,g))}_{\mathcal{L}^2} + \norm{(f,g)}_{\mathcal{L}^2}\big) \\
	&\leq C\big(\norm{(\Grad\phi,\Gradg\psi)}_{\mathcal{L}^4}\norm{(\Grad\mathcal{S}_{L}^\Om[\phi,\psi](f,g),\Gradg\mathcal{S}_{L}^\Ga[\phi,\psi](f,g))}_{\mathcal{L}^4} + \norm{(f,g)}_{\mathcal{L}^2} \big) \\
	&\leq C\big(\norm{(\Grad\phi,\Gradg\psi)}_{\mathcal{L}^2}^{\frac12}\norm{(\phi,\psi)}_{\mathcal{H}^2}^{\frac12}\norm{\mathcal{S}_{L}[\phi,\psi](f,g)}_{L}^{\frac12}\norm{\mathcal{S}_{L}[\phi,\psi](f,g)}_{\mathcal{H}^2}^{\frac12} \\
	&\qquad + \norm{(f,g)}_{\mathcal{L}^2}\big).
\end{align*}
Consequently, we infer from Young's inequality that
\begin{align}\label{Est:Sol:G:H^2}
	\norm{\mathcal{S}_{L}[\phi,\psi](f,g)}_{\mathcal{H}^2} \leq C\big(\norm{(\Grad\phi,\Gradg\psi)}_{\mathcal{L}^2}\norm{(\phi,\psi)}_{\mathcal{H}^2}\norm{\mathcal{S}_{L}[\phi,\psi](f,g)}_{L} + \norm{(f,g)}_{\mathcal{L}^2}\big).
\end{align}

Furthermore, if $(f,g)\in\mathcal{V}^{-1}_{L}\cap\mathcal{L}^4$, we can choose $p = 4$ in \eqref{Est:BSE:W2p}, and exploiting \eqref{Prelim:Est:Inteprol}, we find
\begin{align}\label{Est:Sol:G:W^24}
	&\norm{\mathcal{S}_{L}[\phi,\psi](f,g)}_{\mathcal{W}^{2,4}} \nonumber \\
    &\quad\leq C\big(\norm{(\Grad\phi,\Gradg\psi)}_{\mathcal{L}^8}\norm{(\Grad\mathcal{S}_{L}^\Om[\phi,\psi](f,g),\Gradg\mathcal{S}_{L}^\Ga[\phi,\psi](f,g))}_{\mathcal{L}^8} + \norm{(f,g)}_{\mathcal{L}^4}\big) \\
	&\quad\leq C\big(\norm{(\Grad\phi,\Gradg\psi)}_{\mathcal{L}^2}^{\frac14}\norm{(\phi,\psi)}_{\mathcal{H}^2}^{\frac34}\norm{\mathcal{S}_{L}[\phi,\psi](f,g)}_{L}^{\frac14}\norm{\mathcal{S}_{L}[\phi,\psi](f,g)}_{\mathcal{H}^2}^{\frac34} + \norm{(f,g)}_{\mathcal{L}^4}\big). \nonumber
\end{align}

\medskip
\section{Existence and Uniqueness of Weak Solutions}
\label{Section:Uniqueness}

In this section, we present the proof of Theorem~\ref{Theorem} regarding the well-posedness of weak solutions of system~\eqref{EQ:SYSTEM}.
\begin{proof}[Proof of Theorem~\ref{Theorem}]
The proof is divided into several steps. We start with some basic estimates on a weak solution.

\noindent
\textbf{Properties of weak solutions.} Let $(\phi,\psi,\mu,\theta)$ be a global weak solution as in Definition~\ref{DEF:SING:WS} given by, e.g., \cite[Theorem~3.4]{Knopf2025} or \cite[Theorem~3.2]{Giorgini2025}.  First, we deduce from \eqref{PROP:CONF} and the energy inequality \eqref{WEDL:SING} that
\begin{align}\label{Est:PP:H1:Linfty}
	\sup_{t\geq 0} \norm{(\phi(t),\psi(t))}_{\mathcal{H}^1} \leq C
\end{align}
as well as
\begin{align}\label{Est:MT:LB:L2}
	\int_0^\infty \norm{(\mu(s),\theta(s))}_{L}^2 \ds \leq C,
\end{align}
for some constant depending on $E(\phi_0,\psi_0)$, $\abs{\Om}$ and $\abs{\Ga}$. Here, we have additionally used \eqref{NormEquivalence:1}. The latter implies that $\norm{(\mu,\theta)}_{L}\in L^2(0,\infty)$. Then, by definition of $\mathcal{S}_{L}[\phi,\psi](\delt\phi,\delt\psi)$, it is clear that
\begin{alignat}{2}\label{Id:M:SolOp:Mean}
    \mu &= -\mathcal{S}_{L}^\Om[\phi,\psi](\delt\phi,\delt\psi) + \beta\mean{\mu}{\theta} &&\qquad\text{a.e.~in~}\Om\times(0,\infty), \\
    \theta &= -\mathcal{S}_{L}^\Ga[\phi,\psi](\delt\phi,\delt\psi) + \mean{\mu}{\theta} &&\qquad\text{a.e.~on~} \label{Id:T:SolOp:Mean}\Ga\times(0,\infty).
\end{alignat}
Thus, exploiting \eqref{Id:M:SolOp:Mean}-\eqref{Id:T:SolOp:Mean}, we obtain
\begin{align}\label{Est:Sol:Delt:pp:Lb}
    \norm{\mathcal{S}_{L}[\phi,\psi](\delt\phi,\delt\psi)}_{L} \leq C\norm{(\mu,\theta)}_{L},
\end{align}
from which we deduce with \eqref{Est:MT:LB:L2} and the equivalence of the corresponding norms on $\mathcal{V}^{-1}_{L}$ that
\begin{align}\label{EST:DELT:PP:H1Lb':L2}
	\int_0^\infty \norm{(\delt\phi,\delt\psi)}_{(\mathcal{H}^1_L)^\prime}^2\ds \leq C.
\end{align}
This proves $(\delt\phi,\delt\psi)\in L^2(0,\infty;(\mathcal{H}^1_L)^\prime)$.
Next, analyzing the proof of \cite[Theorem~3.4]{Knopf2025} reveals the estimate
\begin{align}\label{Est:MEAN:MT:DELN}
	\abs{\mean{\mu}{\theta}} \leq C\big(1 + \norm{(\mu,\theta)}_{L}\big),
\end{align}
from which we immediately deduce, after another application of the bulk-surface Poincar\'{e} inequality that
\begin{align}\label{Est:MT:H1:a.e.}
	\norm{(\mu,\theta)}_{\mathcal{H}^1} \leq C\big(1 + \norm{(\mu,\theta)}_{L}\big).
\end{align}
Hence, $(\mu,\theta)\in L^2_{\mathrm{uloc}}([0,\infty);\mathcal{H}^1)$.
The inequality \eqref{Est:MEAN:MT:DELN} is based on the Miranville--Zelik inequality (see \cite[Appendix A.1]{Miranville2004} or \cite[p.~908]{Gilardi2009}). Then, noticing that
\begin{alignat*}{2}
	-\Lap\phi + F_1^\prime(\phi) &= \mu^\ast &&\qquad\text{a.e.~in~}\Om, \\
	-\Lapg\psi + G_1^\prime(\psi) + \alpha\deln\phi &= \theta^\ast &&\qquad\text{a.e.~on~}\Ga, \\
	K\deln\phi &= \alpha\psi - \phi &&\qquad\text{a.e.~on~}\Ga,
\end{alignat*}
almost everywhere in $(0,\infty)$, where $(\mu^\ast,\theta^\ast) = (\mu - F_2^\prime(\phi),\theta - G_2^\prime(\psi)) \in\mathcal{H}^1$, an application of \cite[Proposition~6.5]{Giorgini2025} together with the Lipschitz continuity of $F_2^\prime$ and $G_2^\prime$, respectively, as well as \eqref{Est:PP:H1:Linfty} yields that
\begin{align}\label{Est:PP:Pot:L^p}
	\norm{(\phi,\psi)}_{\mathcal{W}^{2,p}} + 				\norm{(F_1^\prime(\phi),G_1^\prime(\psi))}_{\mathcal{L}^p} \leq C\big(1 + \norm{(\mu,\theta)}_{\mathcal{H}^1}		\big)
\end{align}
almost everywhere in $(0,\infty)$ for all $2\leq p < \infty$. 
Lastly, noting on $K\in(0,\infty)$ and applying \cite[Corollary~5.4]{Giorgini2025}, we find that
\begin{align*}
	\norm{(-\Lap\phi,-\Lapg\psi + \alpha\deln\phi)}_{\mathcal{L}^2}^2 \leq C\big(1 + \norm{(\mu + F_2^\prime(\phi),\theta + G_2^\prime(\psi))}_{\mathcal{H}^1}) \leq C\big(1 + \norm{(\mu,\theta)}_{L}\big).
\end{align*}
Thus, by \eqref{Est:MT:LB:L2} and elliptic regularity theory for systems with bulk-surface coupling (see, e.g., \cite[Theorem~3.3]{Knopf2021}), we have
\begin{align*}
	\sup_{t\geq 0}\int_t^{t+1}\norm{(\phi(s),\psi(s))}_{\mathcal{H}^2}^4\ds \leq C,
\end{align*}
which provides $(\phi,\psi)\in L^4_{\mathrm{uloc}}([0,\infty);\mathcal{H}^2)$.

\noindent
\textbf{Continuous dependence estimate and uniqueness of weak solutions.}
In what follows, we restrict ourselves to the case $L\in[0,\infty)$. This case distinction is needed due to a different mass conservation in the cases $L\in[0,\infty)$ and $L = \infty$, see \eqref{MCL:SING}. The case $L = \infty$ can be handled more easily with the obvious modifications.
To this end, let $(\phi_0^1,\psi_0^1)$ and $(\phi_0^2,\psi_0^2)$ be two admissible pairs of initial data satisfying \eqref{cond:init:int}-\eqref{cond:init:mean:inf}, and consider two weak solutions $(\phi_1,\psi_1,\mu_1,\theta_1)$ and $(\phi_2,\psi_2,\mu_2,\theta_2)$ originating from $(\phi_1^0,\psi_1^0)$ and $(\phi_0^2,\psi_0^2)$, respectively. Then, defining $(\Phi,\Psi) = (\phi_1 - \phi_2, \psi_1 - \psi_2)$, we have
\begin{align}\label{EQ:DELT:PHIPSI}
    \begin{split}
        &\bigang{(\delt\Phi,\delt\Psi)}{(\zeta,\xi)}_{\mathcal{H}^1_L} \\
        &\quad= - \intO m_\Om(\phi_1)\Grad(\mu_1 - \mu_2)\cdot\Grad\zeta\dx - \intG m_\Ga(\psi_1)\Gradg(\theta_1 - \theta_2)\cdot\Gradg\xi\dG \\
        &\qquad - \chi(L) \intG (\beta(\theta_1 - \theta_2) - (\mu_1 - \mu_2))(\beta\xi - \zeta)\dG \\
        &\qquad - \intO \big(m_\Om(\phi_1) - m_\Ga(\psi_1)\big)\Grad\mu_2\cdot\Grad\zeta\dx - \intG \big(m_\Ga(\psi_1) - m_\Ga(\psi_2)\big)\Gradg\theta_2\cdot\Gradg\xi\dG
    \end{split}
\end{align}
a.e. on $(0,\infty)$ for all $(\zeta,\xi)\in\mathcal{H}^1_L$, as well as
\begin{alignat}{2}
	-\Lap\Phi + F^\prime(\phi_1) - F^\prime(\phi_2) &= \mu_1 - \mu_2 &&\qquad\text{a.e.~in~}\Om\times(0,\infty), \label{EQ:PHI}\\
	-\Lapg\Psi + G^\prime(\psi_1) - G^\prime(\psi_2) + \alpha\deln\Phi &= \theta_1 - \theta_2 &&\qquad\text{a.e.~in~}\Ga\times(0,\infty), \label{EQ:PSI} \\
    K\deln\Phi &= \alpha\Psi - \Phi &&\qquad\text{a.e.~in~}\Ga\times(0,\infty). \label{EQ:BC:PHI}
\end{alignat}
Now, we multiply \eqref{EQ:PHI} with $\Phi$ and \eqref{EQ:PSI} with $\Psi$, integrate over $\Om$ and $\Ga$, respectively, and perform integration by parts. Adding the resulting equations leads to
\begin{align}\label{pre:comp}
    \begin{split}
    	&\norm{(\Phi,\Psi)}_{K}^2 + \intO \big(F_1^\prime(\phi_1) - F_1^\prime(\phi_2)\big)\Phi\dx + \intG \big(G_1^\prime(\psi_1) - G_1^\prime(\psi_2)\big)\Psi\dG \\
    	&\qquad - \intO \big(\mu_1 - \mu_2)\Phi\dx - \intG \big(\theta_1 - \theta_2\big)\Psi\dG \\
    	&\quad = \intO \big(F_2^\prime(\phi_1) - F_2^\prime(\phi_2)\big)\Phi\dx + \intG \big(G_2^\prime(\psi_1) - G_2^\prime(\psi_2)\big)\Psi\dG
    \end{split}
\end{align}
almost everywhere on $(0,\infty)$. Next, we want to rewrite the integrals involving the chemical potentials in terms of the solution operator $\mathcal{S}_{L}[\phi_1,\psi_1](\delt\Phi,\delt\Psi)$. To this end, to make the following computations more readable, we use the abbreviation $\mathcal{S}_j(f,g) = \mathcal{S}_{L}[\phi_j,\psi_j](f,g)$ for $j=1,2$ as well as $(\cdot,\cdot)_{L,j} = (\cdot,\cdot)_{L,[\phi_j,\psi_j]}$ and find
\begin{align}\label{comp:mt:1-2}
        &- \intO \big(\mu_1 - \mu_2)\Phi\dx - \intG \big(\theta_1 - \theta_2\big)\Psi\dG \nonumber \\
        &\quad = -\big(\mathcal{S}_1(\Phi,\Psi),(\mu_1 - \mu_2,\theta_1 - \theta_2))_{L,1} \nonumber \\
        &\quad = \bigang{(\delt\Phi,\delt\Psi)}{\mathcal{S}_1(\Phi,\Psi)}_{\mathcal{H}^1_L} - \intO \big(m_\Om(\phi_1) - m_\Om(\phi_2)\big)\Grad\mu_2\cdot\Grad\mathcal{S}^\Om_1(\Phi,\Psi)\dx \nonumber \\
        &\qquad - \intG \big(m_\Ga(\psi_1) - m_\Ga(\psi_2)\big)\Gradg\theta_2\cdot\Gradg\mathcal{S}^\Ga_1(\Phi,\Psi)\dG \nonumber \\
        &\quad = \big(\mathcal{S}_1(\delt\Phi,\delt\Psi),\mathcal{S}_1(\Phi,\Psi)\big)_{L,1} - \intO \big(m_\Om(\phi_1) - m_\Om(\phi_2)\big)\Grad\mu_2\cdot\Grad\mathcal{S}^\Om_1(\Phi,\Psi)\dx \\
        &\qquad - \intG \big(m_\Ga(\psi_1) - m_\Ga(\psi_2)\big)\Gradg\theta_2\cdot\Gradg\mathcal{S}^\Ga_1(\Phi,\Psi)\dG \nonumber \\
        &\quad =\big((\Phi,\Psi),\mathcal{S}_1(\delt\Phi,\delt\Psi)\big)_{\mathcal{L}^2} - \intO \big(m_\Om(\phi_1) - m_\Om(\phi_2)\big)\Grad\mu_2\cdot\Grad\mathcal{S}^\Om_1(\Phi,\Psi)\dx \nonumber \\
        &\qquad - \intG \big(m_\Ga(\psi_1) - m_\Ga(\psi_2)\big)\Gradg\theta_2\cdot\Gradg\mathcal{S}^\Ga_1(\Phi,\Psi)\dG \nonumber 
\end{align}
a.e. on $(0,\infty)$, using, in this order, the weak formulations for $\mathcal{S}_1(\Phi,\Psi)$, for $(\delt\Phi,\delt\Psi)$ (see \eqref{EQ:DELT:PHIPSI}), for $\mathcal{S}_1(\delt\Phi,\delt\Psi)$ and again for $\mathcal{S}_1(\Phi,\Psi)$. Plugging the identity \eqref{comp:mt:1-2} back into \eqref{pre:comp} yields
\begin{align*}
	&\norm{(\Phi,\Psi)}_{K}^2 + \intO \big(F_1^\prime(\phi_1) - F_1^\prime(\phi_2)\big)\Phi\dx + \intG \big(G_1^\prime(\psi_1) - G_1^\prime(\psi_2)\big)\Psi\dG \\
	&\qquad  + \intO \mathcal{S}_1(\delt\Phi,\delt\Psi)\Phi\dx + \intG \mathcal{S}_1(\delt\Phi,\delt\Psi)\Psi\dG \\
    &\qquad - \intO \big(m_\Om(\phi_1) - m_\Om(\phi_2)\big)\Grad\mu_2\cdot\Grad\mathcal{S}^\Om_1(\Phi,\Psi)\dx - \intG \big(m_\Ga(\psi_1) - m_\Ga(\psi_2)\big)\Gradg\theta_2\cdot\Gradg\mathcal{S}^\Ga_1(\Phi,\Psi)\dG \\
	&\quad = \intO \big(F_2^\prime(\phi_1) - F_2^\prime(\phi_2)\big)\Phi\dx + \intG \big(G_2^\prime(\psi_1) - G_2^\prime(\psi_2)\big)\Psi\dG.
\end{align*}
Now, we claim that
\begin{align}\label{ChainRuleFormula}
	\begin{split}
	&\intO \mathcal{S}_1^\Om(\delt\Phi,\delt\Psi)\Phi\dx + \intG \mathcal{S}_1^\Ga(\delt\Phi,\delt\Psi)\Psi\dG  \\
	&\quad = \ddt\frac12 \norm{(\Phi,\Psi)}_{L,[\phi_1,\psi_1],\ast}^2 \\
	&\qquad + \frac12 \Big(\mathcal{S}_1(\delt\phi_1,\delt\psi_1),\big(m_\Om^\prime(\phi_1)\abs{\Grad\mathcal{S}_1^\Om(\Phi,\Psi)}^2,m_\Ga^\prime(\psi_1)\abs{\Gradg\mathcal{S}_1^\Ga(\Phi,\Psi)}^2\big)\Big)_{L} \\
    &\quad = \ddt\frac12 \norm{(\Phi,\Psi)}_{L,[\phi_1,\psi_1],\ast}^2 \\
    &\qquad + \frac12\intO \Grad\mathcal{S}_1^\Om(\delt\phi_1,\delt\psi_1)\cdot m_\Om^{\prime\prime}(\phi_1)\Grad\phi_1\abs{\Grad\mathcal{S}_1^\Om(\Phi,\Psi)}^2\dx \\
    &\qquad + \frac12 \intG\Gradg\mathcal{S}_1^\Ga(\delt\phi_1,\delt\psi_1)\cdot m_\Ga^{\prime\prime}(\psi_1)\Gradg\psi_1\abs{\Gradg\mathcal{S}_1^\Ga(\Phi,\Psi)}^2\dG \\
    &\qquad + \intO \Grad\mathcal{S}_1^\Om(\delt\phi_1,\delt\psi_1)\cdot m_\Om^\prime(\phi_1) D^2\mathcal{S}_1^\Om(\Phi,\Psi)\Grad\mathcal{S}_1^\Om(\Phi,\Psi)\dx \\ 
    &\qquad + \intG \Gradg\mathcal{S}_1^\Ga(\delt\phi_1,\delt\psi_1)\cdot m_\Ga^\prime(\psi_1)D^2_\Ga\mathcal{S}_1^\Ga(\Phi,\Psi)\Gradg\mathcal{S}_1^\Ga(\Phi,\Psi)\dG \\ 
    &\qquad + \frac12\chi(L)\intG \big(\beta\mathcal{S}_1^\Ga(\delt\phi_1,\delt\psi_1) - \mathcal{S}_1^\Om(\delt\phi_1,\delt\psi_1)\big)\\
    &\qquad\qquad\qquad\qquad\times\big(\beta m_\Ga^\prime(\psi_1)\abs{\Gradg\mathcal{S}_1^\Ga(\Phi,\Psi)}^2 - m_\Om^\prime(\phi_1)\abs{\Gradg\mathcal{S}_1^\Om(\Phi,\Psi)}^2\big)\dG.
	\end{split}
\end{align}
Here, $D^2 f$ and $D^2_\Ga g$ denote the Hessian and the surface Hessian of $f$ and $g$, respectively. Then, having \eqref{ChainRuleFormula} at hand, we exploit the monotonicity of $F_1^\prime$ and $G_1^\prime$ as well as the Lipschitz continuity of $F_2^\prime$ and $G_2^\prime$, respectively, which leads to 
\begin{align}\label{DiffIneq}
		\ddt\frac12\norm{(\Phi,\Psi)}_{L,[\phi_1,\psi_1],\ast}^2 + \norm{(\Phi,\Psi)}_{K}^2 \leq C\norm{(\Phi,\Psi)}_{\mathcal{L}^2}^2 + I_1 + I_2,
\end{align}
where
\begin{align*}
	I_1 &=  -\frac12 \Big(\mathcal{S}_1(\delt\phi_1,\delt\psi_1),\big(m_\Om^\prime(\phi_1)\abs{\Grad\mathcal{S}_1^\Om(\Phi,\Psi)}^2,m_\Ga^\prime(\psi_1)\abs{\Gradg\mathcal{S}_1^\Ga(\Phi,\Psi)}^2\big)\Big)_{L},
\end{align*}
and
\begin{align*}
	I_2 =  \intO \big(m_\Om(\phi_1) - m_\Om(\phi_2)\big)\Grad\mu_2\cdot\Grad\mathcal{S}^\Om_1(\Phi,\Psi)\dx + \intG \big(m_\Ga(\psi_1) - m_\Ga(\psi_2)\big)\Gradg\theta_2\cdot\Gradg\mathcal{S}^\Ga_1(\Phi,\Psi)\dG.
\end{align*}

The rest of the proof now consists of justifying the chain-rule formula \eqref{ChainRuleFormula} as well as estimating the nonlinear terms $I_1$ and $I_2$.

\textbf{Proof of \eqref{ChainRuleFormula}.} Let $\rho\in C_c^\infty(\R)$ be non-negative with $\mathrm{supp}\,\rho\subset(0,1)$ and $\norm{\rho}_{L^1(\R)} = 1$. For $k\in\N$, we set
\begin{align*}
	\rho_k(s) \coloneqq k\rho(ks), \qquad s\in\R.
\end{align*}
Then, for any Banach space $X$ and any function $f\in L^p(-1,T;X)$ with $2\leq p < \infty$, we define
\begin{align}\label{Def:f_k}
	f_k(t) \coloneqq (\rho_k\ast f)(t) = \int_{t-\frac1k}^t \rho_k(t-s)f(s)\ds
\end{align}
for all $t\in[0,T]$ and all $k\in\N$. By this construction, we have $f_k\in C^\infty([0,T];X)$ with $f_k\rightarrow f$ strongly in $L^p(0,T;X)$ as $k\rightarrow\infty$. 

Now, let $T > 0$, $k\in\N$ and $p = 4$. We choose $X = H^2(\Om)$ to define $\phi_1^k$ and $X = H^2(\Ga)$ to define $\psi_1^k$ as described above. By this construction, we then have $\delt\phi_1^k = (\delt\phi_1)^k$ and $\delt\Grad\phi_1^k = \Grad\delt\phi_1^k$ a.e.~in $Q$ as well as $\delt\psi_1^k = (\delt\psi_1)^k$ and $\delt\Gradg\psi_1^k = \Gradg\delt\psi_1^k$ a.e.~on $\Sigma$ for all $k\in\N$. Moreover, as $k\rightarrow\infty$,
\begin{alignat}{2}
	\phi_1^k &\rightarrow\phi_1 &&\qquad\text{strongly in~}L^4(0,T;H^2(\Om)), \label{Conv:Phi:H^2:k} \\
	\psi_1^k &\rightarrow\psi_1 &&\qquad\text{strongly in~}L^4(0,T;H^2(\Ga)), \label{Conv:Psi:H^2:k} \\
	(\delt\phi_1^k,\delt\psi_1^k) &\rightarrow(\delt\phi_1,\delt\psi_1^k) &&\qquad\text{strongly in~}L^2(0,T;(\mathcal{H}^1_L)^\prime). \label{Conv:delt:pp:H^1:k}
\end{alignat}
In particular, along a non-relabeled subsequence, as $k\rightarrow\infty$,
\begin{alignat}{3}
	\phi_1^k &\rightarrow \phi_1, \quad \Grad\phi_1^k&&\rightarrow\Grad\phi_1 &&\qquad\text{a.e. in~} \Om\times(0,T), \label{Conv:Phi:a.e.} \\
	\psi_1^k &\rightarrow \psi_1, \quad \Gradg\psi_1^k &&\rightarrow \Gradg\psi_1 &&\qquad\text{a.e. on~} \Ga\times(0,T). \label{Conv:Psi:a.e.}
\end{alignat}
Additionally, we find that
\begin{align}\label{Est:PP:k:LinftyH^1}
    \begin{split}
    	\norm{\phi_1^k}_{L^\infty(0,T;H^1(\Om))} &\leq \norm{\phi_1}_{L^\infty(0,T;H^1(\Om))},
    	\\
    	 \norm{\psi_1^k}_{L^\infty(0,T;H^1(\Ga))} &\leq \norm{\psi_1}_{L^\infty(0,T;H^1(\Ga))},
    \end{split}
\end{align}
as well as
\begin{align}\label{Est:PP:k:infty}
	\abs{\phi_k} \leq 1 \quad\text{a.e.~in~}Q, \qquad \abs{\psi_k} \leq 1 \quad\text{a.e.~on~}\Sigma
\end{align}
for all $k\in\N$. We further point out that in light of the convergences \eqref{Conv:Phi:H^2:k}-\eqref{Conv:Psi:H^2:k} we infer the existence of $k_\ast\in\N$ such that
\begin{align}\label{Est:PP:k:L^4H^2}
	\norm{(\phi_1^k,\psi_1^k)}_{L^4(0,T;\mathcal{H}^2)} \leq 1 + \norm{(\phi_1,\psi_1)}_{L^4(0,T;\mathcal{H}^2)} \qquad\text{for all~}k\geq k_\ast.
\end{align}
In the following, we assume that $k\in\N$ satisfies $k\geq k_\ast$. Furthermore, we denote by the letter $C$ a generic positive constant whose value may change from line to line, but is independent of the parameter $k\in\N$.
We now define, for $k\in\N$, the sequence
\begin{align*}
	\mathcal{S}_k(\Phi,\Psi) \coloneqq \mathcal{S}_{L}[\phi_k,\psi_k](\Phi,\Psi).
\end{align*}
It readily follows from \eqref{Est:EBS:Apriori} and \eqref{Est:PP:k:infty} that
\begin{align}\label{Est:G:k:LinftyH^1}
	\norm{\mathcal{S}_k(\Phi,\Psi)}_{L^\infty(0,T;\mathcal{H}^1)} \leq C.
\end{align}
Next, in view of \eqref{Est:PP:k:LinftyH^1} and \eqref{Est:G:k:LinftyH^1}, an application of \eqref{Est:Sol:G:H^2} entails that
\begin{align*}
	\norm{\mathcal{S}_k(\Phi,\Psi)}_{\mathcal{H}^2} \leq C\big(1 + \norm{(\phi_1^k,\psi_1^k)}_{\mathcal{H}^2}\big),
\end{align*}
which implies that
\begin{align}\label{Est:Gk:PP:L4H2}
	\int_0^T \norm{\mathcal{S}_k(\Phi,\Psi)}_{\mathcal{H}^2}^4\ds \leq CT + C\int_0^T \norm{(\phi_1^k,\psi_1^k)}_{\mathcal{H}^2}^4\ds \leq C,
\end{align}
due to \eqref{Est:PP:k:L^4H^2}. Furthermore, employing the estimate \eqref{Est:Sol:G:W^24} in combination with \eqref{Est:PP:k:LinftyH^1} and \eqref{Est:G:k:LinftyH^1}, we deduce that
\begin{align*}
	\norm{\mathcal{S}_k(\Phi,\Psi)}_{\mathcal{W}^{2,4}} &\leq C\big(\norm{(\Grad\phi_1^k,\Gradg\psi_1^k)}_{\mathcal{L}^2}^{\frac14}\norm{(\phi_1^k,\psi_1^k)}_{\mathcal{H}^2}^{\frac34}\norm{\mathcal{S}_k(\Phi,\Psi)}_{L}^{\frac14}\norm{\mathcal{S}_k(\Phi,\Psi)}_{\mathcal{H}^2}^{\frac34} + \norm{(\Phi,\Psi)}_{\mathcal{L}^4}\big) \\
	&\leq C\big(1 + \norm{(\phi_1^k,\psi_1^k)}_{\mathcal{H}^2}^{\frac34}\norm{\mathcal{S}_k(\Phi,\Psi)}_{\mathcal{H}^2}^{\frac34}\big).
\end{align*}
Hence, integrating the foregoing inequality in time over $(0,T)$ and using Young's inequality together with \eqref{Est:PP:k:L^4H^2} and \eqref{Est:Gk:PP:L4H2}, we have
\begin{align}\label{Est:Gk:PP:L83W24}
	\int_0^T \norm{\mathcal{S}_k(\Phi,\Psi)}_{\mathcal{W}^{2,4}}^{\frac83}\ds \leq CT + C\int_0^T \norm{(\phi_1^k,\psi_1^k)}_{\mathcal{H}^2}^4 + \norm{\mathcal{S}_k(\Phi,\Psi)}_{\mathcal{H}^2}^4\ds \leq C.
\end{align}

Now, we study the convergence properties of the operator $\mathcal{S}_k(\Phi,\Psi)$ as $k\rightarrow\infty$. To this end, let $(f,g)\in L^2(0,T;\mathcal{V}^{-1}_{L})$. Then, by definition of $\mathcal{S}_1$ and $\mathcal{S}_k$, we know that
\begin{align*}
	&\intO m_\Om(\phi_1)\Grad\mathcal{S}_1^\Om(f,g)\cdot\Grad\zeta\dx + \intG m_\Ga(\psi_1)\Gradg\mathcal{S}_1^\Ga(f,g)\cdot\Gradg\xi\dG \\
	&\qquad + \chi(L) \intG (\beta\mathcal{S}_1^\Ga(f,g) - \mathcal{S}_1^\Om(f,g))(\beta\xi-\zeta)\dG \\
	&\quad = \intO m_\Om(\phi_1^k)\Grad\mathcal{S}_k^\Om(f,g)\cdot\Grad\zeta\dx + \intG m_\Ga(\psi_1^k)\Gradg\mathcal{S}_k^\Ga(f,g)\cdot\Gradg\xi\dG \\
	&\qquad + \chi(L)\intG(\beta\mathcal{S}_k^\Ga(f,g) - \mathcal{S}_k^\Ga(f,g))(\beta\xi-\zeta)\dG
\end{align*}
for all $(\zeta,\xi)\in\mathcal{H}^1_L$, which implies that
\begin{align}\label{WF:S1-Sk}
	&\intO m_\Om(\phi_1)\Grad\big(\mathcal{S}_1^\Om(f,g) - \mathcal{S}_k^\Om(f,g)\big)\cdot\Grad\zeta\dx + \intG m_\Ga(\psi_1)\Gradg\big(\mathcal{S}_1^\Ga(f,g) - \mathcal{S}_k^\Ga(f,g)\big)\cdot\Gradg\xi\dG \nonumber \\
	&\qquad + \chi(L) \intG \big(\beta(\mathcal{S}_1^\Ga(f,g) - \mathcal{S}_k^\Ga(f,g))\big)(\beta\xi - \zeta)\dG \\
	&\quad = \intO \big(m_\Om(\phi_1^k) - m_\Om(\phi_1)\big)\Grad\mathcal{S}_k^\Om(f,g)\cdot\Grad\zeta\dx + \intG \big(m_\Ga(\psi_1^k) - m_\Ga(\psi_1)\big)\Gradg\mathcal{S}_k^\Ga(f,g)\cdot\Gradg\xi\dG \nonumber
\end{align}
for all $(\zeta,\xi)\in\mathcal{H}^1_L$. Therefore, taking $(\zeta,\xi) = \mathcal{S}_1(f,g) - \mathcal{S}_k(f,g)\in\mathcal{H}^1_L$ as a test function in \eqref{WF:S1-Sk}, we find with \eqref{NormEquivalence:1} that
\begin{align}\label{Est:G1-Gk:Lb}
		&\int_0^T \norm{\mathcal{S}_1(f,g) - \mathcal{S}_k(f,g)}_{L}^2\ds \nonumber \\
		&\quad\leq \frac{1}{\min\big\{1,m^\ast\big\}}\int_0^T \norm{\big(m_\Om(\phi_1^k) - m_\Om(\phi_1)\big)\Grad\mathcal{S}_k^\Om(f,g)}_{L^2(\Om)}^2 \ds \nonumber \\
		&\qquad + \frac{1}{\min\big\{1,m^\ast\big\}}\int_0^T\norm{\big(m_\Ga(\psi_1^k) - m_\Ga(\psi_1)\big)\Gradg\mathcal{S}_k^\Ga(f,g)}_{L^2(\Ga)}^2 \ds \\
		&\quad\leq \frac{1}{\min\big\{1,m^\ast\big\}}\int_0^T \norm{m_\Om(\phi_1^k) - m_\Om(\phi_1)}_{L^\infty(\Om)}^2\norm{\Grad\mathcal{S}_k^\Om(f,g)}_{L^2(\Om)}^2\ds \nonumber \\
		&\qquad + \frac{1}{\min\big\{1,m^\ast\big\}}\int_0^T \norm{m_\Ga(\psi_1^k) - m_\Ga(\psi_1)}_{L^\infty(\Ga)}^2\norm{\Gradg\mathcal{S}_k^\Ga(f,g)}_{L^2(\Ga)}^2\ds. \nonumber 
\end{align}
We aim to pass to the limit $k\rightarrow\infty$ in \eqref{Est:G1-Gk:Lb} by applying the dominated convergence theorem. Toward this end, we first notice by \eqref{NormEquivalence} that
\begin{align*}
	&\norm{m_\Om(\phi_1^k) - m_\Om(\phi_1)}_{L^\infty(\Om)}^2\norm{\Grad\mathcal{S}_k^\Om(f,g)}_{L^2(\Om)}^2 + \norm{m_\Ga(\psi_1^k) - m_\Ga(\psi_1)}_{L^\infty(\Ga)}^2\norm{\Gradg\mathcal{S}_k^\Ga(f,g)}_{L^2(\Ga)}^2 \\&\quad\leq C\norm{\mathcal{S}(f,g)}_{L}^2\in L^1(0,T).
\end{align*}
On the other hand, since $m_\Om$ and $m_\Ga$ are Lipschitz continuous on $[-1,1]$, the Sobolev inequality shows that
\begin{align*}
	\norm{m_\Om(\phi_1^k) - m_\Om(\phi_1)}_{L^\infty(\Om)}^2 + \norm{m_\Ga(\psi_1^k) - m_\Ga(\psi_1)}_{L^\infty(\Ga)}^2 &\leq C\norm{\phi_1^k - \phi_1}_{L^\infty(\Om)}^2 + C\norm{\psi_1^k - \psi_1}_{L^\infty(\Ga)}^2 \\
	&\leq C\norm{\phi_1^k - \phi_1}_{H^2(\Om)}^2 + C\norm{\psi_1^k - \psi_1}_{H^2(\Ga)}^2
\end{align*}
almost everywhere on $(0,T)$. Thus, in light of \eqref{Conv:Phi:H^2:k}-\eqref{Conv:Psi:H^2:k}, we obtain that
\begin{align*}
	&\norm{m_\Om(\phi_1^k) - m_\Om(\phi_1)}_{L^\infty(\Om)}^2\norm{\Grad\mathcal{S}_k^\Om(f,g)}_{L^2(\Om)}^2 + \norm{m_\Ga(\psi_1^k) - m_\Ga(\psi_1)}_{L^\infty(\Ga)}^2\norm{\Gradg\mathcal{S}_k^\Ga(f,g)}_{L^2(\Ga)}^2 \\
	&\quad\leq C\norm{(\phi_1^k - \phi_1,\psi_1^k - \psi_1)}_{\mathcal{H}^2}^2\norm{\mathcal{S}(f,g)}_{L}^2 \rightarrow 0
\end{align*}
up to a subsequence $k\rightarrow\infty$ almost everywhere on $(0,T)$. Consequently, we can use \eqref{Est:G1-Gk:Lb} and the Lebesgue convergence theorem to conclude that
\begin{align*}
	\norm{\mathcal{S}_k(f,g) - \mathcal{S}_1(f,g)}_{L} \rightarrow 0 \qquad\text{strongly in~}L^2(0,T)
\end{align*}
up to a subsequence $k\rightarrow\infty$. In particular, for $(f,g) = (\delt\Phi,\delt\Psi)$ and $(f,g) = (\Phi,\Psi)$, we find that
\begin{align}\label{Conv:Gk:delt:PP:L2H1}
	\mathcal{S}_k(\delt\Phi,\delt\Psi)\rightarrow\mathcal{S}_1(\delt\Phi,\delt\Psi) \qquad\text{strongly in~}L^2(0,T;\mathcal{H}^1), 
\end{align}
as well as
\begin{align}\label{Conv:Gk:PP:L2H1}
	\mathcal{S}_k(\Phi,\Psi) \rightarrow\mathcal{S}_1(\Phi,\Psi) \qquad\text{strongly in~} L^2(0,T;\mathcal{H}^1),
\end{align}
as $k\rightarrow\infty$, respectively. Furthermore, in view of the uniform estimates \eqref{Est:Gk:PP:L4H2} and \eqref{Est:Gk:PP:L83W24}, we deduce that
\begin{align}\label{Conv:Gk:PP:L4H2}
	\mathcal{S}_k(\Phi,\Psi) \rightharpoonup\mathcal{S}_1(\Phi,\Psi) \qquad\text{weakly in~} L^4(0,T;\mathcal{H}^2)
\end{align}
and
\begin{align}\label{Conv:Gk:PP:L83W24}
	\mathcal{S}_k(\Phi,\Psi)\rightharpoonup\mathcal{S}_1(\Phi,\Psi) \qquad\text{weakly in~} L^{\frac83}(0,T;\mathcal{W}^{2,4}),
\end{align}
respectively, as $k\rightarrow\infty$.
Next, for any $\varepsilon\in(0,\frac12)$, by interpolation, there exists a constant $C = C(\varepsilon) > 0$ such that
\begin{alignat*}{2}
    \norm{f}_{H^{2-\varepsilon}(\Om)} &\leq C\norm{f}_{H^1(\Om)}^\varepsilon\norm{f}_{H^2(\Om)}^{1-\varepsilon} &&\qquad\text{for all~}f\in H^2(\Om), \\
    \norm{g}_{H^{2-\varepsilon}(\Ga)} &\leq C\norm{g}_{H^1(\Ga)}^\varepsilon\norm{g}_{H^2(\Ga)}^{1-\varepsilon} &&\qquad\text{for all~}g\in H^2(\Ga),
\end{alignat*}
see, for instance, \cite[Lemma~2.2]{Colli2024}.
Thus, combining these two estimates, by Young's inequality, we deduce
\begin{align}\label{Interpol:Appl}
    \begin{split}
        \norm{(f,g)}_{L^{\frac{4}{1+\varepsilon}}(0,T;\mathcal{H}^{2-\varepsilon})}
        \leq C\norm{(f,g)}_{L^2(0,T;\mathcal{H}^1)}\norm{(f,g)}_{L^4(0,T;\mathcal{H}^2)}
    \end{split}
\end{align}
for all $(f,g)\in L^4(0,T;\mathcal{H}^2)$.
Choosing $(f,g) = \mathcal{S}_k(\Phi,\Psi) - \mathcal{S}_1(\Phi,\Psi)\in L^4(0,T;\mathcal{H}^2)$ in \eqref{Interpol:Appl}, the convergences \eqref{Conv:Gk:PP:L2H1} and \eqref{Conv:Gk:PP:L4H2} lead to
\begin{align}\label{Conv:Gk:PP:eps}
    \mathcal{S}_k(\Phi,\Psi) \rightarrow \mathcal{S}_1(\Phi,\Psi) \qquad\text{strongly in~} L^{\frac{4}{1+\varepsilon}}(0,T;\mathcal{H}^{2-\varepsilon})
\end{align}
as $k\rightarrow\infty$. Now, we claim that for any $\sigma\in C_c^\infty(0,T)$ it holds that
\begin{align}\label{ChainRuleFormula:k}
	\begin{split}
		&\int_0^T \big(\mathcal{S}_k(\delt\Phi,\delt\Psi),(\Phi,\Psi)\big)_{\mathcal{L}^2}\sigma\ds \\
		&\quad = -\frac12 \int_0^T \norm{(\Phi,\Psi)}_{L,[\phi_1^k,\psi_1^k],\ast}^2\delt\sigma\ds \\
		&\qquad + \frac12 \int_0^T\big(\delt\phi_1^k,m_\Om(\phi_1^k)\Grad\mathcal{S}_k^\Om(\Phi,\Psi)\cdot\Grad\mathcal{S}_k^\Om(\Phi,\Psi)\big)_{\mathcal{L}^2}\sigma\ds \\
		&\qquad + \frac12\int_0^T \big(\delt\psi_1^k,m_\Ga(\psi_1^k)\Gradg\mathcal{S}_k^\Ga(\Phi,\Psi)\cdot\Gradg\mathcal{S}_k^\Ga(\Phi,\Psi)\big)_{\mathcal{L}^2}\sigma\ds.
	\end{split}
\end{align}
To this end, for almost every $t\in(0,T)$, let $c(\cdot,t)$ be measurable in $\Om$, $d(\cdot,t)$ be measurable on $\Ga$, and $(f(\cdot,t),g(\cdot,t))\in \mathcal{V}_{L}^{-1}\cap\mathcal{L}^2$ such that $(\delt c(\cdot,t),\delt d(\cdot,t)), (\delt f(\cdot,t),\delt g(\cdot,t))\in\mathcal{L}^2$. Then, differentiating the weak formulation \eqref{WF:EBS} that is satisfied by $\mathcal{S}_{L}[c,d](f,g)$ and taking $(\zeta,\xi) = \mathcal{S}_{L}[c,d](f,g)\in\mathcal{H}^1_L$ as a test function, we obtain that
\begin{align*}
	&\intO \delt f\,\mathcal{S}_{L}^\Om[c,d](f,g)\dx + \intG \delt g\,\mathcal{S}_{L}^\Ga[c,d](f,g)\dG \\
	&\quad = \intO m_\Om^\prime(c)\delt c\,\big\vert\Grad\mathcal{S}_{L}^\Om[c,d](f,g)\big\vert^2\dx + \intG m_\Ga^\prime(d)\delt d\,\big\vert\Gradg\mathcal{S}_{L}^\Ga[c,d](f,g)\big\vert^2\dG \\
	&\qquad + \intO f\,\delt\mathcal{S}_{L}^\Om[c,d](f,g)\dx + \intG g\,\delt\mathcal{S}_{L}^\Ga[c,d](f,g)\dG.
\end{align*}
Hence, noting again on \eqref{WF:EBS}, we find that
\begin{align}\label{ChainRuleFormula:General}
		\ddt \norm{(f,g)}_{L,[c,d],\ast}^2 &= \ddt\norm{\mathcal{S}_{L}[c,d](f,g)}_{L,[c,d]}^2 \nonumber \\
        &= \ddt \big(\mathcal{S}_{L}[c,d](f,g),(f,g)\big)_{\mathcal{L}^2} \nonumber \\
        &= \big(\delt\mathcal{S}_{L}[c,d](f,g),(f,g)\big)_{\mathcal{L}^2} + \big(\mathcal{S}_{L}[c,d](f,g),(\delt f,\delt g)\big)_{\mathcal{L}^2} \\
		&= 2\big(\mathcal{S}_{L}[c,d](\delt f,\delt g),(f,g)\big)_{\mathcal{L}^2} - \intO m_\Om^\prime(c)\delt c\,\big\vert\Grad\mathcal{S}_{L}^\Om[c,d](f,g)\big\vert^2\dx \nonumber \\
		&\quad - \intG m_\Ga^\prime(d)\delt d\,\big\vert\Gradg\mathcal{S}_{L}^\Ga[c,d](f,g)\big\vert^2\dG. \nonumber 
\end{align}
As $(c,d)$ and $(f,g)$ were arbitrary, we can simply take $(c,d) = (\phi_1^k,\psi_1^k)$ and $(f,g) = (\Phi,\Psi)$ in \eqref{ChainRuleFormula:General} obtaining \eqref{ChainRuleFormula:k}. Now, we wish to take the limit $k\rightarrow\infty$ in \eqref{ChainRuleFormula:k}. To this end, we first notice that from the convergences \eqref{Conv:Gk:delt:PP:L2H1} and \eqref{Conv:Gk:PP:L2H1} we easily conclude that
\begin{align}\label{Conv:CRF:1}
	\lim_{k\rightarrow\infty}\int_0^T \big(\mathcal{S}_k(\delt\Phi,\delt\Psi),(\Phi,\Psi)\big)_{\mathcal{L}^2}\sigma\ds = \int_0^T \big(\mathcal{S}_1(\delt\Phi,\delt\Psi),(\Phi,\Psi)\big)_{\mathcal{L}^2}\sigma\ds
\end{align}
as well as
\begin{align}\label{Conv:CRF:2}
        \lim_{k\rightarrow\infty}\frac12\int_0^T \norm{(\Phi,\Psi)}_{L,[\phi_1^k,\psi_1^k],\ast}^2\delt\sigma\ds &= \lim_{k\rightarrow\infty}\frac12\int_0^T \big(\mathcal{S}_k(\Phi,\Psi),(\Phi,\Psi)\big)_{\mathcal{L}^2}\delt\sigma\ds \nonumber \\
        &= \frac12\int_0^T \big(\mathcal{S}_1(\Phi,\Psi),(\Phi,\Psi)\big)_{\mathcal{L}^2}\delt\sigma\ds \\
        &= \frac12\int_0^T\norm{(\Phi,\Psi)}_{L,[\phi_1,\psi_1],\ast}^2\delt\sigma\ds, \nonumber 
\end{align}
respectively. To handle the last two terms on the right-hand side of \eqref{ChainRuleFormula:k}, we show that there exists a constant $C > 0$, independent of $k\in\N$, such that
\begin{align}\label{Claim}
	\int_0^T\norm{\big(m_\Om^\prime(\phi_1^k)\abs{\Grad\mathcal{S}_k^\Om(\Phi,\Psi)}^2,m_\Ga^\prime(\psi_1^k)\abs{\Gradg\mathcal{S}_k^\Ga(\Phi,\Psi)}^2\big)}_{L}^2\ds \leq C.
\end{align}
In fact, by \eqref{Prelim:Est:Inteprol}, we have
\begin{align*}
	&\norm{\Grad\big(m_\Om^\prime(\phi_1^k)\abs{\Grad\mathcal{S}_k^\Om(\Phi,\Psi)}^2\big)}_{L^2(\Om)} \\
	&\quad\leq \norm{m_\Om^{\prime\prime}(\phi_1^k)\Grad\phi_1^k\abs{\Grad\mathcal{S}_k^\Om(\Phi,\Psi)}^2}_{L^2(\Om)} + 2\norm{m_\Om^\prime(\phi_1^k) D^2\mathcal{S}_k^\Om(\Phi,\Psi)\Grad\mathcal{S}_k^\Om(\Phi,\Psi)}_{L^2(\Om)} \\
	&\quad\leq C\norm{\Grad\phi_1^k}_{L^6(\Om)}\norm{\Grad\mathcal{S}_k^\Om(\Phi,\Psi)}_{L^6(\Om)}^2 + C\norm{D^2\mathcal{S}_k^\Om(\Phi,\Psi)}_{L^4(\Om)}\norm{\Grad\mathcal{S}_k^\Om(\Phi,\Psi)}_{L^4(\Om)} \\
	&\quad\leq C\norm{\Grad\phi_1^k}_{L^2(\Om)}^{\frac13}\norm{\phi_1^k}_{H^2(\Om)}^{\frac23}\norm{\Grad\mathcal{S}_k^\Om(\Phi,\Psi)}_{L^2(\Om)}^{\frac23}\norm{\mathcal{S}_k^\Om(\Phi,\Psi)}_{H^2(\Om)}^{\frac43} \\
	&\qquad + C\norm{D^2\mathcal{S}_k^\Om(\Phi,\Psi)}_{L^4(\Om)}\norm{\Grad\mathcal{S}_k^\Om(\Phi,\Psi)}_{L^2(\Om)}^{\frac12}\norm{\Grad\mathcal{S}_k^\Om(\Phi,\Psi)}_{H^2(\Om)}^{\frac12}.
\end{align*}
Hence, applying the estimates \eqref{Est:PP:k:LinftyH^1} and \eqref{Est:G:k:LinftyH^1}, we find
\begin{align}\label{Est:Claim:1}
	\begin{split}
		&\norm{\Grad\big(m_\Om^\prime(\phi_1^k)\abs{\Grad\mathcal{S}_k^\Om(\Phi,\Psi)}^2\big)}_{L^2(\Om)} \\
		&\quad\leq C\norm{\phi_1^k}_{H^2(\Om)}^{\frac23}\norm{\mathcal{S}_k^\Om(\Phi,\Psi)}_{H^2(\Om)}^{\frac43} + C\norm{D^2\mathcal{S}_k^\Om(\Phi,\Psi)}_{L^4(\Om)}\norm{\mathcal{S}_k^\Om(\Phi,\Psi)}_{H^2(\Om)}^{\frac12} \\
		&\quad\leq C\norm{\phi_1^k}_{H^2(\Om)}^2 + C\norm{\mathcal{S}_k^\Om(\Phi,\Psi)}_{H^2(\Om)}^2 + C\norm{\mathcal{S}_k^\Om(\Phi,\Psi)}_{W^{2,4}(\Om)}^{\frac43}.
	\end{split}
\end{align}
In a similar manner, we derive the estimate
\begin{align}\label{Est:Claim:2}
	\begin{split}
		&\norm{\Gradg\big(m_\Ga(\psi_1^k)\abs{\Gradg\mathcal{S}_k^\Ga(\Phi,\Psi)}^2\big)}_{L^2(\Ga)} \\
		&\quad\leq C\norm{\psi_1^k}_{H^2(\Ga)}^2 + C\norm{\mathcal{S}_k^\Ga(\Phi,\Psi)}_{H^2(\Ga)}^2 + C\norm{\mathcal{S}_k^\Ga(\Phi,\Psi)}_{W^{2,4}(\Ga)}^{\frac43}.
	\end{split}
\end{align}
Lastly, we use the trace theorem and the Sobolev inequality to obtain
\begin{align}\label{Est:Claim:3}
	\begin{split}
		&\norm{\beta m_\Ga^\prime(\psi_1^k)\abs{\Gradg\mathcal{S}_k^\Ga(\Phi,\Psi)}^2 - m_\Om^\prime(\phi_1^k)\abs{\Grad\mathcal{S}_k^\Om(\Phi,\Psi)}^2}_{L^2(\Ga)} \\
		&\quad\leq C\norm{\Grad\mathcal{S}_k^\Om(\Phi,\Psi)}_{L^4(\Ga)}^2 + C\norm{\Gradg\mathcal{S}_k^\Ga(\Phi,\Psi)}_{L^4(\Ga)}^2 \\
		&\quad\leq C\norm{\mathcal{S}_k(\Phi,\Psi)}_{\mathcal{H}^2}^2.
	\end{split}
\end{align}
Thus, in conclusion, combining the estimates \eqref{Est:Claim:1}-\eqref{Est:Claim:3} yields
\begin{align}\label{Est:Claim:Final}
	\begin{split}
		&\norm{\big(m_\Om^\prime(\phi_1^k)\abs{\Grad\mathcal{S}_k^\Om(\Phi,\Psi)}^2,m_\Ga^\prime(\psi_1^k)\abs{\Gradg\mathcal{S}_k^\Ga(\Phi,\Psi)}^2\big)}_{L} \\
		&\quad\leq C\norm{(\phi_1^k,\psi_1^k)}_{\mathcal{H}^2}^2 + C\norm{\mathcal{S}_k(\Phi,\Psi)}_{\mathcal{H}^2}^2 + C\norm{\mathcal{S}_k(\Phi,\Psi)}_{\mathcal{W}^{2,4}}^{\frac43},
	\end{split}
\end{align}
and the desired claim \eqref{Claim} readily follows in light of the estimates \eqref{Est:PP:k:L^4H^2}, \eqref{Est:Gk:PP:L4H2} and \eqref{Est:Gk:PP:L83W24}. We then compute
\begin{align*}
	&\Big(\big(\delt\phi_1^k,\delt\psi_1^k\big),\big(m_\Om^\prime(\phi_1^k)\abs{\Grad\mathcal{S}_k^\Om(\Phi,\Psi)}^2,m_\Ga^\prime(\psi_1^k)\abs{\Gradg\mathcal{S}_k^\Ga(\Phi,\Psi)}^2\big)\Big)_{\mathcal{L}^2}\\
	&\quad =  \Big(\mathcal{S}(\delt\phi_1^k,\delt\psi_1^k),\big(m_\Om^\prime(\phi_1^k)\abs{\Grad\mathcal{S}_k^\Om(\Phi,\Psi)}^2,m_\Ga^\prime(\psi_1^k)\abs{\Gradg\mathcal{S}_k^\Ga(\Phi,\Psi)}^2\big)\Big)_{L} \\
	&\qquad +  \Big(\mathcal{S}(\delt\phi_1^k-\delt\phi_1,\delt\psi_1^k-\delt\psi_1),\big(m_\Om^\prime(\phi_1^k)\abs{\Grad\mathcal{S}_k^\Om(\Phi,\Psi)}^2,m_\Ga^\prime(\psi_1^k)\abs{\Gradg\mathcal{S}_k^\Ga(\Phi,\Psi)}^2\big)\Big)_{L}
\end{align*}
almost everywhere on $(0,T)$. Thus, multiplying the above identity by $\sigma\in C_c^\infty(0,T)$ and integrating in time over $(0,T)$, we take the limit $k\rightarrow\infty$ and obtain using \eqref{Conv:Phi:H^2:k}-\eqref{Conv:Psi:H^2:k} and \eqref{Claim} that
\begin{align}\label{IDK}
	\begin{split}
		&\lim_{k\rightarrow\infty}\int_0^T \Big(\big(\delt\phi_1^k,\delt\psi_1^k\big),\big(m_\Om^\prime(\phi_1^k)\abs{\Grad\mathcal{S}_k^\Om(\Phi,\Psi)}^2,m_\Ga^\prime(\psi_1^k)\abs{\Gradg\mathcal{S}_k^\Ga(\Phi,\Psi)}^2\big)\Big)_{\mathcal{L}^2}\sigma\ds \\
		&\quad = \lim_{k\rightarrow\infty}\int_0^T \Big(\mathcal{S}(\delt\phi_1^k,\delt\psi_1^k),\big(m_\Om^\prime(\phi_1^k)\abs{\Grad\mathcal{S}_k^\Om(\Phi,\Psi)}^2,m_\Ga^\prime(\psi_1^k)\abs{\Gradg\mathcal{S}_k^\Ga(\Phi,\Psi)}^2\big)\Big)_{L}\sigma\ds.
	\end{split}
\end{align}
To compute the limit of the right-hand side of \eqref{IDK}, we first note that by our above computations, specifically \eqref{Est:Claim:1}-\eqref{Est:Claim:3}, there exist functions $h_1,\ldots,h_5$ such that, along a non-relabeled subsequence $k\rightarrow\infty$,
\begin{alignat}{2}
	m_\Om^{\prime\prime}(\phi_1^k)\abs{\Grad\mathcal{S}_k^\Om(\Phi,\Psi)}^2 &\rightharpoonup h_1 &&\qquad\text{weakly in~} L^2(0,T;L^2(\Om)), \label{Conv:h1} \\
	2m_\Om^\prime(\phi_1^k)D^2\mathcal{S}_k^\Om(\Phi,\Psi)\Grad\mathcal{S}_k^\Om(\Phi,\Psi) &\rightharpoonup h_2 &&\qquad\text{weakly in~} L^2(0,T;L^2(\Om)), \label{Conv:h2} \\
	m_\Ga^{\prime\prime}(\psi_1^k)\abs{\Gradg\mathcal{S}_k^\Ga(\Phi,\Psi)}^2 &\rightharpoonup h_3 &&\qquad\text{weakly in~} L^2(0,T;L^2(\Ga)), \label{Conv:h3} \\
	2m_\Ga^\prime(\psi_1^k) D^2_\Ga\mathcal{S}_k^\Ga(\Phi,\Psi)\Gradg\mathcal{S}_k^\Ga(\Phi,\Psi) &\rightharpoonup h_4 &&\qquad\text{weakly in~} L^2(0,T;L^2(\Ga)), \label{Conv:h4} \\
	\beta m_\Ga^\prime(\psi_1^k)\abs{\Gradg\mathcal{S}_k^\Ga(\Phi,\Psi)}^2 - m_\Om^\prime(\phi_1^k)\abs{\Grad\mathcal{S}_k^\Om(\Phi,\Psi)}^2 &\rightharpoonup h_5 &&\qquad\text{weakly in~} L^2(0,T;L^2(\Ga)). \label{Conv:h5}
\end{alignat}
As a last step, we are left with identifying the weak limits $h_1,\ldots,h_5$. First, using \eqref{Conv:Gk:PP:eps} with $\varepsilon = \frac14$, we find that
\begin{align*}
    \mathcal{S}_k(\Phi,\Psi) \rightarrow \mathcal{S}_1(\Phi,\Psi) \qquad\text{strongly in~} L^{\frac{16}{5}}(0,T;\mathcal{H}^{\frac74})
\end{align*}
as $k\rightarrow\infty$. In particular, this implies that
\begin{align*}
    \big(\big\vert\Grad\mathcal{S}_k^\Om(\Phi,\Psi)\big\vert^2,\big\vert\Gradg\mathcal{S}_k^\Ga(\Phi,\Psi)\big\vert^2\big) \rightarrow \big(\big\vert\Grad\mathcal{S}_1^\Om(\Phi,\Psi)\big\vert^2,\big\vert\Gradg\mathcal{S}_1^\Ga(\Phi,\Psi)\big\vert^2\big) \qquad\text{strongly in~} L^{\frac{16}{5}}(0,T;\mathcal{L}^3).
\end{align*}
Next, due to \eqref{Conv:Phi:H^2:k}-\eqref{Conv:Psi:H^2:k}, we have
\begin{align*}
    (\Grad\phi_1^k,\Gradg\psi_1^k) \rightarrow (\Grad\phi_1,\Gradg\psi_1) \qquad\text{strongly in~} L^4(0,T;\mathcal{L}^3).
\end{align*}
Finally, since $m_\Om^{\prime\prime},m_\Ga^{\prime\prime}\in C([-1,1])$, \eqref{Conv:Phi:a.e.} implies that
\begin{align*}
	(m_\Om^{\prime\prime}(\phi_1^k),m_\Ga^{\prime\prime}(\psi_1^k)) \rightarrow (m_\Om^{\prime\prime}(\phi_1),m_\Ga^{\prime\prime}(\psi_1)) \qquad\text{strongly in~} L^8(0,T;\mathcal{L}^{12}).
\end{align*}
Thus, we readily infer that
\begin{alignat*}{2}
	m_\Om^{\prime\prime}(\phi_1^k)\Grad\phi_1^k\abs{\Grad\mathcal{S}_k^\Om(\Phi,\Psi)}^2 &\rightharpoonup m_\Om^{\prime\prime}(\phi_1)\Grad\phi_1\abs{\Grad\mathcal{S}_1^\Om(\Phi,\Psi)}^2 &&\qquad\text{weakly in~} L^1(0,T;L^1(\Om)), \\
    m_\Ga^{\prime\prime}(\psi_1^k)\Gradg\psi_1^k\abs{\Gradg\mathcal{S}_k^\Ga(\Phi,\Psi)}^2 &\rightharpoonup m_\Ga^{\prime\prime}(\psi_1)\Gradg\psi_1\abs{\Gradg\mathcal{S}_1^\Ga(\Phi,\Psi)} ^2&&\qquad\text{weakly in~} L^1(0,T;L^1(\Ga)), 
\end{alignat*}
from which we deduce that
\begin{align*}
	h_1 = m_\Om^{\prime\prime}(\phi_1)\Grad\phi_1\abs{\Grad\mathcal{S}_1^\Om(\Phi,\Psi)}^2,
\end{align*}
as well as
\begin{align*}
    h_3 = m_\Ga^{\prime\prime}(\psi_1)\Gradg\psi_1\abs{\Gradg\mathcal{S}_1^\Ga(\Phi,\Psi)}^2.
\end{align*}
Similarly, owing to \eqref{Conv:Gk:PP:L2H1} and \eqref{Conv:Gk:PP:L83W24}, and noting on
\begin{align*}
	(m_\Om^\prime(\phi_1^k),m_\Ga^\prime(\psi_1^k)) \rightarrow (m_\Om^\prime(\phi_1),m_\Ga^\prime(\psi_1)) \qquad\text{strongly in~} L^8(0,T;\mathcal{L}^4),
\end{align*}
we obtain
\begin{align*}
	&2m_\Om^\prime(\phi_1^k)D^2\mathcal{S}_k^\Om(\Phi,\Psi)\Grad\mathcal{S}_k^\Om(\Phi,\Psi) &&\\
	&\qquad \rightharpoonup 2m_\Om^\prime(\phi_1)D^2\mathcal{S}_1^\Om(\Phi,\Psi)\Grad\mathcal{S}_1^\Om(\Phi,\Psi) &&\qquad\text{weakly in~} L^1(0,T;L^1(\Om)), \\
    &2m_\Ga^\prime(\psi_1^k)D_\Ga^2\mathcal{S}_k^\Ga(\Phi,\Psi)\Gradg\mathcal{S}_k^\Ga(\Phi,\Psi) && \\
    &\qquad \rightharpoonup 2m_\Ga^\prime(\psi_1)D_\Ga^2\mathcal{S}_1^\Ga(\Phi,\Psi)\Gradg\mathcal{S}_1^\Ga(\Phi,\Psi) &&\qquad\text{weakly in~} L^1(0,T;L^1(\Ga)).
\end{align*}
This entails that
\begin{align*}
	h_2 = 2m_\Om^\prime(\phi_1)D^2\mathcal{S}_1^\Om(\Phi,\Psi)\Grad\mathcal{S}_1^\Om(\Phi,\Psi)
\end{align*}
and
\begin{align*}
    h_4 = 2m_\Ga^\prime(\psi_1)D_\Ga^2\mathcal{S}_1^\Ga(\Phi,\Psi)\Gradg\mathcal{S}_1^\Ga(\Phi,\Psi).
\end{align*}
Lastly, to identify $h_5$ we note that given the aforementioned convergences and the trace theorem, we find that
\begin{align*}
    \abs{\Grad\mathcal{S}_k^\Om(\Phi,\Psi)\vert_\Ga}^2 &\rightharpoonup \abs{\Grad\mathcal{S}_1^\Om(\Phi,\Psi)\vert_\Ga}^2 &&\text{strongly in~} L^{\frac{16}{3}}(0,T;L^{\frac32}(\Ga)),\\
    \abs{\Gradg\mathcal{S}_k^\Ga(\Phi,\Psi)}^2 &\rightharpoonup \abs{\Gradg\mathcal{S}_1^\Ga(\Phi,\Psi)}^2  &&\text{strongly in~} L^{\frac{16}{3}}(0,T;L^{\frac32}(\Ga))
\end{align*}
as $k\rightarrow\infty$, and thus, readily deduce that
\begin{align*}
	&\beta m_\Ga^\prime(\psi_1^k)\abs{\Gradg\mathcal{S}_k^\Ga(\Phi,\psi)}^2 - m_\Om^\prime(\phi_1^k)\abs{\Grad\mathcal{S}_k^\Om(\Phi,\Psi)}^2 \\
	&\qquad\rightharpoonup \beta m_\Ga^\prime(\psi_1)\abs{\Gradg\mathcal{S}_1^\Ga(\Phi,\Psi)}^2 - m_\Om^\prime(\phi_1)\abs{\Grad\mathcal{S}_1^\Om(\Phi,\Psi)}^2 \qquad\text{weakly in~} L^1(0,T;L^1(\Ga)).
\end{align*}
This shows 
\begin{align*}
	h_5 = \beta m_\Ga^\prime(\psi_1)\abs{\Gradg\mathcal{S}_1^\Ga(\Phi,\Psi)}^2 - m_\Om^\prime(\phi_1)\abs{\Grad\mathcal{S}_1^\Om(\Phi,\Psi)}^2.
\end{align*}
Therefore, by exploiting \eqref{Conv:h1}-\eqref{Conv:h5} in \eqref{IDK}, we finally deduce that
\begin{align*}
	&\int_0^T \big(\mathcal{S}_1(\delt\Phi,\delt\Psi)(\Phi,\Psi)\big)_{\mathcal{L}^2}\sigma\ds \\
	&\quad = \frac12 \int_0^T \norm{(\Phi,\Psi)}_{L,[\phi_1,\psi_1],\ast}\delt\sigma\ds \\
	&\qquad + \frac12 \int_0^T \Big(\mathcal{S}(\delt\phi_1,\delt\psi_1),\big(m_\Om^\prime(\phi_1)\abs{\Grad\mathcal{S}_1^\Om(\Phi,\Psi)}^2,m_\Ga^\prime(\psi_1)\abs{\Gradg\mathcal{S}_1^\Ga(\Phi,\Psi)}^2\big)\Big)_{L}\sigma\ds
\end{align*}
for any $\sigma\in C_c^\infty(0,T)$. The latter readily implies the desired conclusion \eqref{ChainRuleFormula}.

\textbf{Estimates of the nonlinear terms.}
In the rest of the proof, the letter $C$ denotes a generic positive constant that may change its value from line to line, and which depends on the parameter of the system and the initial energy $E(\phi_0,\psi_0)$.
We now intend to bound the terms $I_1$ and $I_2$. 

Regarding $I_1$, we have
\begin{align*}
	\abs{I_1} &= \big\vert\big(\mathcal{S}_{L}(\delt\phi_1,\delt\psi_1)),(m_\Om^\prime(\phi_1)\abs{\Grad\mathcal{S}_1^\Om(\Phi,\Psi)}^2,m_\Ga^\prime(\psi_1)\abs{\Gradg\mathcal{S}_1^\Ga(\Phi,\Psi)}^2)\big)_L\big\vert \\
	&\leq C\norm{(\delt\phi_1,\delt\psi_1)}_{(\mathcal{H}^1_L)^\prime}\norm{(m_\Om^\prime(\phi_1)\abs{\Grad\mathcal{S}_1^\Om(\Phi,\Psi)}^2,m_\Ga^\prime(\psi_1)\abs{\Gradg\mathcal{S}_1^\Ga(\Phi,\Psi)}^2)}_{L}.
\end{align*}
For the second term on the right-hand side, we note that
\begin{align*}
	&\norm{(m_\Om^\prime(\phi_1)\abs{\Grad\mathcal{S}_1^\Om(\Phi,\Psi)}^2,m_\Ga^\prime(\psi_1)\abs{\Gradg\mathcal{S}_1^\Ga(\Phi,\Psi)}^2)}_{L} \\
	&\quad \leq\norm{\big(m_\Om^{\prime\prime}(\phi_1)\Grad\phi_1\abs{\Grad\mathcal{S}_1^\Om(\Phi,\Psi)}^2,m_\Ga^{\prime\prime}(\psi_1)\Gradg\psi_1\abs{\Gradg\mathcal{S}_1^\Ga(\Phi,\Psi)}^2\big)}_{\mathcal{L}^2} \\
	&\qquad + \norm{\big(m_\Om^\prime(\phi_1)D^2\mathcal{S}_1^\Om(\Phi,\Psi)\Grad\mathcal{S}_1^\Om(\Phi,\psi),m_\Ga^\prime(\psi_1)D^2_\Ga\mathcal{S}_1^\Ga(\Phi,\Psi)\Gradg\mathcal{S}_1^\Ga(\Phi,\Psi)\big)}_{\mathcal{L}^2} \\
	&\qquad + \chi(L)\norm{\beta m_\Ga^\prime(\psi_1)\abs{\Gradg\mathcal{S}_1^\Ga(\Phi,\Psi)}^2 - m_\Om^\prime(\phi_1)\abs{\Grad\mathcal{S}_1^\Om(\Phi,\Psi)}^2}_{L^2(\Ga)},
\end{align*}
which entails
\begin{align*}
	\abs{I_1} &\leq C\norm{(\delt\phi_1,\delt\psi_1)}_{(\mathcal{H}^1_L)^\prime}\norm{\big(\Grad\phi_1\abs{\Grad\mathcal{S}_1^\Om(\Phi,\Psi)}^2,\Gradg\psi_1\abs{\Gradg\mathcal{S}_1^\Ga(\Phi,\Psi)}^2\big)}_{\mathcal{L}^2} \\
	&\quad + C\norm{(\delt\phi_1,\delt\psi_1)}_{(\mathcal{H}^1_L)^\prime}\norm{\big(D^2\mathcal{S}_1^\Om(\Phi,\Psi)\Grad\mathcal{S}_1^\Om(\Phi,\Psi),D^2_\Ga\mathcal{S}_1^\Ga(\Phi,\Psi)\Gradg\mathcal{S}_1^\Ga(\Phi,\Psi)\big)}_{\mathcal{L}^2} \\
	&\quad + C\norm{(\delt\phi_1,\delt\psi_1)}_{(\mathcal{H}^1_L)^\prime}\chi(L)\norm{\beta m_\Ga^\prime(\psi_1)\abs{\Gradg\mathcal{S}_1^\Ga(\Phi,\Psi)}^2 - m_\Om^\prime(\phi_1)\abs{\Grad\mathcal{S}_1^\Om(\Phi,\Psi)}^2}_{L^2(\Ga)} \\
	&\eqqcolon J_1 + J_2 + J_3.
\end{align*}
Utilizing \eqref{Est:Sol:G:H^2} and owing to \eqref{Est:fg:L^2:SolOp:2} in the case $K\in(0,\infty)$, we see that
\begin{align}\label{Est:S_1:H^2:PhiPsi}
    \begin{split}
        &\norm{\mathcal{S}_1(\Phi,\Psi)}_{\mathcal{H}^2} \\
        &\quad\leq C\Big(\norm{(\Grad\phi_1,\Gradg\psi_1)}_{\mathcal{L}^2}\norm{(\phi_1,\psi_1)}_{\mathcal{H}^2}\norm{\mathcal{S}_1(\Phi,\Psi)}_{L} + \norm{\mathcal{S}_1(\Phi,\Psi)}_{L}^{\frac12}\norm{(\Phi,\Psi)}_{K}^{\frac12}\Big).
    \end{split}
\end{align}
Then, concerning $J_1$, we employ the estimates \eqref{Prelim:Est:Inteprol}, \eqref{Est:PP:H1:Linfty} and \eqref{Est:S_1:H^2:PhiPsi}, and find
\begin{align}\label{Est:J1}
    J_1 &\leq C\norm{(\delt\phi_1,\delt\psi_1)}_{(\mathcal{H}^1_L)^\prime}\norm{(\Grad\phi_1,\Gradg\psi_1)}_{\mathcal{L}^6}\norm{\big(\Grad\mathcal{S}_1^\Om(\Phi,\Psi),\Gradg\mathcal{S}_1^\Ga(\Phi,\Psi)\big)}_{\mathcal{L}^6}^2 \nonumber \\
    &\leq C\norm{(\delt\phi_1,\delt\psi_1)}_{(\mathcal{H}^1_L)^\prime}\norm{(\Grad\phi_1,\Gradg\psi_1)}_{\mathcal{L}^2}^{\frac13}\norm{(\phi_1,\psi_1)}_{\mathcal{H}^2}^{\frac23}\norm{\mathcal{S}_1(\Phi,\Psi)}_{L}^{\frac23}\norm{\mathcal{S}_1(\Phi,\Psi)}_{\mathcal{H}^2}^{\frac43} \nonumber \\
    &\leq C\norm{(\delt\phi_1,\delt\psi_1)}_{(\mathcal{H}^1_L)^\prime}\norm{(\phi_1,\psi_1)}_{\mathcal{H}^2}^{\frac23}\norm{\mathcal{S}_1(\Phi,\Psi)}_{L}^{\frac23} \nonumber \\
    &\quad\times\Big(\norm{(\phi_1,\psi_1)}_{\mathcal{H}^2}\norm{\mathcal{S}_1(\Phi,\Psi)}_{L} + \norm{\mathcal{S}_1(\Phi,\Psi)}_{L}^{\frac12}\norm{(\Phi,\Psi)}_{L}^{\frac12}\Big)^{\frac43} \\
    &\leq C\norm{(\delt\phi_1,\delt\psi_1)}_{(\mathcal{H}^1_L)^\prime}\Big(\norm{(\phi_1,\psi_1)}_{\mathcal{H}^2}^2\norm{\mathcal{S}_1(\Phi,\Psi)}_{L}^2 + C\norm{(\phi_1,\psi_1)}_{\mathcal{H}^2}^{\frac23}\norm{\mathcal{S}_1(\Phi,\Psi)}_{L}^{\frac43}\norm{(\Phi,\Psi)}_{K}^{\frac23}\Big) \nonumber \\
    &\leq \frac{1}{18}\norm{(\Phi,\Psi)}_{K}^2 + C\Big(\norm{(\delt\phi_1,\delt\psi_1)}_{(\mathcal{H}^1_L)^\prime}\norm{(\phi_1,\psi_1)}_{\mathcal{H}^2}^2 + \norm{(\delt\phi_1,\delt\psi_1)}_{(\mathcal{H}^1_L)^\prime}\norm{(\phi_1,\psi_1)}_{\mathcal{H}^2}\Big) \nonumber \\
    &\quad\times\norm{\mathcal{S}_1(\Phi,\Psi)}_{L}^2 \nonumber \\
    &\leq \frac{1}{18}\norm{(\Phi,\Psi)}_{K}^2 + C\Big(\norm{(\delt\phi_1,\delt\psi_1)}_{(\mathcal{H}^1_L)^\prime}^2 + \norm{(\phi_1,\psi_1)}_{\mathcal{H}^2}^4\Big)\norm{\mathcal{S}_1(\Phi,\Psi)}_{L}^2. \nonumber 
\end{align}
If however $K = \infty$, a similar argumentation owing to \eqref{Est:fg:L^2:SolOp:2:K=infty} leads to the estimate
\begin{align*}
    J_1 \leq \frac{1}{18}\norm{(\Phi,\Psi)}_{K}^2 + C\Big(1 + \norm{(\delt\phi_1,\delt\psi_1)}_{(\mathcal{H}^1_L)^\prime}^2 + \norm{(\phi_1,\psi_1)}_{\mathcal{H}^2}^4\Big)\norm{\mathcal{S}_1(\Phi,\Psi)}_{L}^2.
\end{align*}
In the following, we restrict ourselves to the case $K\in(0,\infty)$, as the case $K = \infty$ can be handled similarly as above.
Next, we consider $J_2$. Due to \eqref{Prelim:Est:Inteprol}, \eqref{Est:PP:H1:Linfty} and \eqref{Est:S_1:H^2:PhiPsi} we have
\begin{align}\label{Est:S_1:Grad:L4:PhiPsi}
    \begin{split}
    	&\norm{(\Grad\mathcal{S}_1^\Om(\Phi,\Psi),\Gradg\mathcal{S}_1^\Ga(\Phi,\Psi))}_{\mathcal{L}^4} \\
    	&\quad\leq C\norm{\mathcal{S}_1(\Phi,\Psi)}_{L}^{\frac12}\norm{\mathcal{S}_1(\Phi,\Psi)}_{\mathcal{H}^2}^{\frac12} \\
    	&\quad\leq C\Big(\norm{(\Grad\phi_1,\Gradg\psi_1)}_{\mathcal{L}^2}^{\frac12}\norm{(\phi_1,\psi_1)}_{\mathcal{H}^2}^{\frac12}\norm{\mathcal{S}_1(\Phi,\Psi)}_{L} + \norm{\mathcal{S}_1(\Phi,\Psi)}_{L}^{\frac34}\norm{(\Phi,\Psi)}_{K}^{\frac14}\Big) \\
        &\quad\leq C\Big(\norm{(\phi_1,\psi_1)}_{\mathcal{H}^2}^{\frac12}\norm{\mathcal{S}_1(\Phi,\Psi)}_{L} + \norm{\mathcal{S}_1(\Phi,\Psi)}_{L}^{\frac34}\norm{(\Phi,\Psi)}_{K}^{\frac34}\Big)
    \end{split}
\end{align}
On the other hand, exploiting
\begin{align}\label{Est:PhiPsi:L4:Interpol}
    \begin{split}
    	\norm{(\Phi,\Psi)}_{\mathcal{L}^4} &\leq C\norm{(\Phi,\Psi)}_{\mathcal{L}^2}^{\frac12}\norm{(\Phi,\Psi)}_{\mathcal{H}^1}^{\frac12} \leq C\norm{\mathcal{S}_1(\Phi,\Psi)}_{L}^{\frac14}\norm{(\Phi,\Psi)}_{K}^{\frac14}\norm{(\Phi,\Psi)}_{\mathcal{H}^1}^{\frac12} \\
    	&\leq C\norm{\mathcal{S}_1(\Phi,\Psi)}_{L}^{\frac14}\norm{(\Phi,\Psi)}_{K}^{\frac34}, 
    \end{split}
\end{align}
which follows from Lemma~\ref{Prelim:Poincare} (note that $\mean{\Phi}{\Psi} = 0$), \eqref{Prelim:Est:Inteprol} and \eqref{Est:fg:L^2:SolOp:2}, we obtain similarly from \eqref{Est:Sol:G:W^24}, using again \eqref{Est:PP:H1:Linfty} as well as \eqref{Est:S_1:H^2:PhiPsi}, that
\begin{align}\label{Est:S_1:W24:PhiPsi}
    	\norm{\mathcal{S}_1(\Phi,\Psi)}_{\mathcal{W}^{2,4}} &\leq C\Big(\norm{(\phi_1,\psi_1)}_{\mathcal{H}^2}\norm{\mathcal{S}_1(\Phi,\Psi)}_{L}^{\frac14}\norm{\mathcal{S}_1(\Phi,\Psi)}_{\mathcal{H}^2}^{\frac34} + \norm{\mathcal{S}_1(\Phi,\Psi)}_{L}^{\frac14}\norm{(\Phi,\Psi)}_{K}^{\frac34}\Big) \nonumber \\
    	&\leq C\Big(\norm{(\phi_1,\psi_1)}_{\mathcal{H}^2}^{\frac32}\norm{\mathcal{S}_1(\Phi,\Psi)}_{L} + \norm{(\phi_1,\psi_1)}_{\mathcal{H}^2}^{\frac34}\norm{\mathcal{S}_1(\Phi,\Psi)}_{L}^{\frac58}\norm{(\Phi,\Psi)}_{L}^{\frac38} \\
    	&\qquad + \norm{\mathcal{S}_1(\Phi,\Psi)}_{L}^{\frac14}\norm{(\Phi,\Psi)}_{K}^{\frac34}\Big). \nonumber 
\end{align}
Consequently, noticing that
\begin{align*}
	J_2 &\leq \norm{(\delt\phi_1,\delt\psi_1)}_{(\mathcal{H}^1_L)^\prime}\norm{(D^2\mathcal{S}_1^\Om(\Phi,\Psi),D^2_\Ga\mathcal{S}_1^\Ga(\Phi,\Psi))}_{\mathcal{L}^4}\\
    &\quad\times\norm{(\Grad\mathcal{S}_1^\Om(\Phi,\Psi),\Gradg\mathcal{S}_1^\Ga(\Phi,\Psi))}_{\mathcal{L}^4},
\end{align*}
we deduce from \eqref{Est:S_1:Grad:L4:PhiPsi} and \eqref{Est:S_1:W24:PhiPsi} that
\begin{align*}
	J_2 &\leq C\norm{(\delt\phi_1,\delt\psi_1)}_{(\mathcal{H}^1_L)^\prime}\norm{(\phi_1,\psi_1)}_{\mathcal{H}^2}^2\norm{\mathcal{S}_1(\Phi,\Psi)}_{L}^2 \\
	&\qquad + C\norm{(\delt\phi_1,\delt\psi_1)}_{(\mathcal{H}^1_L)^\prime}\norm{(\phi_1,\psi_1)}_{\mathcal{H}^2}^{\frac32}\norm{\mathcal{S}_1(\Phi,\Psi)}_{L}^{\frac74}\norm{(\Phi,\Psi)}_{K}^{\frac14} \\
	&\qquad + C\norm{(\delt\phi_1,\delt\psi_1)}_{(\mathcal{H}^1_L)^\prime}\norm{(\phi_1,\psi_1)}_{\mathcal{H}^2}^{\frac54}\norm{\mathcal{S}_1(\Phi,\Psi)}_{L}^{\frac{13}{8}}\norm{(\Phi,\Psi)}_{K}^{\frac38} \\
	&\qquad + C\norm{(\delt\phi_1,\delt\psi_1)}_{(\mathcal{H}^1_L)^\prime}\norm{(\phi_1,\psi_1)}_{\mathcal{H}^2}^{\frac34}\norm{\mathcal{S}_1(\Phi,\Psi)}_{L}^{\frac{11}{8}}\norm{(\Phi,\Psi)}_{K}^{\frac58} \\
	&\qquad + C\norm{(\delt\phi_1,\delt\psi_1)}_{(\mathcal{H}^1_L)^\prime}\norm{(\phi_1,\psi_1)}_{\mathcal{H}^2}^{\frac12}\norm{\mathcal{S}_1(\Phi,\Psi)}_{L}^{\frac54}\norm{(\Phi,\Psi)}_{K}^{\frac34} \\
	&\qquad + C\norm{(\delt\phi_1,\delt\psi_1)}_{(\mathcal{H}^1_L)^\prime}\norm{\mathcal{S}_1(\Phi,\Psi)}_{L}\norm{(\Phi,\Psi)}_{K} \\
	&\quad \eqqcolon K_1 + \ldots K_6.
\end{align*}
Next, we control the terms $K_1,\ldots,K_6$ separately. Applying Young's inequality, we readily find that
\begin{align*}
	K_1 &\leq C\Big(\norm{(\delt\phi_1,\delt\psi_1)}_{(\mathcal{H}^1_L)^\prime}^2 + \norm{(\phi_1,\psi_1)}_{\mathcal{H}^2}^4\Big)\norm{\mathcal{S}_1(\Phi,\Psi)}_{L}^2, \\
	K_2 &\leq \frac{1}{18}\norm{(\Phi,\Psi)}_{K}^2 + C\norm{(\delt\phi_1,\delt\psi_1)}_{(\mathcal{H}^1_L)^\prime}^{\frac87}\norm{(\phi_1,\psi_1)}_{\mathcal{H}^2}^{\frac{12}{7}}\norm{\mathcal{S}_1(\Phi,\Psi)}_{L}^2 \\
	&\leq \frac{1}{18}\norm{(\Phi,\Psi)}_{K}^2 + C\Big(\norm{(\delt\phi_1,\delt\psi_1)}_{(\mathcal{H}^1_L)^\prime}^2 + \norm{(\phi_1,\psi_1)}_{\mathcal{H}^2}^4\Big)\norm{\mathcal{S}_1(\Phi,\Psi)}_{L}^2, \\
	K_3 &\leq \frac{1}{18}\norm{(\Phi,\Psi)}_{K}^2 + C\norm{(\delt\phi_1,\delt\psi_1)}_{(\mathcal{H}^1_L)^\prime}^{\frac{16}{13}}\norm{(\phi_1,\psi_1)}_{\mathcal{H}^2}^{\frac{20}{13}}\norm{\mathcal{S}_1(\Phi,\Psi)}_{L}^2 \\
	&\leq \frac{1}{18}\norm{(\Phi,\Psi)}_{K}^2 + C\Big(\norm{(\delt\phi_1,\delt\psi_1)}_{(\mathcal{H}^1_L)^\prime}^2 + \norm{(\phi_1,\psi_1)}_{\mathcal{H}^2}^4\Big)\norm{\mathcal{S}_1(\Phi,\Psi)}_{L}^2, \\
	K_4 &\leq \frac{1}{18}\norm{(\Phi,\Psi)}_{K}^2 + C\norm{(\delt\phi_1,\delt\psi_1)}_{(\mathcal{H}^1_L)^\prime}^{\frac{16}{11}}\norm{(\phi_1,\psi_1)}_{\mathcal{H}^2}^{\frac{12}{11}}\norm{\mathcal{S}_1(\Phi,\Psi)}_{L}^2 \\
	&\leq \frac{1}{18}\norm{(\Phi,\Psi)}_{K}^2 + C\Big(\norm{(\delt\phi_1,\delt\psi_1)}_{(\mathcal{H}^1_L)^\prime}^2 + \norm{(\phi_1,\psi_1)}_{\mathcal{H}^2}^4\Big)\norm{\mathcal{S}_1(\Phi,\Psi)}_{L}^2, \\
	K_5 &\leq \frac{1}{18}\norm{(\Phi,\Psi)}_{K}^2 + C\norm{(\delt\phi_1,\delt\psi_1)}_{(\mathcal{H}^1_L)^\prime}^{\frac85}\norm{(\phi_1,\psi_1)}_{\mathcal{H}^2}^{\frac45}\norm{\mathcal{S}_1(\Phi,\Psi)}_{L}^2 \\
	&\leq \frac{1}{18}\norm{(\Phi,\Psi)}_{K}^2 + C\Big(\norm{(\delt\phi_1,\delt\psi_1)}_{(\mathcal{H}^1_L)^\prime}^2 + \norm{(\phi_1,\psi_1)}_{\mathcal{H}^2}^4\Big)\norm{\mathcal{S}_1(\Phi,\Psi)}_{L}^2, \\
	K_6 &\leq \frac{1}{18}\norm{(\Phi,\Psi)}_{K}^2 + C\norm{(\delt\phi_1,\delt\psi_1)}_{(\mathcal{H}^1_L)^\prime}^2\norm{\mathcal{S}_1(\Phi,\Psi)}_{L}^2.
\end{align*}
Therefore, we conclude that
\begin{align}\label{Est:J2}
    J_2 \leq \frac{5}{18}\norm{(\Phi,\Psi)}_{K}^2 + C\Big(\norm{(\delt\phi_1,\delt\psi_1)}_{(\mathcal{H}^1_L)^\prime}^2 + \norm{(\phi_1,\psi_1)}_{\mathcal{H}^2}^4\Big)\norm{\mathcal{S}_1(\Phi,\Psi)}_{L}^2.
\end{align}
For $J_3$, we can simply use the trace theorem and the Sobolev inequality in combination with \eqref{Est:EBS:Apriori} and \eqref{Est:S_1:H^2:PhiPsi} to find that
\begin{align}\label{Est:J3}
        J_3 &\leq C\norm{(\delt\phi_1,\delt\psi_1)}_{(\mathcal{H}^1_L)^\prime}\chi(L)\norm{\beta m_\Ga^\prime(\psi_1)\big\vert\Gradg\mathcal{S}_1^\Ga(\Phi,\Psi)\big\vert^2 - m_\Om^\prime(\phi_1)\big\vert\Grad\mathcal{S}_1^\Om(\Phi,\Psi)\big\vert^2}_{L^2(\Ga)} \nonumber \\
        &\leq C\chi(L)\norm{(\delt\phi_1,\delt\psi_1)}_{(\mathcal{H}^1_L)^\prime}\norm{(\Grad\mathcal{S}_1^\Om(\Phi,\Psi),\Gradg\mathcal{S}_1^\Ga(\Phi,\Psi))}_{L^4(\Ga)}^2 \nonumber \\
        &\leq C\chi(L)\norm{(\delt\phi_1,\delt\psi_1)}_{(\mathcal{H}^1_L)^\prime}\norm{\mathcal{S}_1(\Phi,\Psi)}_{\mathcal{H}^2}^2 \nonumber \\
        &\leq C\chi(L)\norm{(\delt\phi_1,\delt\psi_1)}_{(\mathcal{H}^1_L)^\prime}\Big(\norm{(\phi_1,\psi_1)}_{\mathcal{H}^2}^2\norm{\mathcal{S}_1(\Phi,\Psi)}_{L}^2 + \norm{\mathcal{S}_1(\Phi,\Psi)}_{L}\norm{(\Phi,\Psi)}_{K}\Big) \\
        &\leq \frac{1}{18}\norm{(\Phi,\Psi)}_{K}^2 + C\chi(L)\Big(\norm{(\delt\phi_1,\delt\psi_1)}_{(\mathcal{H}^1_L)^\prime}^2 + \norm{(\delt\phi_1,\delt\psi_1)}_{(\mathcal{H}^1_L)^\prime}\norm{(\phi_1,\psi_1)}_{\mathcal{H}^2}^2\Big) \nonumber \\
        &\quad\times\norm{\mathcal{S}_1(\Phi,\Psi)}_{L}^2 \nonumber \\
        &\leq \frac{1}{18}\norm{(\Phi,\Psi)}_{K}^2 + C\chi(L)\Big(\norm{(\delt\phi_1,\delt\psi_1)}_{(\mathcal{H}^1_L)^\prime}^2 + \norm{(\phi_1,\psi_1)}_{\mathcal{H}^2}^4\Big)\norm{\mathcal{S}_1(\Phi,\Psi)}_{L}^2. \nonumber 
\end{align}
Collecting the the estimates \eqref{Est:J1}, \eqref{Est:J2} and \eqref{Est:J3}, we obtain the following bound for $I_1$:
\begin{align}\label{Est:I1}
    \abs{I_1} \leq \frac{7}{18}\norm{(\Phi,\Psi)}_{K}^2 + C\Big(\norm{(\delt\phi_1,\delt\psi_1)}_{(\mathcal{H}^1_L)^\prime}^2 + \norm{(\phi_1,\psi_1)}_{\mathcal{H}^2}^4\Big)\norm{\mathcal{S}_1(\Phi,\Psi)}_{L}^2.
\end{align}
Next, concerning $I_2$, we employ again \eqref{Prelim:Est:Inteprol}, \eqref{Est:S_1:Grad:L4:PhiPsi} and \eqref{Est:PhiPsi:L4:Interpol}, and deduce
\begin{align}\label{Est:I2}
    \abs{I_2} &= \Big\vert \intO \big(m_\Om(\phi_1) - m_\Om(\phi_2)\big)\Grad\mu_2\cdot\Grad\mathcal{S}^\Om_1(\Phi,\Psi)\dx \nonumber \\
    &\qquad + \intG \big(m_\Ga(\psi_1) - m_\Ga(\psi_2)\big)\Gradg\theta_2\cdot\Gradg\mathcal{S}^\Ga_1(\Phi,\Psi)\dG\Big\vert \nonumber \\
    &\leq \norm{(\Grad\mu_2,\Gradg\theta_2)}_{\mathcal{L}^2}\norm{(m_\Om(\phi_1) - m_\Om(\phi_2),m_\Ga(\psi_1) - m_\Ga(\psi_2))}_{\mathcal{L}^4}\norm{(\Grad\mathcal{S}_1^\Om(\Phi,\Psi),\Gradg\mathcal{S}_1^\Ga(\Phi,\Psi))}_{\mathcal{L}^4} \nonumber \\
    &\leq C\norm{(\Grad\mu_2,\Gradg\theta_2)}_{\mathcal{L}^2}\norm{(\Phi,\Psi)}_{\mathcal{L}^4}\norm{(\Grad\mathcal{S}_1^\Om(\Phi,\Psi),\Gradg\mathcal{S}_1^\Ga(\Phi,\Psi))}_{\mathcal{L}^4} \nonumber \\
    &\leq C\norm{(\mu_2,\theta_2)}_{L}\norm{\mathcal{S}_1(\Phi,\Psi)}_{L}^{\frac14}\norm{(\Phi,\Psi)}_{K}^{\frac34} \nonumber \\
    &\quad\times\Big(\norm{(\phi_1,\psi_1)}_{\mathcal{H}^2}^{\frac12}\norm{\mathcal{S}_1(\Phi,\Psi)}_{L} + \norm{\mathcal{S}_1(\Phi,\Psi)}_{L}^{\frac34}\norm{(\Phi,\Psi)}_{K}^{\frac14}\Big) \nonumber \\
    &\leq C\norm{(\mu_2,\theta_2)}_{L}\norm{(\phi_1,\psi_1)}_{\mathcal{H}^2}^{\frac12}\norm{\mathcal{S}_1(\Phi,\Psi)}_{L}^{\frac54}\norm{(\Phi,\Psi)}_{K}^{\frac34} \\
    &\quad + C\norm{(\mu_2,\theta_2)}_{L}\norm{\mathcal{S}_1(\Phi,\Psi)}_{L}\norm{(\Phi,\Psi)}_{K} \nonumber \\
    &\leq \frac{1}{18}\norm{(\Phi,\Psi)}_{K}^2 + C\Big(\norm{(\mu_2,\theta_2)}_{L}^{\frac85}\norm{(\phi_1,\psi_1)}_{\mathcal{H}^2}^{\frac45} + \norm{(\mu_2,\theta_2)}_{L}^2\Big)\norm{\mathcal{S}_1(\Phi,\Psi)}_{L}^2 \nonumber \\
    &\leq \frac{1}{18}\norm{(\Phi,\Psi)}_{K}^2 + C\Big(\norm{(\mu_2,\theta_2)}_{L}^2 + \norm{(\phi_1,\psi_1)}_{\mathcal{H}^2}^4\Big)\norm{\mathcal{S}_1(\Phi,\Psi)}_{L}^2. \nonumber
\end{align}
Lastly, utilizing \eqref{Est:fg:L^2:SolOp:2} once more, we find
\begin{align}\label{Est:PhiPsi:L2}
	C\norm{(\Phi,\Psi)}_{\mathcal{L}^2}\leq C\norm{\mathcal{S}_1(\Phi,\Psi)}_{L}\norm{(\Phi,\Psi)}_{K} \leq \frac{1}{18}\norm{(\Phi,\Psi)}_{K}^2 + C\norm{\mathcal{S}_1(\Phi,\Psi)}_{L}^2.
\end{align}
Recalling the equivalence of the respective norms on $\mathcal{V}^{-1}_{L}$, we combine the differential inequality \eqref{DiffIneq} with the estimates \eqref{Est:I1}, \eqref{Est:I2} and \eqref{Est:PhiPsi:L2}, and end up with
\begin{align}\label{Est:PreGronwall:Uniq}
	&\ddt \frac12\norm{(\Phi,\Psi)}_{L,[\phi_1,\psi_1],\ast}^2 + \frac12\norm{(\Phi,\Psi)}_{K}^2 \leq Q(t)\norm{(\Phi,\Psi)}_{L,[\phi_1,\psi_1],\ast}^2,
\end{align}
where
\begin{align*}
	Q(\cdot) = C\Big(1 + \norm{(\mu_2,\theta_2)}_{L}^2 + \norm{(\delt\phi_1,\delt\psi_1)}_{(\mathcal{H}^1_L)^\prime}^2 + \norm{(\phi_1,\psi_1)}_{\mathcal{H}^2}^4\Big) \in L^1(0,T) \qquad\text{for all~}T > 0.
\end{align*}
Therefore, noting again the equivalence of the respective norms on $\mathcal{V}^{-1}_{L}$, an application of Gronwall's lemma entails that
\begin{align*}
	\norm{\big(\phi_1(t) - \phi_2(t), \psi_1(t) - \psi_2(t)\big)}_{(\mathcal{H}^1_L)^\prime}^2 \leq \norm{\big(\phi_1^0 - \phi_2^0,\psi_1^0 - \psi_2^0\big)}_{(\mathcal{H}^1_L)^\prime}^2\exp\Big(\int_0^t Q(s)\ds\Big)
\end{align*}
for all $t\geq 0$. Consequently, if $\phi_1^0 = \phi_2^0$ a.e. in $\Om$ and $\psi_1^0 = \psi_2^0$ a.e. on $\Ga$, the above inequality yields the uniqueness of weak solutions.
\end{proof}

\medskip
\section{Propagation of Regularity and Instantaneous Separation Property}
\label{Section:PropRegSep}

In this section, we prove Theorem~\ref{Theorem:PropReg} and Theorem~\ref{Theorem:Separation} regarding the propagation of regularity and the instantaneous separation property.

The proof of Theorem~\ref{Theorem:PropReg} is based on deriving suitable estimates for the time difference quotients $(\delth\phi,\delth\psi)$ by exploiting the differential inequality \eqref{Est:PreGronwall:Uniq}. This argument, however, requires the mobility functions to satisfy higher regularity assumptions, namely, $m_\Om,m_\Ga\in C^2([-1,1])$. Therefore, the strategy of the proof is as follows. First, under the stronger assumption $m_\Om,m_\Ga\in C^2([-1,1])$, we prove that the unique weak solution to \eqref{EQ:SYSTEM} satisfies the desired propagation of regularity. Then, in the second step, we return to the original assumption $m_\Om,m_\Ga\in C^1([-1,1])$, and construct a sequence of smooth approximations of the mobility functions lying in $C^2([-1,1])$. This allows us to apply the results from the first step to obtain a corresponding sequence of approximate solutions to \eqref{EQ:SYSTEM}, where now the mobility functions are replaced by their respective regularizations. Then, employing the chain-rule formula established in Proposition~\ref{App:Proposition:CR}, we derive uniform estimates for the approximate solutions with respect to the approximation parameter. Finally, standard compactness arguments yield the desired result.

\begin{proof}[Proof of Theorem~\ref{Theorem:PropReg}]

\textbf{Step 1:} We start by assuming $m_\Om,m_\Ga\in C^2([-1,1])$.
To prove the propagation of regularity, we make use of the differential inequality \eqref{Est:PreGronwall:Uniq} proven in the uniqueness part of Theorem~\ref{Theorem}. First, for brevity, we use the notation
\begin{align*}
    &\mathcal{S}_h(f,g) = \mathcal{S}_{L}[\phi(\cdot+h),\psi(\cdot+h)](f,g), \\
    &\norm{\cdot}_{L,h,\ast} = \norm{\cdot}_{L,[\phi(\cdot+h),\psi(\cdot+h)],\ast}.
\end{align*}
Then for any $h\in(0,1)$, we apply the estimate \eqref{Est:PreGronwall:Uniq} to $(\phi_1,\psi_1) = (\phi,\psi)$ and $(\phi_2,\psi_2) = (\phi(\cdot + h),\psi(\cdot + h))$. Dividing the resulting inequality by $h^2$, and denoting the difference quotient in time of a function $f$ by $\delth f(t) = \frac1h\big(f(t + h) - f(t)\big)$, we obtain
\begin{align}\label{PreGronwall:G:Delth}
    \begin{split}
    	&\ddt \frac12\norm{(\delth\phi(t),\delth\psi(t))}_{L,h,\ast}^2 + \frac12\norm{(\delth\phi(t),\delth\psi(t))}_{K}^2 \leq Q_h(t)\norm{(\delth\phi(t),\delth\psi(t))}_{L,h,\ast}^2
    \end{split}
\end{align}
where, now,
\begin{align*}
	Q_h(\cdot) = C\Big(1 + \norm{(\mu(\cdot + h),\theta\cdot + h))}_{L}^2 + \norm{(\delt\phi,\delt\psi)}_{(\mathcal{H}^1_L)^\prime}^2 + \norm{(\phi,\psi)}_{\mathcal{H}^2}^4\Big)\in L^1(0,T)
\end{align*}
for any $T > 0$.
To apply the uniform Gronwall lemma~\ref{Lemma:Gronwall}, we note that
\begin{align}\label{Est:UniformGronwalL:Pre}
	\sup_{t\geq 0}\int_t^{t+1}\norm{(\delth\phi(s),\delth\psi(s))}_{L,h,\ast}^2\ds \leq C_0, \qquad \sup_{t\geq 0}\int_t^{t+1} Q_h(s)\ds \leq C_1,
\end{align}
where the constants $C_0, C_1 > 0$ solely depend on $E(\phi_0,\psi_0)$, $\mean{\phi_0}{\psi_0}$, and the parameters of the system. Let $\tau > 0$. Then, an application of Lemma~\ref{Lemma:Gronwall} with $t_0 = 0$ and $r = \tau$ yields
\begin{align}\label{Appl:GronwallUniform}
	&\norm{(\delth\phi(t),		\delth\psi(t))}_{L,h,\ast}^2 \leq \frac{C_0}{\tau}\exp(C_1) \qquad\text{for all~}t \geq\tau.
\end{align}
Recalling again the equivalence of the corresponding norms on $\mathcal{V}^{-1}_{L}$, we readily obtain
    \begin{align}\label{Est:HighReg:Delth:pp:Prime}
    &\norm{(\delth\phi(t),\delth\psi(t))}_{(\mathcal{H}^1_L)^\prime}^2\leq \frac{CC_0}{\tau}\exp(C_1) \qquad\text{for all~}t\geq\tau,
\end{align}
where the constant $C > 0$ depends solely on the parameters of the system. This allows us to pass to the limit $h\rightarrow 0$ in \eqref{Est:HighReg:Delth:pp:Prime} to deduce that
\begin{align}\label{Est:HighReg:Delt:pp:Prime}
    \sup_{t\geq\tau}\norm{(\delt\phi(t),\delt\psi(t))}_{(\mathcal{H}^1_L)^\prime}^2 \leq \frac{CC_0}{\tau}\exp(C_1).
\end{align}
Then, we integrate \eqref{PreGronwall:G:Delth} over the time interval $[t,t+1]$, and employ the estimates \eqref{Est:UniformGronwalL:Pre} and \eqref{Est:HighReg:Delt:pp:Prime}. Passing again to the limit $h\rightarrow 0$ in the resulting estimate yields
\begin{align}\label{Est:HighReg:pp:Ka}
    \sup_{t\geq\tau}\int_t^{t+1} \norm{(\delt\phi(s),\delt\psi(s))}_{K}^2\ds \leq \frac{(1+C_1)CC_0}{\tau}\exp(C_1).
\end{align} 
Next, testing \eqref{WF:PP:SING} with $\mathcal{S}_{L}[\phi,\psi](\delt\phi,\delt\psi)$ and noting on the identities \eqref{Id:M:SolOp:Mean}-\eqref{Id:T:SolOp:Mean}, an application of the bulk-surface Poincar\'{e} inequality together with Young's inequality shows that
\begin{align*}
    \norm{(\mu,\theta)}_{L} \leq C\norm{(\delt\phi,\delt\psi)}_{(\mathcal{H}^1_L)^\prime},
\end{align*}
which proves that
\begin{align*}
	t\mapsto \norm{(\mu(t),\theta(t))}_{L} \in L^\infty(\tau,\infty)
\end{align*}
in view of \eqref{Est:HighReg:Delt:pp:Prime}.
Thus, in light of \eqref{Est:MEAN:MT:DELN}, we use the bulk-surface Poincar\'{e} inequality to obtain that
\begin{align}\label{Est:HighReg:MT:H1}
	\norm{(\mu,\theta)}_{L^\infty(\tau,\infty;\mathcal{H}^1)}\leq C.
\end{align}
By \eqref{Est:PP:Pot:L^p}we learn that
\begin{align}\label{Est:HighReg:PP:Pot:p}
    \norm{(\phi,\psi)}_{L^\infty(\tau,\infty;\mathcal{W}^{2,p})} + \norm{(F_1^\prime(\phi),G_1^\prime(\psi))}_{L^\infty(\tau,\infty;\mathcal{L}^p)} \leq C
\end{align}
for all $2 \leq p < \infty$. To derive higher regularity estimates for the chemical potentials, we first recall that
\begin{align*}
    &\norm{(\mu - \beta\mean{\mu}{\theta},\theta - \mean{\mu}{\theta})}_{\mathcal{H}^2} \\
    &\quad = \norm{\mathcal{S}_{L}[\phi,\psi](\delt\phi,\delt\psi)}_{\mathcal{H}^2} \\
    &\quad\leq C\Big(\norm{(\Grad\phi,\Gradg\psi)}_{\mathcal{L}^2}\norm{(\phi,\psi)}_{\mathcal{H}^2}\norm{\mathcal{S}_{L}[\phi,\psi](\delt\phi,\delt\psi)}_{L} + \norm{(\mu,\theta)}_{L}^{\frac12}\norm{(\delt\phi,\delt\psi)}_{K}^{\frac12}\Big).
\end{align*}
Thus, thanks to \eqref{Est:HighReg:Delt:pp:Prime}, \eqref{Est:HighReg:MT:H1} and \eqref{Est:HighReg:PP:Pot:p}, we have
\begin{align*}
    \norm{(\mu - \beta\mean{\mu}{\theta},\theta - \mean{\mu}{\theta})}_{\mathcal{H}^2} \leq C\Big(1 + \norm{(\delt\phi,\delt\psi)}_{K}^{\frac12}\Big)
\end{align*}
a.e. on $(\tau,\infty)$. Therefore, we infer with \eqref{Est:HighReg:pp:Ka} that
\begin{align}\label{Est:HighReg:MT:mean:H^2}
    \sup_{t\geq\tau}\int_t^{t+1}\norm{(\mu - \mean{\mu}{\theta},\theta - \mean{\mu}{\theta})}_{\mathcal{H}^2}^4 \ds \leq C.
\end{align}
In light of \eqref{Est:MEAN:MT:DELN}, the latter implies that $(\mu,\theta)\in L^4_{\mathrm{uloc}}([\tau,\infty);\mathcal{H}^2)$.

Furthermore, exploiting the additional assumption $m_\Om,m_\Ga\in C^2([-1,1])$, we can use Corollary~\ref{Corollary:BSE:H3} to find that
\begin{align}\label{Est:MT:H^3:mean}
    \begin{split}
        &\norm{(\mu - \beta\mean{\mu}{\theta},\theta - \mean{\mu}{\theta})}_{\mathcal{H}^3} \\
        &\quad\leq C\Bigg( \Bignorm{\bigg(\frac{\delt\phi}{m_\Om(\phi)},\frac{\delt\psi}{m_\Ga(\psi)}\bigg)}_{\mathcal{H}^1} + \Bignorm{\bigg(\frac{m_\Om^\prime(\phi)\Grad\phi\cdot\Grad\mu}{m_\Om(\phi)},\frac{m_\Ga^\prime(\psi)\Gradg\psi\cdot\Gradg\theta}{m_\Ga(\psi)}\bigg)}_{\mathcal{H}^1}\Bigg).
    \end{split}
\end{align}
By standard computations, we estimate the terms on the right-hand side of \eqref{Est:MT:H^3:mean} as
\begin{align*}
    &\Bignorm{\bigg(\frac{\delt\phi}{m_\Om(\phi)},\frac{\delt\psi}{m_\Ga(\psi)}\bigg)}_{\mathcal{H}^1} \\
    &\quad\leq \Bignorm{\bigg(\frac{\delt\phi}{m_\Om(\phi)},\frac{\delt\psi}{m_\Ga(\psi)}\bigg)}_{\mathcal{L}^2} + \Bignorm{\bigg(\frac{m_\Om^\prime(\phi)\delt\phi\Grad\phi}{m_\Om(\phi)^2},\frac{m_\Ga^\prime(\psi)\delt\psi\Gradg\psi}{m_\Ga(\psi)^2}\bigg)}_{\mathcal{L}^2} + \Bignorm{\bigg(\frac{\Grad\delt\phi}{m_\Om(\phi)},\frac{\Gradg\delt\psi}{m_\Ga(\psi)}\bigg)}_{\mathcal{L}^2} \\
    &\quad\leq C\norm{(\delt\phi,\delt\psi)}_{\mathcal{L}^2} + C\norm{(\Grad\phi,\Gradg\psi)}_{\mathcal{L}^\infty}\norm{(\delt\phi,\delt\psi)}_{\mathcal{L}^2} + C\norm{(\delt\phi,\delt\psi)}_{K}
\end{align*}
and
\begin{align*}
    &\Bignorm{\bigg(\frac{m_\Om^\prime(\phi)\Grad\phi\cdot\Grad\mu}{m_\Om(\phi)},\frac{m_\Ga^\prime(\psi)\Gradg\psi\cdot\Gradg\theta}{m_\Ga(\psi)}\bigg)}_{\mathcal{H}^1} \\
    &\quad\leq \Bignorm{\bigg(\frac{m_\Om^\prime(\phi)\Grad\phi\cdot\Grad\mu}{m_\Om(\phi)},\frac{m_\Ga^\prime(\psi)\Gradg\psi\cdot\Gradg\theta}{m_\Ga(\psi)}\bigg)}_{\mathcal{L}^2} \\
    &\qquad + \Bignorm{\bigg(\frac{\big(m_\Om^{\prime\prime}(\phi)m_\Om(\phi) - m_\Om^\prime(\phi)^2\big)\big(\Grad\phi\cdot\Grad\mu\big)\Grad\phi}{m_\Om(\phi)^2},\frac{\big(m_\Ga^{\prime\prime}(\psi)m_\Ga(\psi) - m_\Ga^\prime(\psi)^2\big)\big(\Gradg\psi\cdot\Gradg\theta\big)\Gradg\psi}{m_\Ga(\psi)^2}\bigg)}_{\mathcal{L}^2} \\
    &\qquad + \Bignorm{\bigg(\frac{m_\Om^\prime(\phi)D^2\phi\Grad\mu}{m_\Om(\phi)},\frac{m_\Ga^\prime(\psi)D_\Ga^2\psi\Gradg\theta}{m_\Ga(\psi)}\bigg)}_{\mathcal{L}^2} + \Bignorm{\bigg(\frac{m_\Om^\prime(\phi)D^2\mu\Grad\phi}{m_\Om(\phi)},\frac{m_\Ga^\prime(\psi)D_\Ga^2\theta\Gradg\psi}{m_\Ga(\psi)}\bigg)}_{\mathcal{L}^2} \\
    &\quad\leq C\norm{(\Grad\phi,\Gradg\psi)}_{\mathcal{L}^\infty}\norm{(\Grad\mu,\Gradg\theta)}_{\mathcal{L}^2} + C\norm{(\Grad\phi,\Gradg\psi)}_{\mathcal{L}^\infty}^2\norm{(\Grad\mu,\Gradg\theta)}_{\mathcal{L}^2} \\
    &\qquad + C\norm{(\phi,\psi)}_{\mathcal{W}^{2,4}}\norm{(\Grad\mu,\Gradg\theta)}_{\mathcal{L}^4} + C\norm{(\Grad\phi,\Gradg\psi)}_{\mathcal{L}^\infty}\norm{(\mu - \beta\mean{\mu}{\theta},\theta - \mean{\mu}{\theta})}_{\mathcal{H}^2}.
\end{align*}
Recalling the Sobolev embedding $\mathcal{W}^{2,3}\emb\mathcal{W}^{1,\infty}$, and exploiting \eqref{Est:HighReg:PP:Pot:p}, we arrive at
\begin{align*}
    &\norm{(\mu - \beta\mean{\mu}{\theta},\theta - \mean{\mu}{\theta})}_{\mathcal{H}^3} \\
    &\quad\leq C\Big(1 + \norm{(\mu - \beta\mean{\mu}{\theta},\theta - \mean{\mu}{\theta})}_{\mathcal{H}^2} + \norm{(\delt\phi,\delt\psi)}_{K}\Big).
\end{align*}
Hence, by \eqref{Est:HighReg:pp:Ka} and \eqref{Est:HighReg:MT:mean:H^2}, we conclude that
\begin{align*}
    \sup_{t\geq\tau}\int_t^{t+1}\norm{(\mu - \beta\mean{\mu}{\theta},\theta - \mean{\mu}{\theta})}_{\mathcal{H}^3}^2\ds \leq C.
\end{align*}
In light of \eqref{Est:MEAN:MT:DELN}, the latter entails that $(\mu,\theta)\in L^2_{\mathrm{uloc}}([\tau,\infty);\mathcal{H}^3)$.

\noindent
\textbf{Step 2:} Now, let $m_\Om,m_\Ga\in C^1([-1,1])$. Then, we can construct a sequence $(m_{\Om,k})_{k\in\N}\subset C^2([-1,1])$ such that
\begin{enumerate}[label=\textbf{(M\arabic*)},topsep=0ex,leftmargin=*]
    \item\label{M1} $0 < \frac{m^\ast}{2} \leq m_{\Om,k}(s) \leq 2M^\ast$ for all $s\in[-1,1]$ and $k\in\N$;
    \item\label{M2} $m_{\Om,k}\rightarrow m_\Om$ in $C^1([-1,1])$ as $k\rightarrow\infty$;
    \item\label{M3} $\vert m_{\Om,k}^\prime(s)\vert \leq C_{\mathrm{mob}}\norm{m_\Om^\prime}_{L^\infty([-1,1])}$ for all $s\in[-1,1]$ and $k\in\N$,
\end{enumerate}
where $m_\ast$ and $M_\ast$ are the constants from \ref{Ass:Mobility}, and the constant $C_{\mathrm{mob}} > 0$ does not depend on $k\in\N$.
Similarly, we can construct a sequence $(m_{\Ga,k})_{k\in\N}\subset C^2([-1,1])$ that satisfies analogous properties. Considering the corresponding unique weak solution $(\phi_k,\psi_k,\mu_k,\theta_k)$ which exists according to Theorem~\ref{Theorem}, the results from Step 1 show that for any $\tau > 0$, it holds that
\begin{align*}
    (\phi_k,\psi_k)&\in L^\infty(\tau,\infty;\mathcal{W}^{2,p}), \quad  (\delt\phi_k,\delt\psi_k)\in L^\infty(\tau,\infty;(\mathcal{H}^1_L)^\prime)\cap L^2_{\mathrm{uloc}}([\tau,\infty);\mathcal{H}^1), \\
    (\mu_k,\theta_k)&\in L^\infty(\tau,\infty;\mathcal{H}^1_L)\cap L^2_{\mathrm{uloc}}([\tau,\infty);\mathcal{H}^3), \quad (F^\prime(\phi_k),G^\prime(\psi_k))\in L^\infty(\tau,\infty;\mathcal{L}^p),
\end{align*}
for any $2 \leq p < \infty$ and all $k\in\N$. Moreover, with these regularities at hand, and exploiting the additional assumption \ref{Ass:Potentials:Sep}, one can show similarly to \cite[Theorem~2.3]{Lv2024a} and \cite[Theorem~2.2]{Lv2024b} that
\begin{align*}
    (F^{\prime\prime}(\phi_k),G^{\prime\prime}(\psi_k))\in L^\infty(\tau,\infty;\mathcal{L}^p)
\end{align*}
for any $2 \leq p < \infty$ (see also \cite[Theorem~3.11]{Knopf2025}). From this, one can deduce that the time derivative $(\delt\mu_k,\delt\theta_k)$ exists for all $k\in\N$ in the sense that $(\delt\mu_k,\delt\theta_k)\in L^2_{\mathrm{uloc}}([\tau,\infty);(\mathcal{H}^1)^\prime)$, and it satisfies
\begin{align}\label{WF:DELT:MT}
    \begin{split}
        \bigang{(\delt\mu_k,\delt\theta_k)}{(\eta,\vartheta)}_{\mathcal{H}^1} &= \intO \Grad\delt\phi_k\cdot\Grad\eta + F^{\prime\prime}(\phi_k)\delt\phi_k\eta\dx \\
        &\quad + \intG \Gradg\delt\psi_k\cdot\Gradg\vartheta + G^{\prime\prime}(\psi_k)\delt\psi_k\vartheta\dG \\
        &\quad + \chi(K) \intG (\alpha\delt\psi_k - \delt\phi_k)(\alpha\psi_k - \phi_k)\dG
    \end{split}
\end{align}
a.e. on $(\tau,\infty)$ for all $(\eta,\vartheta)\in\mathcal{H}^1$. For more details, we refer to \cite{Giorgini2025}. In the following, the letter $C$ will denote generic positive constants that are independent of $k$.

As shown at the beginning of the proof of Theorem~\ref{Theorem}, one readily deduces from the energy inequality \eqref{WEDL:SING} the following uniform estimates
\begin{align}
    &\sup_{t\geq 0}\norm{(\phi_k(t),\psi_k(t))}_{\mathcal{H}^1} \leq C, \\
    &\int_0^\infty \norm{(\mu_k,\theta_k)}_{L}^2 + \norm{(\delt\phi_k,\delt\psi_k)}_{(\mathcal{H}^1_L)^\prime}^2\ds \leq C, \label{Est:Low:MT:DELT:PP:k}\\
    &\sup_{t\geq0} \int_t^{t+1}\norm{(\phi_k,\psi_k)}_{\mathcal{H}^2}^4\ds \leq C \label{Est:Low:PP:L4H2:k}
\end{align}
as well as
\begin{align*}
    \norm{(\phi_k,\psi_k)}_{\mathcal{W}^{2,p}} + \norm{(F_1^\prime(\phi_k),G_1^\prime(\psi_k))}_{\mathcal{L}^p} \leq C\big(1 + \norm{(\mu_k,\theta_k)}_{L}\big)
\end{align*}
a.e. on $(0,\infty)$. To establish the main estimates for the regularity argument, we aim to apply Proposition~\ref{App:Proposition:CR}. To this end, noting on
\begin{align*}
    \big(\Div(m_{\Om,k}(\phi_k)\Grad\mu_k),\Divg(m_{\Ga,k}(\psi_k)\Gradg\theta_k) - \beta m_{\Om,k}(\phi_k)\deln\mu_k\big) = (\delt\phi_k,\delt\psi_k) \in L^2_{\mathrm{uloc}}([\tau,\infty);\mathcal{H}^1),
\end{align*}
the aforementioned proposition yields that
\begin{align*}
    &\ddt\frac12\Big(\intO m_{\Om,k}(\phi_k)\abs{\Grad\mu_k}^2\dx + \intG m_{\Ga,k}(\psi_k)\abs{\Gradg\theta_k}^2\dG + \chi(L)\intG (\beta\theta_k - \mu_k)^2\dG \Big) \nonumber \\
    &\quad = \bigang{(\delt\mu_k,\delt\theta_k)}{(-\Div(m_{\Om,k}(\phi_k)\Grad\mu_k),-\Divg(m_{\Ga,k}(\psi_k)\Gradg\theta_k) + \beta m_{\Om,k}(\phi_k)\deln\mu_k)}_{\mathcal{H}^1} \nonumber \\
    &\qquad + \intO m_{\Om,k}^\prime(\phi_k)\delt\phi_k\abs{\Grad\mu_k}^2\dx + \intG m_{\Ga,k}^\prime(\psi_k)\delt\psi_k\abs{\Gradg\theta_k}^2\dG \nonumber \\
    &\quad = -\bigang{(\delt\mu_k,\delt\theta_k)}{(\delt\phi_k,\delt\psi_k)}_{\mathcal{H}^1} \\
    &\qquad + \intO m_{\Om,k}^\prime(\phi_k)\delt\phi_k\abs{\Grad\mu_k}^2\dx + \intG m_{\Ga,k}^\prime(\psi_k)\delt\psi_k\abs{\Gradg\theta_k}^2\dG \nonumber \\
    &\quad = - \norm{(\delt\phi_k,\delt\psi_k)}_{K}^2 - \intO F^{\prime\prime}(\phi_k)\abs{\delt\phi_k}^2\dx - \intG G^{\prime\prime}(\psi_k)\abs{\delt\psi_k}^2\dG \nonumber \\
    &\qquad + \intO m_{\Om,k}^\prime(\phi_k)\delt\phi_k\abs{\Grad\mu_k}^2\dx + \intG m_{\Ga,k}^\prime(\psi_k)\delt\psi_k\abs{\Gradg\theta_k}^2\dG \nonumber
\end{align*}
a.e. on $(\tau,\infty)$. Exploiting the strong convexity of $F_1$ and $G_1$, respectively, as well as the Lipschitz continuity of $F_2^\prime$ and $G_2^\prime$, respectively, we readily infer
\begin{align}\label{Est:PreGronwall:HighReg:k:1}
    \begin{split}
        &\ddt\frac12\Big(\intO m_{\Om,k}(\phi_k)\abs{\Grad\mu_k}^2\dx + \intG m_{\Ga,k}(\psi_k)\abs{\Gradg\theta_k}^2\dG + \chi(L)\intG (\beta\theta_k - \mu_k)^2\dG \Big) \\
        &\qquad + \norm{(\delt\phi_k,\delt\psi_k)}_{K}^2\\
        &\quad \leq C\norm{(\delt\phi_k,\delt\psi_k)}_{\mathcal{L}^2}^2 + \intO m_{\Om,k}^\prime(\phi_k)\delt\phi_k\abs{\Grad\mu_k}^2\dx + \intG m_{\Ga,k}^\prime(\psi_k)\delt\psi_k\abs{\Gradg\theta_k}^2\dG.
    \end{split}
\end{align}

For the first term on the right-hand side of \eqref{Est:PreGronwall:HighReg:k:1}, if $K\in[0,\infty)$, we employ Ehrling's lemma and \eqref{Est:Sol:Delt:pp:Lb} to find that
\begin{align}\label{Est:delt:pp:k:Ehrling}
     C\norm{(\delt\phi_k,\delt\psi_k)}_{\mathcal{L}^2}^2 \leq \frac14\norm{(\delt\phi_k,\delt\psi_k)}_{K}^2 + C\norm{(\mu_k,\theta_k)}_{L}^2.
\end{align}
On the other hand, if $K = \infty$, Ehrling's lemma yields that, for any $\varepsilon > 0$,
\begin{align*}
    C\norm{(\delt\phi_k,\delt\psi_k)}_{\mathcal{L}^2}^2 &\leq \varepsilon\norm{(\delt\phi_k,\delt\psi_k)}_{\mathcal{H}^1}^2 + C_\varepsilon\norm{(\mu_k,\theta_k)}_{L}^2 \\
    &\leq \varepsilon\norm{(\delt\phi_k,\delt\psi_k)}_{\mathcal{L}^2}^2 + \varepsilon\norm{(\delt\phi_k,\delt\psi_k)}_{K}^2 + C_\varepsilon\norm{(\mu_k,\theta_k)}_{L}^2.
\end{align*}
Consequently, choosing $\varepsilon = \frac{C}{5}$, we conclude that \eqref{Est:delt:pp:k:Ehrling} also holds in the case $K = \infty$.

Next, to control the remaining terms on the right-hand side of \eqref{Est:PreGronwall:HighReg:k:1}, we use \eqref{Prelim:Est:Inteprol} and \ref{M3}, obtaining
\begin{align}\label{Est:jonas:1}
    \begin{split}
        &\abs{\intO m_{\Om,k}^\prime(\phi_k)\delt\phi_k\abs{\Grad\mu_k}^2\dx + \intG m_{\Ga,k}^\prime(\psi_k)\delt\psi_k\abs{\Gradg\theta_k}^2\dG} \\
        &\quad\leq C\norm{(\delt\phi_k,\delt\psi_k)}_{\mathcal{L}^2}\norm{(\Grad\mu_k,\Gradg\theta_k)}_{\mathcal{L}^4}^2 \\
        &\quad\leq C\norm{(\delt\phi_k,\delt\psi_k)}_{\mathcal{L}^2}\norm{(\mu_k,\theta_k)}_{L}\norm{(\mu_k - \beta\mean{\mu_k}{\theta_k},\theta_k - \mean{\mu_k}{\theta_k})}_{\mathcal{H}^2}.
    \end{split}
\end{align}
Arguing as previously done in Step 1, we deduce with \eqref{Est:fg:L^2:SolOp:2}, \eqref{Est:Sol:G:H^2}, \eqref{Est:PP:H1:Linfty}, and \eqref{Est:Sol:Delt:pp:Lb} that
\begin{align*}
    &\norm{(\mu_k - \beta\mean{\mu_k}{\theta_k},\theta_k - \mean{\mu_k}{\theta_k})}_{\mathcal{H}^2} \\
    &\quad = \norm{\mathcal{S}_{L}[\phi_k,\psi_k](\delt\phi_k,\delt\psi_k)}_{\mathcal{H}^2} \\
    &\quad\leq C\Big(\norm{(\Grad\phi_k,\Gradg\psi_k)}_{\mathcal{L}^2}\norm{(\phi_k,\psi_k)}_{\mathcal{H}^2}\norm{\mathcal{S}_{L}[\phi_k,\psi_k](\delt\phi_k,\delt\psi_k)}_{L} + \norm{(\delt\phi_k,\delt\psi_k)}_{\mathcal{L}^2}\Big) \\
    &\quad\leq C\Big(\norm{(\phi_k,\psi_k)}_{\mathcal{H}^2}\norm{(\mu_k,\theta_k)}_{L} + \norm{(\mu_k,\theta_k)}_{L}^{\frac12}\norm{(\delt\phi_k,\delt\psi_k)}_{K}^{\frac12}\Big).
\end{align*}
Consequently, we infer with Young's inequality that
\begin{align}\label{Est:jonas:2}
    &\abs{\intO m_{\Om,k}^\prime(\phi_k)\delt\phi_k\abs{\Grad\mu_k}^2\dx + \intG m_{\Ga,k}^\prime(\psi_k)\delt\psi_k\abs{\Gradg\theta_k}^2\dG} \nonumber \\
    &\quad\leq C\norm{(\delt\phi_k,\delt\psi_k)}_{\mathcal{L}^2}\norm{(\mu_k,\theta_k)}_{L}\Big(\norm{(\phi_k,\psi_k)}_{\mathcal{H}^2}\norm{(\mu_k,\theta_k)}_{L} + \norm{(\mu_k,\theta_k)}_{L}^{\frac12}\norm{(\delt\phi_k,\delt\psi_k)}_{K}^{\frac12}\Big) \nonumber \\
    &\quad\leq C\norm{(\mu_k,\theta_k)}_{L}^{\frac32}\norm{(\delt\phi_k,\delt\psi_k)}_{K}^{\frac12}\Big(\norm{(\phi_k,\psi_k)}_{\mathcal{H}^2}\norm{(\mu_k,\theta_k)}_{L} + \norm{(\mu_k,\theta_k)}_{L}^{\frac12}\norm{(\delt\phi_k,\delt\psi_k)}_{K}^{\frac12}\Big) \nonumber \\
    &\quad\leq
    C\norm{(\delt\phi_k,\delt\psi_k)}_{K}^{\frac12}\norm{(\phi_k,\psi_k)}_{\mathcal{H}^2}\norm{(\mu_k,\theta_k)}_{\mathcal{H}^2}^{\frac52} + C\norm{(\delt\phi_k,\delt\psi_k)}_{K}\norm{(\mu_k,\theta_k)}_{L}^2 \\
    &\quad\leq \frac14\norm{(\delt\phi_k,\delt\psi_k)}_{K}^2 + C\norm{(\phi_k,\psi_k)}_{\mathcal{H}^2}^{\frac43}\norm{(\mu_k,\theta_k)}_{L}^{\frac{10}{3}} + C\norm{(\mu_k,\theta_k)}_{L}^4 \nonumber \\
    &\quad\leq \frac14\norm{(\delt\phi_k,\delt\psi_k)}_{K}^2 + C\Big(\norm{(\phi_k,\psi_k)}_{\mathcal{H}^2}^4 + \norm{(\mu_k,\theta_k)}_{L}^2\Big)\norm{(\mu_k,\theta_k)}_{L}^2. \nonumber
\end{align}
Thus, in view of \eqref{Est:PreGronwall:HighReg:k:1}, we obtain from \eqref{Est:jonas:1} and \eqref{Est:jonas:2} that
\begin{align}\label{Est:PreGronwall:HighReg:k:2}
    \begin{split}
        &\ddt\frac12\Big(\intO m_{\Om,k}(\phi_k)\abs{\Grad\mu_k}^2\dx + \intG m_{\Ga,k}(\psi_k)\abs{\Gradg\theta_k}^2\dG + \chi(L)\intG (\beta\theta_k - \mu_k)^2\dG \Big) \\ 
        &\quad + \frac12\norm{(\delt\phi_k,\delt\psi_k)}_{K}^2 \\
        &\quad\leq C\Big(1 + \norm{(\phi_k,\psi_k)}_{\mathcal{H}^2}^4 + \norm{(\mu_k,\theta_k)}_{L}^2\Big) \\
        &\qquad\times\Big(\intO m_{\Om,k}(\phi_k)\abs{\Grad\mu_k}^2\dx + \intG m_{\Ga,k}(\psi_k)\abs{\Gradg\theta_k}^2\dG + \chi(L)\intG (\beta\theta_k - \mu_k)^2\dG \Big)
    \end{split}
\end{align}
a.e. on $(\tau,\infty)$. In light of \eqref{Est:Low:MT:DELT:PP:k} and \eqref{Est:Low:PP:L4H2:k}, we may now apply the uniform Gronwall Lemma~\ref{Appl:GronwallUniform} and deduce with \ref{M1} that
\begin{align*}
    \sup_{t\geq\tau} \norm{(\mu_k(t),\theta_k(t))}_{L}^2 \leq \frac{C}{\tau}.
\end{align*}
Integrating \eqref{Est:PreGronwall:HighReg:k:2} in time over $[t,t+1]$ for $t\geq\tau$, we find
\begin{align*}
    \sup_{t\geq\tau}\int_t^{t+1}\norm{(\delt\phi_k(s),\delt\psi_k(s))}_{K}^2\ds \leq \frac{C}{\tau}.
\end{align*}
Based on these bounds, we can follow along the lines of Step 1 and conclude that
\begin{align*}
    &\norm{(\mu_k,\theta_k)}_{L^\infty(\tau,\infty;\mathcal{H}^1_L)} + \norm{(\phi_k,\psi_k)}_{L^\infty(\tau,\infty;\mathcal{W}^{2,p})} + \norm{(F_1^\prime(\phi_k),G_1^\prime(\psi_k))}_{L^\infty(\tau,\infty;\mathcal{L}^p)} \\
    &\quad + \norm{(\mu_k,\theta_k)}_{L^4_{\mathrm{uloc}}([\tau,\infty);\mathcal{H}^2)} \leq C
\end{align*}
for all $2 \leq p < \infty$. These estimates ensure, by standard compactness arguments, the existence of a limit quadruple $(\phi,\psi,\mu,\theta)$ solving \eqref{EQ:SYSTEM} in the sense of Definition~\ref{DEF:SING:WS}, and satisfying
\begin{align*}
    &(\phi,\psi)\in L^\infty(\tau,\infty;\mathcal{W}^{2,p}), \\
    &(\delt\phi,\delt\psi)\in L^\infty(\tau,\infty;(\mathcal{H}^1_L)^\prime)\cap L^2_{\mathrm{uloc}}([\tau,\infty);\mathcal{H}^1), \\
    &(\mu,\theta)\in L^\infty(\tau,\infty;\mathcal{H}^1_L)\cap L^4_{\mathrm{uloc}}([\tau,\infty);\mathcal{H}^2), \\
    &(F^\prime(\phi),G^\prime(\psi))\in L^\infty(\tau,\infty;\mathcal{L}^p)
\end{align*}
for any $2 \leq p < \infty$. This finishes the proof.
\end{proof}

As a consequence of Theorem~\ref{Theorem:PropReg}, we can improve the energy inequality \eqref{WEDL:SING} to be an energy equality.
\begin{proposition}
    Suppose the assumptions from Theorem~\ref{Theorem:PropReg} hold, and consider a global weak solution $(\phi,\psi,\mu,\theta)$ that satisfies the propagation of regularity. Then
    \begin{align}\label{Id:ddt:Energy}
        \ddt E(\phi(t),\psi(t)) + \intO m_\Om(\phi)\abs{\Grad\mu}^2\dx + \intG m_\Ga(\psi)\abs{\Gradg\theta}^2\dG + \chi(L) \intG (\beta\theta - \mu)^2\dG = 0
    \end{align}
    for a.e. $t > 0$, and
    \begin{align}\label{Id:Energy:Strong}
        \begin{split}
            &E(\phi(t),\psi(t)) + \int_0^t\intO m_\Om(\phi)\abs{\Grad\mu}^2\dx\ds + \int_0^t\intG m_\Ga(\psi)\abs{\Gradg\theta}^2\dG\ds \\
            &\quad + \chi(L)\int_0^t\intG (\beta\theta - \mu)^2\dG\ds = E(\phi_0,\psi_0) 
        \end{split}
    \end{align}
    for all $t\geq 0$.
\end{proposition}

\begin{proof}
	We start by defining the functional $E_0:\mathcal{L}				^2\rightarrow(-\infty,\infty]$ given by
	\begin{align*}
		E_0(\zeta,\xi) \coloneqq \frac12\norm{(\zeta,\xi)}_K^2 + \intO F_1(\zeta)\dx + \intG G_1(\xi)\dG,
	\end{align*}
	where $F_1$ and $G_1$ are the convex parts of the potentials $F$ and $G$ in \ref{Ass:Potentials}, respectively. Then, the functional $E_0$ is proper, convex, and lower-semicontinuous (see, e.g., \cite[Lemma~5.1]{Giorgini2025}). Owing to \cite[Lemma~4.1]{Rocca2004} in combination with \cite[Proposition~5.2]{Giorgini2025} we deduce that $[0,\infty)\ni t\mapsto 				E_0(\phi(t),\psi(t))$ is absolutely continuous and it holds
	\begin{align*}
		\ddt E_0(\phi,\psi) &= \bigang{(\delt\phi,\delt\psi)}{(-\Lap\phi + F_1^\prime(\phi), -\Lapg\psi + G_1^\prime(\psi) + \alpha\deln\phi)}_{\mathcal{H}^1_L} \\
		&= \bigang{(\delt\phi,\delt\psi)}{(\mu - F_2^\prime(\phi), 			\theta - G_2^\prime(\psi))}_{\mathcal{H}^1_L} \\
		&= \intO m_\Om(\phi)\abs{\Grad\mu}^2\dx + \intG m_\Ga(\psi)\abs{\Gradg\theta}^2\dG \\
		&\quad +\chi(L)\intG (\beta\theta - \mu)^2\dG - \intO F_2^\prime(\phi)\delt\phi\dx - \intG G_2^\prime(\psi)\delt\psi\dG
	\end{align*}
	almost everywhere in $(\tau,\infty)$ for any $\tau > 0$. Consequently,
	\begin{align}\label{Id:Energy:tau}
		\ddt E(\phi,\psi) = \intO m_\Om(\phi)\abs{\Grad\mu}^2\dx + \intG m_\Ga(\psi)\abs{\Gradg\theta}^2\dG + \chi(L)\intG (\beta\theta - \mu)^2\dG
	\end{align}
	almost everywhere in $(\tau,\infty)$. Since $\tau > 0$ was arbitrary, we readily obtain \eqref{Id:ddt:Energy}. Then, integrating \eqref{Id:Energy:tau} over $(s,t)$ for $s,t > \tau$ with $s \leq t$, we infer
	\begin{align}\label{EnergyID:t-s}
        \begin{split}
    		&E(\phi(t),\psi(t)) - E(\phi(s),\psi(s)) \\
            &\quad = \int_s^t \intO m_\Om(\phi)\abs{\Grad\mu}^2\dxs + \int_s^t\intG m_\Ga(\psi)\abs{\Gradg\theta}^2\dGs + \chi(L)\int_s^t\intG(\beta\theta - \mu)^2\dGs.
        \end{split}
	\end{align}
    It follows from \eqref{WEDL:SING} that $\limsup_{s\rightarrow 0}E(\phi(s),\psi(s)) \leq E(\phi(0),\psi(0))$. On the other hand, by weak lower-semicontinuity of norms and Lebesgue's dominated convergence theorem, we have $\liminf_{s\rightarrow 0}E(\phi(s),\psi(s)) \\ \geq E(\phi(0),\psi(0))$. As a result, it holds that $\lim_{s\rightarrow 0} E(\phi(s),\psi(s)) = E(\phi(0),\psi(0))$. This allows us to pass to the limit $s\rightarrow 0$ in \eqref{EnergyID:t-s} and conclude the energy identity \eqref{Id:Energy:Strong}. 
\end{proof}

Now, we are in a position to present the proof of Theorem~\ref{Theorem:Separation}. 

\begin{proof}[Proof of Theorem~\ref{Theorem:Separation}.]
    \textbf{The case with~\ref{Ass:Potentials:Sep:1}.} Let $(\phi,\psi,\mu,\theta)$ be a weak solution to \eqref{EQ:SYSTEM} that exhibits the propagation of regularity. Then, we have already seen in the proof of Theorem~\ref{Theorem:PropReg} that
    \begin{align*}
        F_1^{\prime\prime}(\phi)\in L^\infty(\tau,\infty;L^p(\Om))
    \end{align*}
    for any $2 \leq p < \infty$ (see, e.g., \cite[Theorem~3.11]{Knopf2025}). Analogously, one can show that
    \begin{align*}
        F_1^{\prime\prime}(\psi)\in L^\infty(\tau,\infty;L^p(\Ga))
    \end{align*}
    for all $2 \leq p < \infty$ (see, e.g., \cite[Theorem~2.2]{Lv2024b}). On the other hand, since $(\phi,\psi)\in L^\infty(\tau,\infty;\mathcal{W}^{2,p})$ and $(F_1^\prime(\phi),F_1^\prime(\psi))\in L^\infty(\tau,\infty;\mathcal{L}^p)$ for any $2 \leq p < \infty$, it holds that
    \begin{align*}
        \sup_{t\geq\tau} \norm{F_1^\prime(\phi(t))}_{W^{1,3}(\Om)} + \sup_{t\geq\tau}\norm{F_1^\prime(\psi(t))}_{W^{1,3}(\Ga)} \leq C.
    \end{align*}
    As $d = 2$, we have the Sobolev embeddings $W^{1,3}(\Om)\emb C(\overline\Om)$ and $W^{1,3}(\Ga)\emb C(\Ga)$, and deduce that
    \begin{align*}
        \sup_{t\geq\tau} \norm{F_1^\prime(\phi(t))}_{L^\infty(\Om)} + \sup_{t\geq\tau}\norm{F_1^\prime(\psi(t))}_{L^\infty(\Ga)} \leq C \eqqcolon C_\ast.
    \end{align*}
    Thus, taking
    \begin{align*}
        \delta = 1 - (F_1^\prime)^{-1}(C_\ast),
    \end{align*}
    we arrive at the conclusion \eqref{Separation:tau}.

    \textbf{The case with~\ref{Ass:Potentials:Sep:2}.} 
    In this case, we exploit the dissipative structure of \eqref{EQ:SYSTEM} in combination with a De Giorgi-type iteration scheme, following the recent approach developed in \cite{Gal2025}. This method has already been successfully adapted to dynamic boundary conditions in the case of constant mobility functions, see, for instance, \cite{Lv2024a, Lv2024b}. The extension to the setting with a non-degenerate mobility is straightforward and follows along similar lines. For brevity, we omit the details and refer the interested reader to the aforementioned works.
\end{proof}

As a direct consequence of the separation property \eqref{Separation:tau} and regularity theory for elliptic systems with bulk-surface coupling (see, e.g., \cite[Theorem~3.3]{Knopf2021}, we can prove further regularity of the phase fields.

\begin{corollary}
    Let the assumptions from Theorem~\ref{Theorem:Separation} hold, and consider a weak solution of \eqref{EQ:SYSTEM} that satisfies the propagation of regularity. Then we have $(\phi,\psi)\in L^\infty(\tau,\infty;\mathcal{H}^3)$.
\end{corollary}

\begin{proof}
    First, recall that
    \begin{alignat*}{2}
        -\Lap\phi(t) &= \mu(t) - F^\prime(\phi(t)) &&\qquad\text{a.e.~in~}\Om, \\
        -\Lapg\psi(t) + \alpha\deln\phi(t) &= \theta(t) - G^\prime(\psi(t)) &&\qquad\text{a.e.~on~}\Ga, \\
        K\deln\phi(t) &= \alpha\psi(t) - \phi(t) &&\qquad\text{a.e.~on~}\Ga
    \end{alignat*}
    for almost all $t\geq\tau > 0$. Then, from the separation property \eqref{Separation:tau}, the Lipschitz continuity of $F_2^\prime$ and $G_2^\prime$, respectively, and the fact that $(\phi,\psi)\in L^\infty(\tau,\infty;\mathcal{W}^{2,p})$ for any $\tau > 0$, we readily deduce that
    \begin{align*}
        \sup_{t\geq\tau}\norm{(F^\prime(\phi(t)),G^\prime(\psi(t)))}_{\mathcal{H}^1} \leq C
    \end{align*}
    for any $\tau > 0$. As additionally $(\mu,\theta)\in L^\infty(\tau,\infty;\mathcal{H}^1)$, the claim follows from elliptic regularity theory (see, e.g., \cite[Theorem~3.3]{Knopf2021}).
\end{proof}

\medskip
\section{Convergence to Equilibrium}\label{Section:LongTime}

The following argument is inspired by the approach developed in \cite{Abels2007} and its subsequent extension to the Cahn--Hilliard system with dynamic boundary conditions, see, for instance, \cite{Fukao2021, Lv2024a, Lv2024b}.

Let
\begin{subequations}\label{Assumption:mean}
\begin{align}\label{Assumption:mean:L}
    m\in\R \quad\text{with}\quad \beta m,m\in(-1,1)\quad\text{if~} L\in[0,\infty),
\end{align}
and
\begin{align}\label{Assumption:mean:infty}
    m = (m_1,m_2)\in\R^2\quad\text{with}\quad m_1,m_2\in(-1,1)\quad\text{if~} L = \infty.
\end{align}
\end{subequations}
We define the phase space
\begin{align*}
    \mathcal{Z}_m^{K,L} = \{(\phi,\psi)\in\mathcal{W}^1_{K,L,m} : E(\phi,\psi) < \infty\},
\end{align*}
equipped with the metric
\begin{align*}
    \mathrm{d}_{\mathcal{Z}_m^{K,L}}\big((\phi,\psi),(\zeta,\xi)\big) &\coloneqq \norm{(\phi-\zeta,\psi-\xi)}_{K} + \bigg\vert\intO F_1(\phi) \dx - \intO F_1(\zeta)\dx\bigg\vert^{\frac12} \\
    &\quad + \bigg\vert\intG G_1(\psi)\dG - \intG G_1(\xi)\dG\bigg\vert^{\frac12} \qquad\text{for all~}(\phi,\psi), (\zeta,\xi)\in\mathcal{Z}_m^{K,L}.
\end{align*}
Thus, $\big(\mathcal{Z}_m^{K,L},\mathrm{d}_{\mathcal{Z}_m^{K,L}}\big)$ is a complete metric space. We then have the following conclusion, which follows from Theorem~\ref{Theorem} and is a straightforward extension of \cite[Proposition~4.1]{Fukao2021} (see also \cite[Proposition~4.1]{Lv2024a}).

\begin{proposition}
    Suppose that the assumptions from Theorem~\ref{Theorem:LongTime} hold.
    Then, the system \eqref{EQ:SYSTEM} defines a strongly continuous semigroup $\mathcal{S}^{K,L}:\mathcal{Z}_m^{K,L}\rightarrow\mathcal{Z}_m^{K,L}$ such that
    \begin{align*}
        \mathcal{S}^{K,L}(t)(\phi_0,\psi_0) = (\phi(t),\psi(t)) \qquad\text{for all~}t\geq 0,
    \end{align*}
    where $(\phi,\psi)$ is the unique global weak solution of \eqref{EQ:SYSTEM} subject to the initial datum $(\phi_0,\psi_0)\in\mathcal{Z}_m^{K,L}$. Moreover, $\mathcal{S}^{K,L}\in C(\mathcal{Z}_m^{K,L},\mathcal{Z}_m^{K,L})$.
\end{proposition}

Next, we define the $\omega$-limit set
\begin{align*}
    \omega^{K,L}(\phi_0, \psi_0) \coloneqq \Bigg\{(\phi_\infty, \psi_\infty)\in\mathcal{H}^2\cap\mathcal{Z}_m^{K,L} \Bigg\vert
    \begin{aligned}
        &\exists(t_n)_{n\in\N}\subset\R_{\geq 0} \text{~with~} t_n\rightarrow\infty\text{~such that~} \\
        &\mathcal{S}^{K,L}_m(t_n)(\phi_0,\psi_0)\rightarrow (\phi_\infty, \psi_\infty) \text{~in~}\mathcal{H}^2 \ \text{as~}n\rightarrow\infty
    \end{aligned}
    \Bigg\}.
\end{align*}

Let $(\phi_0,\psi_0)\in\mathcal{Z}_m^{K,L}$ and consider the unique global weak solution to \eqref{EQ:SYSTEM} departing from $(\phi_0,\psi_0)$. Then, since for any $\tau > 0$ it holds that $(\phi,\psi)\in L^\infty(\tau,\infty;\mathcal{H}^3)$ as well as $(\delt\phi,\delt\psi)\in L^2_{\mathrm{uloc}}([\tau,\infty);\mathcal{H}^1)$, we find from the Aubin--Lions--Simon lemma that $(\phi,\psi)\in C([t,t+1];\mathcal{H}^2)$ for all $t\geq\tau > 0$. Hence, it holds that
\begin{align*}
    (\phi,\psi)\in BC([\tau,\infty);\mathcal{H}^s)
\end{align*}
for any $s\in(2,3)$ and for any $\tau > 0$. It then follows that the $\omega$-limit set $\omega^{K,L}(\phi_0,\psi_0)$ is non-empty, compact and connected in $\mathcal{H}^2$ (see, e.g., \cite[Theorem~9.1.8]{Cazenave1998}), and we have
\begin{align}\label{CompactnessOrbit}
    \lim_{t\rightarrow\infty} \mathrm{dist}_{\mathcal{H}^2}(\mathcal{S}^{K,L}(t)(\phi_0,\psi_0),\omega^{K,L}(\phi_0,\psi_0)) = 0.
\end{align}

Additionally, as $E:\mathcal{Z}_m^{K,L}\rightarrow\R$ serves as a strict Lyapunov functional for the strongly continuous semigroup $\mathcal{S}^{K,L}$, we observe that every $(\phi_\infty,\psi_\infty)\in\omega^{K,L}(\phi_0,\psi_0)$ is a stationary point of $\{\mathcal{S}^{K,L}(t)\}_{t\geq 0}$, that is, $\mathcal{S}^{K,L}(t)(\phi_\infty,\psi_\infty) = (\phi_\infty,\psi_\infty)$ for all $t\geq 0$. Denoting the corresponding bulk and surface chemical potentials by $\mu_\infty$ and $\theta_\infty$, respectively, we find that $(\phi_\infty,\psi_\infty,\mu_\infty,\theta_\infty)$ can be regarded as a global weak solution of
\begin{subequations}
    \begin{align*}
        &\delt\phi_\infty = \Div(m_\Om(\phi_\infty)\Grad\mu_\infty) && \text{in~} \Om\times(0,\infty), \\
        &\mu = -\Lap\phi_\infty + F'(\phi_\infty)   && \text{in~} \Om\times(0,\infty), \\
        &\delt\psi_\infty = \Divg(m_\Ga(\psi_\infty)\Gradg\theta_\infty) - \beta m_\Om(\phi_\infty)\deln\mu_\infty && \text{on~} \Ga\times(0,\infty), \\
        &\theta_\infty = - \Lapg\psi_\infty + G'(\psi_\infty) + \alpha\deln\phi_\infty && \text{on~} \Ga\times(0,\infty), \\
        &\begin{cases} K\deln\phi_\infty = \alpha\psi_\infty - \phi_\infty &\text{if~} K\in (0,\infty), \\
        \deln\phi_\infty = 0 &\text{if~} K = \infty
        \end{cases} && \text{on~} \Ga\times(0,\infty), \\
        &\begin{cases} 
        L m_\Om(\phi_\infty)\deln\mu_\infty = \beta\theta_\infty - \mu_\infty &\text{if~} L\in[0,\infty), \\
        m_\Om(\phi_\infty)\deln\mu_\infty = 0 &\text{if~} L=\infty
        \end{cases} &&\text{on~} \Ga\times(0,\infty), \\
        &\phi_\infty\vert_{t=0} = \phi_0 &&\text{in~} \Om, \\
        &\psi_\infty\vert_{t=0} = \psi_0 &&\text{on~} \Ga.
    \end{align*}
\end{subequations}

In light of the regularity properties proven in Theorem~\ref{Theorem:PropReg}, $(\phi_\infty,\psi_\infty,\mu_\infty,\theta_\infty)$ is actually a strong solution to the stationary problem
\begin{subequations}\label{EQ:SYSTEM:STATIONARY}
    \begin{align}
        \label{EQ:SYSTEM:STATIONARY:1}
        &\Div(m_\Om(\phi_\infty)\Grad\mu_\infty) = 0 &&\text{in~}\Om, \\
        \label{EQ:SYSTEM:STATIONARY:2}
        &\mu_\infty = -\Lap\phi_\infty + F^\prime(\phi_\infty) &&\text{in~}\Om, \\
        \label{EQ:SYSTEM:STATIONARY:3}
        &\Divg(m_\Ga(\psi_\infty)\Gradg\theta_\infty) - \beta m_\Om(\phi_\infty)\deln\mu_\infty = 0 &&\text{on~}\Ga, \\
        \label{EQ:SYSTEM:STATIONARY:4}
        &\theta_\infty = -\Lapg\psi_\infty + G^\prime(\psi_\infty) + \alpha\deln\phi_\infty &&\text{on~}\Ga, \\
        \label{EQ:SYSTEM:STATIONARY:5}
        &\begin{cases} K\deln\phi_\infty = \alpha\psi_\infty - \phi_\infty &\text{if} \ K\in (0,\infty), \\
        \deln\phi_\infty = 0 &\text{if} \ K = \infty
        \end{cases} && \text{on~} \Ga, \\
        \label{EQ:SYSTEM:STATIONARY:6}
        &\begin{cases} 
        L m_\Om(\phi_\infty)\deln\mu_\infty = \beta\theta_\infty - \mu_\infty &\text{if~} L\in[0,\infty), \\
        m_\Om(\phi_\infty)\deln\mu_\infty = 0 &\text{if~} L=\infty
        \end{cases} &&\text{on~} \Ga.
    \end{align}
\end{subequations}
Multiplying \eqref{EQ:SYSTEM:STATIONARY:1} with $\mu_\infty$ and \eqref{EQ:SYSTEM:STATIONARY:3} with $\theta_\infty$, integrating over $\Om$ and $\Ga$, respectively, adding the resulting equations and using the boundary condition \eqref{EQ:SYSTEM:STATIONARY:6}, we obtain
\begin{align*}
    \intO m_\Om(\phi_\infty)\abs{\Grad\mu_\infty}^2\dx + \intG m_\Ga(\psi_\infty)\abs{\Gradg\theta_\infty}^2\dG + \chi(L)\intG (\beta\theta_\infty - \mu_\infty)^2\dG = 0.
\end{align*}
Thus, for all $L\in[0,\infty]$ we infer that $\mu_\infty$ and $\theta_\infty$ are both constant. Furthermore, if $L\in[0,\infty)$, we can conclude that $\beta\theta_\infty = \mu_\infty$. Then, multiplying \eqref{EQ:SYSTEM:STATIONARY:2} with $\alpha$ and integrating over $\Om$, and integrating \eqref{EQ:SYSTEM:STATIONARY:4} over $\Ga$, we get
\begin{align}\label{STAT:MT:L}
    \mu_\infty = \beta\theta_\infty = \frac{\alpha}{\alpha\beta\abs{\Om} + \abs{\Ga}}\Big(\alpha\intO F^\prime(\phi_\infty)\dx + \intG G^\prime(\psi_\infty)\dG\Big)
\end{align}
if $L\in[0,\infty)$. If $L = \infty$, we find instead
\begin{align}\label{STAT:MT:infty}
    \begin{split}
        \mu_\infty &= \frac{1}{\abs{\Om}} \Big(\intO F^\prime(\phi_\infty)\dx - \intG \deln\phi_\infty\dG \Big), \\
        \theta_\infty &= \frac{1}{\abs{\Ga}}\Big(\intG G^\prime(\psi_\infty) + \alpha\deln\phi_\infty\dG \Big).
    \end{split}
\end{align}
Consequently the stationary problem \eqref{EQ:SYSTEM:STATIONARY} reduces to
\begin{align*}
    &\mu_\infty = -\Lap\phi_\infty + F^\prime(\phi_\infty) &&\text{in~}\Om, \\
    &\theta_\infty = -\Lapg\psi_\infty + G^\prime(\psi_\infty) + \alpha\deln\phi_\infty &&\text{on~}\Ga, \\
    &\begin{cases} K\deln\phi_\infty = \alpha\psi_\infty - \phi_\infty &\text{if} \ K\in (0,\infty), \\
    \deln\phi_\infty = 0 &\text{if} \ K = \infty
    \end{cases} && \text{on~} \Ga,
\end{align*}
with $\mu_\infty$ and $\theta_\infty$ given by \eqref{STAT:MT:L} or \eqref{STAT:MT:infty} depending on the value of $L\in[0,\infty]$.

Finally, we learn from \cite[Theorems~9.2.3 and 9.2.7]{Cazenave1998} that
\begin{align}\label{E_infty}
    E_\infty = \lim_{t\rightarrow\infty} E(\phi(t),\psi(t)) \quad \text{exists, and~}\quad E(\phi_\infty,\psi_\infty) = E_\infty \quad \text{for all~}(\phi_\infty,\psi_\infty)\in\omega^{K,L}(\phi_0,\psi_0).
\end{align}

Now, to prove that the $\omega$-limit set $\omega^{K,L}(\phi_0,\psi_0)$ is a singleton, we apply the {\L}ojasiewicz--Simon approach, see, for instance, \cite{Abels2007, Fukao2021}. The main tool is the following extended {\L}ojasiewicz--Simon inequality.

\begin{lemma}[{\L}ojasiewicz--Simon inequality]\label{Lemma:LS} 
Suppose that the assumptions from Theorem~\ref{Theorem:LongTime} are satisfied. In addition, assume that $F_1, G_1$ are real analytic functions on $(-1,1)$, and $F_2, G_2$ are real analytic functions on $\R$. Let $(\phi_\infty, \psi_\infty)\in\omega^{K,L}(\phi_0,\psi_0)$. Then, there exist constants $\varpi\in (0,\frac12)$, $b > 0$, and $C>0$, such that
\begin{align}
   C \left\|\mathbf{P}_{L}\begin{pmatrix}
        -\Lap\zeta + F^\prime(\zeta) \\
        -\Lapg\xi + G^\prime(\xi) + \alpha\deln\zeta
    \end{pmatrix}\right\|_{\mathcal{L}^2} \geq \abs{E(\zeta, \xi) - E(\phi_\infty, \psi_\infty)}^{1 - \varpi}
\end{align}
for all $(\zeta,\xi)\in\mathcal{H}^2\cap\mathcal{W}^1_{K,L,m}$ satisfying $\norm{(\zeta - \phi_\infty, \xi - \psi_\infty)}_{\mathcal{H}^2}\leq b$.
Here, $\mathbf{P}_{L}$ denotes the projection of $\mathcal{L}^2$ onto
\begin{equation*}
    \begin{cases} 
        \{\scp{\phi}{\psi}\in\mathcal{L}^2 : \mean{\phi}{\psi} = 0 \}, &\text{if~} L\in[0,\infty), \\
        \{\scp{\phi}{\psi}\in\mathcal{L}^2: \meano{\phi} = \meang{\psi} = 0 \}, &\text{if~}L=\infty.
        \end{cases} 
\end{equation*}
\end{lemma}

The proof of Lemma~\ref{Lemma:LS} can be found in \cite[Lemma~5.2]{Lv2024a} in the case $L = \infty$, and can be readily adapted to our current setting (see also \cite[Lemma~5.2]{Fukao2021} and \cite[Lemma~5.2]{Lv2024b}).

\begin{proof}[Proof of Theorem~\ref{Theorem:LongTime}]

Let $(\phi_0,\psi_0)\in\mathcal{Z}_m^{K,L}$ with $m$ as in \eqref{Assumption:mean}. Then, as $\omega^{K,L}(\phi_0,\psi_0)$ is compact in $\mathcal{H}^2\cap\mathcal{W}^1_{K,L,m}$, we can cover $\omega^{K,L}(\phi_0,\psi_0)$ with finitely many open balls $\{B_j\}_{j=1,\ldots,N}$ in $\mathcal{H}^2\cap\mathcal{W}^1_{K,L,m}$ centered at $(\phi_\infty^j,\psi_\infty^j)\in\omega^{K,L}(\phi_0,\psi_0)$ with radius $b_j$, where $b_j > 0$ is the constant from Lemma~\ref{Lemma:LS} corresponding to $(\phi_\infty^j,\psi_\infty^j)$. Recalling that $E\vert_{\omega^{K,L}(\phi_0,\psi_0)} = E_\infty$ and setting $U \coloneqq \bigcup_{j=1}^N B_j$, we find  universal constants $\widetilde\varpi\in(0,\frac12)$ and $\widetilde{C} > 0$ such that
\begin{align*}
   \widetilde{C} \left\|\mathbf{P}_{L}\begin{pmatrix}
        -\Lap\zeta + F^\prime(\zeta) \\
        -\Lapg\xi + G^\prime(\xi) + \alpha\deln\zeta
    \end{pmatrix}\right\|_{\mathcal{L}^2} \geq \abs{E(\zeta, \xi) - E_\infty}^{1 - \widetilde\varpi} \qquad\text{for all~}(\zeta,\xi)\in U.
\end{align*}
Then, in light of \eqref{CompactnessOrbit}, there exists $t^\ast > 0$ such that $(\phi(t),\psi(t))\in U$ for all $t\geq t^\ast$. Thus, recalling the energy identity \eqref{Id:ddt:Energy} and setting $H(t) \coloneqq \big(E(\phi(t),\psi(t)) - E_\infty)^{\widetilde\varpi}$, we have that
\begin{align*}
    -\ddt H(t) &= -\widetilde\varpi \big(E(\phi(t),\psi(t)) - E_\infty)^{\widetilde\varpi - 1}\ddt E(\phi(t),\psi(t)) \\
    &\geq \frac{\widetilde\varpi}{\widetilde C}\frac{\norm{(\mu(t),\theta(t))}_{L,[\phi,\psi]}^2}{\Bignorm{\mathbf{P}_{L}\begin{pmatrix}
        -\Lap\phi(t) + F^\prime(\phi(t)) \\
        -\Lapg\psi(t) + G^\prime(\psi(t)) + \alpha\deln\phi(t)
    \end{pmatrix}}_{\mathcal{L}^2}} \\
    &\geq \frac{\widetilde\varpi\min\{1,m^\ast\}}{\widetilde C}\frac{\norm{(\mu(t),\theta(t))}_{L}^2}{\Bignorm{\mathbf{P}_{L}\begin{pmatrix}
        -\Lap\phi(t) + F^\prime(\phi(t)) \\
        -\Lapg\psi(t) + G^\prime(\psi(t)) + \alpha\deln\phi(t)
    \end{pmatrix}}_{\mathcal{L}^2}}
\end{align*}
for almost every $t\geq t^\ast$. By \eqref{EQ:SYSTEM:2} and \eqref{EQ:SYSTEM:4} we deduce that
\begin{align*}
    &\Bignorm{\mathbf{P}_{L}\begin{pmatrix}
        -\Lap\phi(t) + F^\prime(\phi(t)) \\
        -\Lapg\psi(t) + G^\prime(\psi(t)) + \alpha\deln\phi(t)
    \end{pmatrix}}_{\mathcal{L}^2} \\
    &\quad = \norm{(\mu(t) - \beta\mean{\mu(t)}{\theta(t)},\theta(t) - \mean{\mu(t)}{\theta(t)})}_{\mathcal{L}^2} \\
    &\quad\leq C_P\norm{(\mu(t),\theta(t))}_{L}.
\end{align*}
Therefore, we arrive at
\begin{align*}
    -\ddt H(t) \geq \frac{\widetilde\varpi\min\{1,m^\ast\}}{C_p\widetilde C}\norm{(\mu(t),\theta(t))}_{L} \qquad\text{for~a.e.~}t\geq t^\ast.
\end{align*}
Integrating the previous inequality in time from $t^\ast$ to $\infty$, we derive from \eqref{E_infty} that
\begin{align*}
    \int_{t^\ast}^\infty \norm{(\mu(t),\theta(t))}_{L}\dt \leq \frac{C_p\widetilde C}{\widetilde\varpi\min\{1,m^\ast\}}H(t^\ast),
\end{align*}
from which we deduce that $t\mapsto\norm{(\mu(t),\theta(t))}_{L}\in L^1(t^\ast,\infty)$, entailing by comparison $(\delt\phi,\delt\psi)\in L^1(t^\ast,\infty;(\mathcal{H}^1_L)^\prime)$. 
Hence, there exists $(\phi_\infty,\psi_\infty)\in\omega^{K,L}(\phi_0,\psi_0)$ such that
\begin{align*}
    (\phi(t),\psi(t)) = (\phi(t^\ast),\psi(t^\ast)) + \int_{t^\ast}^t (\delt\phi(s),\delt\psi(s))\ds \longrightarrow (\phi_\infty,\psi_\infty) \quad\text{in~}(\mathcal{H}^1_L)^\prime\quad\text{as~}t\rightarrow\infty,
\end{align*}
and, by the uniqueness of the limit, we conclude that $\omega^{K,L}(\phi_0,\psi_0) = \{(\phi_\infty,\psi_\infty)\}$.
\end{proof}
\medskip

\appendix
\section{A bulk-surface chain rule}
\label{App-0}
\setcounter{equation}{0}

\begin{proposition}\label{App:Proposition:CR}
    Let $\Om\subset\R^d$, $d=2,3$, be an open bounded domain with $C^3$-boundary, let $I = (a,b) \subset\R$ be an open interval and $(\phi,\psi)\in H^1(I;\mathcal{L}^3)\cap L^\infty(I;\mathcal{W}^{2,4})$ with $\abs{\phi}\leq 1$ a.e. in $\Om$ and $\abs{\psi}\leq 1$ a.e. on $\Ga$. Furthermore, let $m_\Om, m_\Ga\in C^2([-1,1])$ satisfy \ref{Ass:Mobility}. Consider $(u,v)\in C(\overline{I};\mathcal{L}^2)\cap L^\infty(I;\mathcal{H}^1_L)\cap L^2(I;\mathcal{H}^3)$ such that $Lm_\Om(\phi)\deln u = \beta v - u$ a.e.~on $\Ga$ and $(\delt u, \delt v)\in L^2(I;(\mathcal{H}^1_K)^\prime)$. In addition, assume that
    \begin{align*}
        (\Div(m_\Om(\phi)\Grad u), \Divg(m_\Ga(\psi)\Gradg v) - \beta m_\Om(\phi)\deln u)\in L^2(I;\mathcal{H}^1_K).
    \end{align*}
    Then, the continuity property $(u,v)\in C(\overline{I};\mathcal{H}^1_L)$ holds, the mapping
    \begin{align*}
        I\ni t \mapsto \intO m_\Om(\phi(t))\abs{\Grad u(t)}^2\dx + \intG m_\Ga(\psi(t))\abs{\Gradg v(t)}^2\dG + \chi(L)\intG (\beta v(t) - u(t))^2\dG
    \end{align*}
    is absolutely continuous, and the chain rule formula 
    \begin{align}\label{App:ChainRule}
        \begin{split}
            &\ddt\frac12 \Big(\intO m_\Om(\phi)\abs{\Grad u}^2\dx + \intG m_\Ga(\psi)\abs{\Gradg v}^2\dG + \chi(L)\intG (\beta v - u)^2\dG\Big) \\
            &\quad = \bigang{(\delt u, \delt v)}{(-\Div(m_\Om(\phi)\Grad u), -\Divg(m_\Ga(\psi)\Gradg v) + \beta m_\Om(\phi)\deln u)}_{\mathcal{H}^1_K} \\
            &\qquad + \intO m_\Om^\prime(\phi)\delt\phi\abs{\Grad u}^2\dx + \intG m_\Ga^\prime(\psi)\delt\psi\abs{\Gradg v}^2\dG
        \end{split}
    \end{align}
    holds a.e. on $I$.
\end{proposition}

\begin{proof}
    Our proof is inspired by the approach in \cite[Proposition~A.1]{Colli2024}. In the present setting, additional care is required to handle the terms arising from the non-degenerate mobility, which necessitates further technical considerations. First, fix $u$ and $v$ as arbitrary representative of their respective equivalence class. Then, since $(u,v)\in C([a,b];\mathcal{L}^2)$, we can extend the functions $u$ and $v$ onto $[2a-b,a]$ by reflection for all $t < a$.

    Let $\rho\in C_c^\infty(\R)$ be non-negative with $\mathrm{supp}\,\rho\subset(0,1)$ and $\norm{\rho}_{L^1(\R)} = 1$. For any $k\in\N$, we set
    \begin{align*}
        \rho_k(s) \coloneqq k\rho(ks) \qquad\text{for all~}s\in\R.
    \end{align*}
    Then, for any Banach space $X$ and any function $f\in L^2(a-1,b;X)$, we define
    \begin{align*}
        f_k(t) \coloneqq (\rho_k\ast f)(t) = \int_{t-\tfrac{1}{k}}^t \rho_k(t-s)f(s)\ds \qquad\text{for all~}t\in[a,b] \quad\text{and}\quad k\in\N.
    \end{align*}
    By this construction, we have $f_k\in C^\infty([a,b];X)$ with $f_k\rightarrow f$ strongly in $L^2(a,b;X)$ as $k\rightarrow\infty$.

    Now, for any $k\in\N$, we use $X = H^3(\Om)$ to define $u_k$ and $X = H^3(\Ga)$ to define $v_k$ as described above. By this construction, it holds that $\delt u_k = (\delt u)_k$ and $\delt\Grad u_k = \Grad\delt u_k$ a.e. in $\Om\times(a,b)$ as well as $\delt v_k = (\delt v)_k$ and $\delt\Gradg v_k = \Gradg\delt v_k$ a.e. on $\Ga\times(a,b)$ for all $k\in\N$. Moreover, we have
    \begin{alignat}{2}
        u_k &\rightarrow u &&\qquad\text{strongly in~} L^2(a,b;H^3(\Om)), \label{App:Conv:u:H^3} \\
        v_k &\rightarrow v &&\qquad\text{strongly in~} L^2(a,b;H^3(\Ga)), \label{App:Conv:v:H^3} \\
        (u_k,v_k) &\rightarrow (u,v) &&\qquad\text{strongly in~} L^2(a,b;\mathcal{H}^1_L), \label{App:Conv:uv:H^1}\\
        (\delt u_k,\delt v_k) &\rightarrow (\delt u, \delt v) &&\qquad\text{strongly in~} L^2(a,b;(\mathcal{H}^1_K)^\prime) \label{App:Conv:delt:uv:H^1}
    \end{alignat}
    as $k\rightarrow\infty$. Furthermore, we readily see that
    \begin{align}\label{App:Est:uv:k:LinftyH1}
        \begin{split}
            \norm{u_k}_{L^\infty(a,b;H^1(\Om))} &\leq \norm{u}_{L^\infty(a,b;H^1(\Om))}, \\
            \norm{v_k}_{L^\infty(a,b;H^1(\Ga))} &\leq \norm{v}_{L^\infty(a,b;H^1(\Ga))}
        \end{split}
    \end{align}
    for all $k\in\N$.
    In the following, we will denote with $C$ generic positive constants independent of $k\in\N$, which may change their value from line to line. Now, for any $k\in\N$, we derive the identity
    \begin{align}\label{App:ChainRule:Approx}
            &\ddt\frac12 \Big(\intO m_\Om(\phi)\abs{\Grad u_k}^2\dx + \intG m_\Ga(\psi)\abs{\Gradg v_k}^2\dG + \chi(L)\intG (\beta v_k - u_k)^2\dG\Big) \nonumber \\
            &\quad = \bigang{(\delt u_k, \delt v_k)}{(-\Div(m_\Om(\phi)\Grad u_k), -\Divg(m_\Ga(\psi)\Gradg v_k) + \beta m_\Om(\phi)\deln u_k)}_{\mathcal{H}^1_K} \\
            &\qquad + \intO m_\Om^\prime(\phi)\delt\phi\abs{\Grad u_k}^2\dx + \intG m_\Ga^\prime(\psi)\delt\psi\abs{\Gradg v_k}^2\dG \nonumber 
    \end{align}
    a.e. on $[a,b]$ by differentiating under the integral sign and applying integration by parts. Similarly, for $j,k\in\N$, we calculate
    \begin{align}\label{App:ChainRule:Approx:Difference}
            &\ddt\frac12 \Big(\intO m_\Om(\phi)\abs{\Grad (u_j - u_k)}^2\dx + \intG m_\Ga(\psi)\abs{\Gradg (v_j - v_k)}^2\dG \nonumber \\
            &\qquad + \chi(L)\intG \big(\beta (v_j - v_k) - (u_j - u_k)\big)^2\dG\Big) \nonumber \\
            &\quad = \big\langle(\delt (u_j - u_k), \delt (v_j - v_k)), \nonumber \\
            &\qquad\qquad (-\Div(m_\Om(\phi)\Grad (u_j - u_k)), - \Divg(m_\Ga(\psi)\Gradg (v_j - v_k)) + \beta m_\Om(\phi)\deln (u_j - u_k))\big\rangle_{\mathcal{H}^1_K} \\
            &\qquad + \intO m_\Om^\prime(\phi)\delt\phi\abs{\Grad (u_j - u_k)}^2\dx + \intG m_\Ga^\prime(\psi)\delt\psi\abs{\Gradg (v_j - v_k)}^2\dG \nonumber \\
            &\quad\leq \norm{(\delt(u_j - u_k),\delt(v_j - v_k))}_{(\mathcal{H}^1_K)^\prime} \nonumber \\
            &\qquad\times \norm{(\Div(m_\Om(\phi)\Grad (u_j - u_k)), \Divg(m_\Ga(\psi)\Gradg (v_j - v_k)) - \beta m_\Om(\phi)\deln (u_j - u_k))}_{\mathcal{H}^1} \nonumber \\
            &\qquad + C\norm{(\delt\phi,\delt\psi)}_{\mathcal{L}^3}\norm{(\Grad(u_j - u_k),\Gradg(v_j - v_k))}_{\mathcal{L}^6}\norm{(\Grad(u_j - u_k),\Gradg(v_j - v_k))}_{\mathcal{L}^2}. \nonumber 
    \end{align}
    To estimate the right-hand side suitably, we use standard computations together with the Sobolev inequality to find that
    \begin{align*}
        &\norm{(\Div(m_\Om(\phi)\Grad (u_j - u_k)), \Divg(m_\Ga(\psi)\Gradg (v_j - v_k)) - \beta m_\Om(\phi)\deln (u_j - u_k))}_{\mathcal{H}^1} \\
        &\quad\leq \norm{(\Div(m_\Om(\phi)\Grad (u_j - u_k)), \Divg(m_\Ga(\psi)\Gradg (v_j - v_k)) - \beta m_\Om(\phi)\deln (u_j - u_k))}_{\mathcal{L}^2} \\
        &\qquad + \norm{(\Grad\big(\Div(m_\Om(\phi)\Grad (u_j - u_k))\big), \Gradg\big(\Divg(m_\Ga(\psi)\Gradg (v_j - v_k)) - \beta m_\Om(\phi)\deln (u_j - u_k))\big)}_{\mathcal{H}^1} \\
        &\quad \leq \norm{(m_\Om(\phi)\Lap(u_j - u_k),m_\Ga(\psi)\Lapg(v_j - v_k))}_{\mathcal{L}^2} \\
        &\qquad + \norm{(m_\Om^\prime(\phi)\Grad\phi\cdot\Grad(u_j - u_k),m_\Ga^\prime(\psi)\Gradg\psi\cdot\Gradg(v_j - v_k))}_{\mathcal{L}^2} \\
        &\qquad + \norm{(m_\Om(\phi)\Grad\Lap(u_j - u_k), m_\Ga(\psi)\Gradg\Lapg(v_j - v_j))}_{\mathcal{L}^2} \\
        &\qquad + \norm{(m_\Om^\prime(\phi)\Lap(u_j - u_k)\Grad\phi,m_\Ga^\prime(\psi)\Lapg(v_j - v_k)\Gradg\psi)}_{\mathcal{L}^2} \\
        &\qquad + \norm{(m_\Om^{\prime\prime}(\phi)(\Grad\phi\cdot\Grad(u_j - u_k))\Grad\phi, m_\Ga^{\prime\prime}(\psi)(\Gradg\psi\cdot\Gradg(v_j - v_k))\Gradg\psi)}_{\mathcal{L}^2} \\
        &\qquad + \norm{(m_\Om^\prime(\phi)D^2\phi\Grad(u_j - u_k), m_\Ga^\prime(\psi)D^2_\Ga\psi\Gradg(v_j - v_k))}_{\mathcal{L}^2} \\
        &\qquad + \norm{(m_\Om^\prime(\phi)D^2(u_j - u_k)\Grad\phi, m_\Ga^\prime(\psi)D^2_\Ga(v_j - v_j)\Gradg\psi)}_{\mathcal{L}^2} \\
        &\qquad + \norm{\beta m_\Om(\phi)\deln(u_j - u_k)}_{L^2(\Ga)} + \norm{\beta m_\Om^\prime(\phi)\Gradg\phi\deln(u_j - u_k)}_{L^2(\Ga)} \\
        &\qquad + \norm{\beta m_\Om(\phi)\Gradg\deln(u_j - u_k)}_{L^2(\Ga)} \\
        &\quad\leq C\norm{(u_j - u_k, v_j - v_k)}_{\mathcal{H}^2} + C\norm{(\Grad\phi,\Gradg\psi)}_{\mathcal{L}^\infty}\norm{(u_j - u_k, v_j - v_k)}_{\mathcal{H}^1} \\
        &\qquad + C\norm{(u_j - u_k, v_j - v_k)}_{\mathcal{H}^3} + C\norm{(\Grad\phi,\Gradg\psi)}_{\mathcal{L}^\infty}\norm{(u_j - u_k, v_j - v_k)}_{\mathcal{H}^2} \\
        &\qquad + C\norm{(\Grad\phi,\Gradg\psi)}_{\mathcal{L}^\infty}^2\norm{(u_j - u_k, v_j - v_k)}_{\mathcal{H}^1} + \norm{(\phi,\psi)}_{\mathcal{W}^{2,6}}\norm{(\Grad(u_j - u_k), \Gradg(v_j - v_k))}_{\mathcal{L}^3} \\
        &\qquad + C\norm{(\Grad\phi,\Gradg\psi)}_{\mathcal{L}^\infty}\norm{(u_j - u_k, v_j - v_k)}_{\mathcal{H}^2} + C\norm{u_j - u_k}_{H^2(\Om)} \\
        &\qquad + C\norm{\phi}_{W^{2,4}(\Om)}\norm{(u_j - u_k)}_{\mathcal{H}^2} + C\norm{u_j - u_k}_{H^3(\Om)} \\
        &\quad\leq C\norm{(u_j - u_k, v_j - v_k)}_{\mathcal{H}^3}.
    \end{align*}
    Here, we have additionally used the embedding $H^3(\Om)\emb H^2(\Ga)$ yielding
    \begin{align*}
        \norm{\deln(u_j - u_k)}_{H^1(\Ga)} \leq \norm{u_j - u_k}_{H^2(\Ga)} \leq C\norm{u_j - u_k}_{H^3(\Om)},
    \end{align*}
    whereas the embedding $W^{2,4}(\Om)\emb W^{1,\infty}(\Ga)$ shows that
    \begin{align*}
        \norm{\Gradg\phi}_{L^\infty(\Ga)} \leq \norm{\phi}_{W^{1,\infty}(\Ga)} \leq C\norm{\phi}_{W^{2,4}(\Om)}.
    \end{align*}
    Next, employing \eqref{App:Est:uv:k:LinftyH1}, we have
    \begin{align*}
        &\norm{(\Grad(u_j - u_k),\Gradg(v_j - v_k))}_{\mathcal{L}^6}\norm{(\Grad(u_j - u_k),\Gradg(v_j - v_k))}_{\mathcal{L}^2} \\
        &\quad\leq C\norm{(u,v)}_{L^\infty(I;\mathcal{H}^1)}\norm{(u_j - u_k, v_j - v_k)}_{\mathcal{H}^3}.
    \end{align*}
    Collecting our previous estimates, we conclude from \eqref{App:ChainRule:Approx:Difference} that
    \begin{align}\label{App:ChainRule:Approx:Diff:Est}
        \begin{split}
            &\ddt\frac12 \Big(\intO m_\Om(\phi)\abs{\Grad (u_j - u_k)}^2\dx + \intG m_\Ga(\psi)\abs{\Gradg (v_j - v_k)}^2\dG \\
            &\qquad + \chi(L)\intG \big(\beta (v_j - v_k) - (u_j - u_k)\big)^2\dG\Big) \\
            &\quad\leq C\big(\norm{(\delt(u_j - u_k),\delt(v_j - v_k))}_{(\mathcal{H}^1_K)^\prime} + \norm{(\delt\phi,\delt\psi)}_{\mathcal{L}^3}\big)\norm{(u_j - u_k, v_j - v_k)}_{\mathcal{H}^3}.
        \end{split}
    \end{align}
    Now, let $s,t\in [a,b]$ be arbitrary with $s \leq t$. We then integrate \eqref{App:ChainRule:Approx:Diff:Est} with respect to time over $[s,t]$, and obtain
    \begin{align}\label{App:ChainRule:Diff:Approx:Est:Int}
            &\intO m_\Om(\phi(t))\abs{\Grad (u_j - u_k)(t)}^2\dx + \intG m_\Ga(\psi(t))\abs{\Gradg (v_j - v_k)(t)}^2\dG \nonumber \\
            &\qquad + \chi(L)\intG \big(\beta (v_j - v_k)(t) - (u_j - u_k)(t)\big)^2\dG \nonumber \\
            &\quad \leq \intO m_\Om(\phi(s))\abs{\Grad (u_j - u_k)(s)}^2\dx + \intG m_\Ga(\psi(s))\abs{\Gradg (v_j - v_k)(s)}^2\dG \\
            &\qquad + \chi(L)\intG \big(\beta (v_j - v_k)(s) - (u_j - u_k)(s)\big)^2\dG \nonumber \\
            &\qquad + C\int_s^t \norm{(\delt(u_j - u_k), \delt(v_j - v_k))}_{(\mathcal{H}^1_K)^\prime}^2 + \norm{(u_j - u_k, v_j - v_k)}_{\mathcal{H}^3}^2 \dtau \nonumber \\
            &\qquad + C\int_s^t \norm{(\delt\phi,\delt\psi)}_{\mathcal{H}^1}\norm{(u_j - u_k, v_j - v_k)}_{\mathcal{H}^2} \dtau. \nonumber 
    \end{align}
    In light of the convergences \eqref{App:Conv:u:H^3}-\eqref{App:Conv:v:H^3}, we can fix $s\in[a,t]$ such that $(u_k(s),v_k(s)) \rightarrow (u(s),v(s))$ strongly in $\mathcal{H}^3$ along a non-relabeled subsequence $k\rightarrow\infty$. Recalling \eqref{App:Conv:u:H^3}-\eqref{App:Conv:delt:uv:H^1}, we thus deduce that the right-hand side of \eqref{App:ChainRule:Diff:Approx:Est:Int} tends to zero as $j,k\rightarrow\infty$. As $m_\Om$ and $m_\Ga$ are uniformly positive according to \eqref{Ass:Mobility:Bound}, we infer that $(\Grad u_k)_{k\in\N}$ is a Cauchy sequence in $C([a,b];L^2(\Om))$ and $(\Gradg v_k)_{k\in\N}$ is a Cauchy sequence in $C([a,b];L^2(\Ga))$. Consequently,
    \begin{alignat}{2}
        \Grad u_k &\rightarrow \Grad u &&\qquad\text{strongly in~} C([a,b];L^2(\Om)), \label{App:Conv:Grad:u}\\
        \Gradg v_k &\rightarrow \Gradg v &&\qquad\text{strongly in~} C([a,b];L^2(\Ga)) \label{App:Conv:Grad:v}
    \end{alignat}
    as $k\rightarrow\infty$. In view of the assumption $(u,v)\in C([a,b];\mathcal{L}^2)$, we readily deduce that $(u,v)\in C([a,b];\mathcal{H}^1)$.

    Let now $s,t\in [a,b]$ be arbitrary with $s \leq t$. We then integrate \eqref{App:ChainRule:Approx} in time from $s$ to $t$ and find 
    \begin{align}\label{App:ChainRule:Approx:Id}
        &\intO m_\Om(\phi(t))\abs{\Grad u_k(t)}^2\dx + \intG m_\Ga(\psi(t))\abs{\Gradg v_k(t)}^2\dG + \chi(L)\intG (\beta v_k(t) - u_k(t))^2\dG \nonumber \\
        &\quad = \intO m_\Om(\phi(s))\abs{\Grad u_k(s)}^2\dx + \intG m_\Ga(\psi(s))\abs{\Gradg v_k(s)}^2\dG + \chi(L)\intG (\beta v_k(s) - u_k(s))^2\dG \nonumber \\
        &\qquad + 2\int_s^t \bigang{(\delt u_k, \delt v_k)}{(-\Div(m_\Om(\phi)\Grad u_k), -\Divg(m_\Ga(\psi)\Gradg v_k) + \beta m_\Om(\phi)\deln u_k)}_{\mathcal{H}^1_K}\dtau \\
        &\qquad + 2\int_s^t \Big(\intO m_\Om^\prime(\phi)\delt\phi\abs{\Grad u_k}^2\dx + \intG m_\Ga^\prime(\psi)\delt\psi\abs{\Gradg v_k}^2\dG\Big)\dtau. \nonumber 
    \end{align}
    It is now clear from the convergences \eqref{App:Conv:u:H^3}-\eqref{App:Conv:delt:uv:H^1} and \eqref{App:Conv:Grad:u}-\eqref{App:Conv:Grad:v} to pass to the limit in all terms except the last one on the right-hand side. Here, we notice that
    \begin{align*}
        \begin{split}
            &\Big\vert\int_s^t\intO m_\Om^\prime(\phi)\delt\phi\abs{\Grad u_k}^2\dx\dtau - \int_s^t\intO m_\Om^\prime(\phi)\delt\phi\abs{\Grad u}^2\dx \dtau\Big\vert\\
            &\quad = \Big\vert\int_s^t\intO m_\Om^\prime(\phi)\delt\phi(\Grad u_k + \Grad u)\cdot(\Grad u_k - \Grad u)\dx\dtau\Big\vert \\
            &\quad\leq \norm{m_\Om^\prime}_{L^\infty(-1,1)}\norm{\delt\phi}_{L^2(a,b;L^3(\Om))} \norm{\Grad u_k + \Grad u}_{L^2(a,b;L^6(\Om))}\norm{\Grad u_k - \Grad u}_{C([a,b];L^2(\Om))}  \\
            &\quad \leq2\norm{m_\Om^\prime}_{L^\infty(-1,1)}\norm{\delt\phi}_{L^2(a,b;L^3(\Om))}\norm{u}_{L^2(a,b;H^3(\Om))}\norm{\Grad u_k - \Grad u}_{C([a,b];L^2(\Om))} \\
            &\quad \longrightarrow 0
        \end{split}
    \end{align*}
    as $k\rightarrow\infty$ in light of \eqref{App:Conv:Grad:u}. A similar computation shows that the corresponding term on the boundary also converges. Altogether, we are now able to pass to the limit $k\rightarrow\infty$ in \eqref{App:ChainRule:Approx:Id} and conclude that 
    \begin{align*}
        &\intO m_\Om(\phi(t))\abs{\Grad u(t)}^2\dx + \intG m_\Ga(\psi(t))\abs{\Gradg v(t)}^2\dG + \chi(L)\intG (\beta v(t) - u(t))^2\dG \\
        &\quad = \intO m_\Om(\phi(s))\abs{\Grad u(s)}^2\dx + \intG m_\Ga(\psi(s))\abs{\Gradg v(s)}^2\dG + \chi(L)\intG (\beta v(s) - u(s))^2\dG \\
        &\qquad + 2\int_s^t \bigang{(\delt u, \delt v)}{(-\Div(m_\Om(\phi)\Grad u), -\Divg(m_\Ga(\psi)\Gradg v) + \beta m_\Om(\phi)\deln u_k)}_{\mathcal{H}^1_K}\dtau \\
        &\qquad + 2\int_s^t \Big(\intO m_\Om^\prime(\phi)\delt\phi\abs{\Grad u}^2\dx + \intG m_\Ga^\prime(\psi)\delt\psi\abs{\Gradg v}^2\dG\Big)\dtau.
    \end{align*}
    By our previous considerations, we readily see that the integrands on the right-hand side belong to $L^1(a,b)$, and thus, that the mapping
    \begin{align*}
        [a,b]\ni t \mapsto \intO m_\Om(\phi(t))\abs{\Grad u(t)}^2\dx + \intG m_\Ga(\psi(t))\abs{\Gradg v(t)}^2\dG + \chi(L)\intG (\beta v(t) - u(t))^2\dG
    \end{align*}
    is absolutely continuous. It is therefore differentiable almost everywhere on $[a,b]$ and its derivative satisfies the formula \eqref{App:ChainRule}. This finishes the proof.
\end{proof}

\section*{Acknowledgement}
\noindent
The author wishes to thank Andrea Giorgini for some fruitful discussions, as well as Patrik Knopf for his valuable feedback on an earlier version of this work. This work was completed while the author was visiting the Dipartimento di Matematica of the Politecnico di Milano, whose hospitality is greatly appreciated.
This work was supported by the Deutsche Forschungsgemeinschaft (DFG,
German Research Foundation): on the one hand by the DFG-project 524694286, and on the other
hand by the RTG 2339 “Interfaces, Complex Structures, and Singular Limits”. Their support is
gratefully acknowledged.

\section*{Conflict of Interests and Data Availability Statement}

There is no conflict of interest.

There is no associated data with the manuscript.

\bibliographystyle{abbrv}
\bibliography{S.bib}
\end{document}